\documentclass[12pt]{article}

\usepackage[T1]{fontenc}
\usepackage{amsmath}
\usepackage{amssymb}
\usepackage{amsthm}
\usepackage{amsfonts}
\usepackage{graphicx}
\usepackage{subfigure}
\usepackage[font=small]{caption}
\usepackage{enumitem}
\usepackage{colonequals}
\usepackage{bm}
\usepackage{color}
\usepackage{mathpazo_modified}
\usepackage{verbatim}

\definecolor{dgreen}{rgb}{0.0,0.5,0.0}
\definecolor{dorange}{rgb}{1.0,0.5,0.0}
\definecolor{dpurple}{rgb}{0.490196, 0.14902, 0.803922}

\makeatletter
\def\@rst #1 #2other{#1}
\newcommand\MR[1]{\relax\ifhmode\unskip\spacefactor3000 \space\fi
  \MRhref{\expandafter\@rst #1 other}{#1}}
\newcommand{\MRhref}[2]{\href{http://www.ams.org/mathscinet-getitem?mr=#1}{MR#2\
}}
\makeatother

\newif\ifhyper\IfFileExists{hyperref.sty}{\hypertrue}{\hyperfalse}
\ifhyper\usepackage[pdftitle={Fractional Gaussian Fields},
            pdfauthor={Asad Lodhia, Scott Sheffield, Xin Sun, and Samuel S.\ Watson},
            bookmarks=true,bookmarksopen=true,bookmarksopenlevel=3,unicode]{hyperref}\fi

\newif\ifdraft
\drafttrue
\def\note#1/{\ifdraft {\bf [#1]}\fi}
\long\def\comment#1{}
\numberwithin{equation}{section}
\numberwithin{figure}{section}

\newtheorem{theorem}{Theorem}
\numberwithin{theorem}{section}

\newtheorem{lemma}[theorem]{Lemma}
\newtheorem{proposition}[theorem]{Proposition}

\theoremstyle{remark}\newtheorem{definition}[theorem]{Definition}
\theoremstyle{remark}\newtheorem{remark}[theorem]{Remark}

\newcommand{\D}{\mathbb{D}} \newcommand{\E}{\mathbb{E}}
 \newcommand{\N}{\mathbb{N}}
\newcommand{\Z}{\mathbb{Z}} 
 \newcommand{\R}{\mathbb{R}}
 \newcommand{\s}{\mathcal{S}}
 
\newcommand{\La}{\Delta}
\newcommand{\Lap}{(-\Delta)}
\newcommand{\Hs}{\frac{s}{2}}
\newcommand{\dH}{\dot{H}}
\newcommand{\Hu}{H}
\newcommand{\wh}{\widehat}
\newcommand{\wt}{\widetilde}

\newcommand{\one}{{\mathbf{1}}}
\newcommand{\ang}[1]{\langle #1 \rangle}
\newcommand{\eps}{\epsilon}
\newcommand{\resop}{\mathfrak{R}}

\newcommand{\nablaa}{\nabla_{\!\!\alpha}}

\DeclareMathOperator{\dist}{dist}

\DeclareMathOperator{\Har}{Har}
\DeclareMathOperator{\SLE}{SLE}
\DeclareMathOperator{\FGF}{FGF}
\DeclareMathOperator{\DFGF}{DFGF}
\DeclareMathOperator{\EFGF}{EFGF}

\DeclareMathOperator{\Cov}{Cov}
\DeclareMathOperator{\Var}{Var}

\DeclareMathOperator{\Res}{Res}

\setenumerate{labelsep=3pt,itemsep=0pt,parsep=0pt,topsep=0pt}
\setitemize{labelsep=3pt,itemsep=0pt,parsep=0pt,topsep=0pt}

\begin{document}
\title{Fractional Gaussian fields: a survey}
\author{\normalsize\textsc{{Asad Lodhia\footnote{Partially supported by NSF
        grant DMS 1209044.}}, \, {Scott
      Sheffield\footnotemark[1]},\, {Xin Sun\footnotemark[1]},\, {Samuel S.\
      Watson\footnote{Supported by NSF GRFP award number 1122374.}}}}
\date{}
\maketitle

\thispagestyle{empty}

\begin{abstract}
  We discuss a family of random fields indexed by a parameter $s\in \R$
  which we call the {\em fractional Gaussian fields}, given by
  \[
  \FGF_s(\R^d)=(-\Delta)^{-s/2} W,
  \]
  where $W$ is a white noise on $\mathbb R^d$ and $(-\Delta)^{-s/2}$ is the
  fractional Laplacian.  These fields can also be parameterized by their
  Hurst parameter $H = s-d/2$.  In one dimension, examples of
  $\FGF_s$ processes include Brownian motion ($s = 1$) and fractional
  Brownian motion ($1/2 < s < 3/2$). Examples in arbitrary dimension
  include white noise ($s = 0$), the Gaussian free field ($s = 1$), the
  bi-Laplacian Gaussian field ($s = 2$), the log-correlated Gaussian field
  ($s = d/2$), L\'evy's Brownian motion ($s = d/2 + 1/2$), and
  multidimensional fractional Brownian motion ($d/2 < s < d/2 + 1$).  These
  fields have applications to statistical physics, early-universe
  cosmology, finance, quantum field theory, image processing, and other
  disciplines.

  We present an overview of fractional Gaussian fields including covariance
  formulas, Gibbs properties, spherical coordinate decompositions,
  restrictions to linear subspaces, local set theorems, and other basic
  results. We also define a discrete fractional Gaussian field and explain how the
  $\FGF_s$ with $s \in (0,1)$ can be understood as a long range Gaussian
  free field in which the potential theory of Brownian motion is replaced
  by that of an isotropic $2s$-stable L\'evy process.
\end{abstract}

\newpage
\enlargethispage{1cm}
\renewcommand{\contentsname}{}
\parskip = 0.0 in
\begin{footnotesize}
\tableofcontents
\end{footnotesize}
\parskip = 0.1 in
\newpage

\section{Introduction}\label{sec:Introduction}

\makeatletter{}\begin{figure}[hpt]
  \center
  \subfigure[White Noise, $s=0$]{\includegraphics[width=0.45\textwidth]{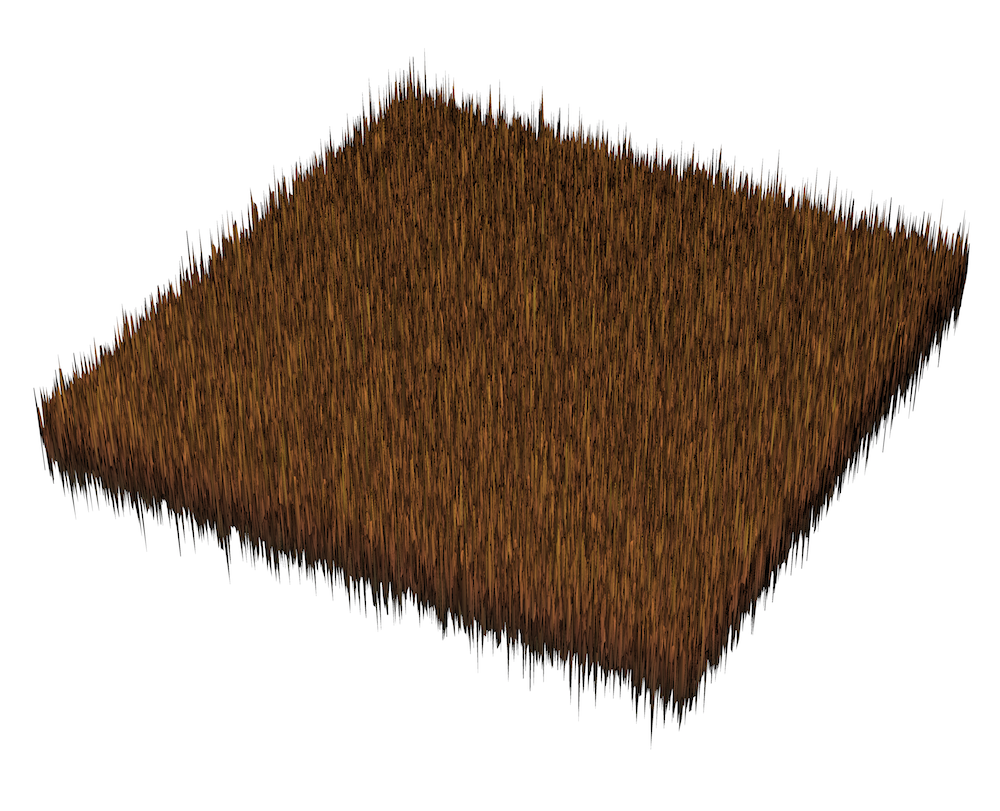}}
  \subfigure[GFF,
  $s=1$]{\includegraphics[width=0.45\textwidth]{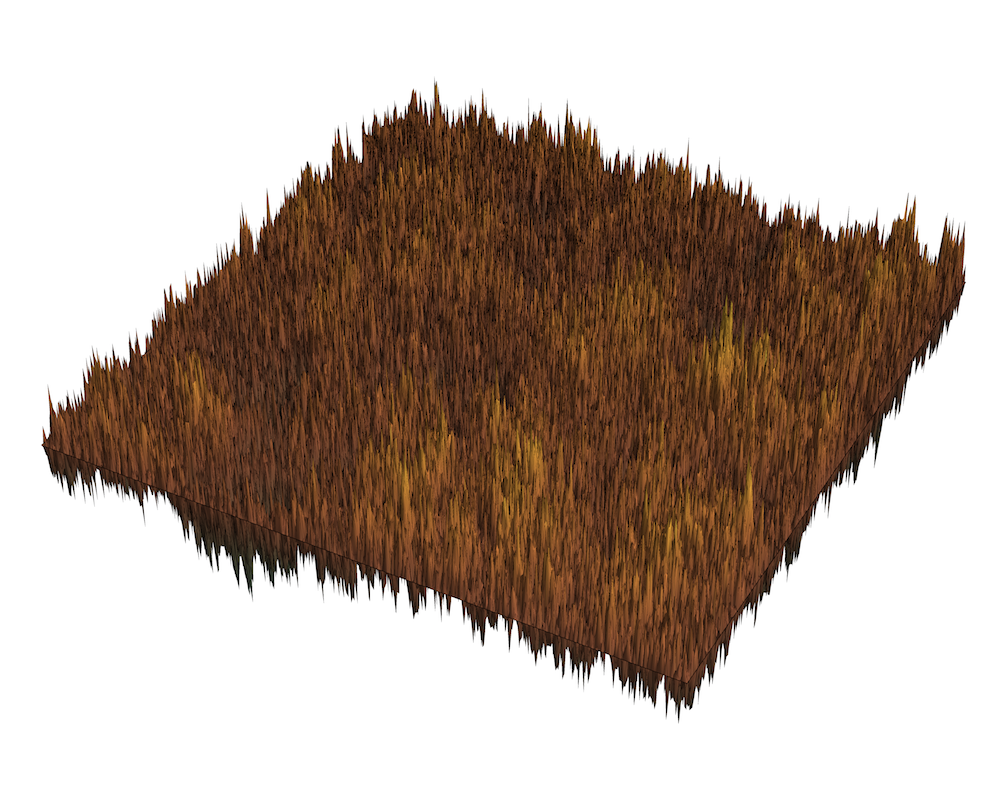}}  \\
  \subfigure[Bi-Laplacian,
  $s=2$]{\includegraphics[width=0.45\textwidth]{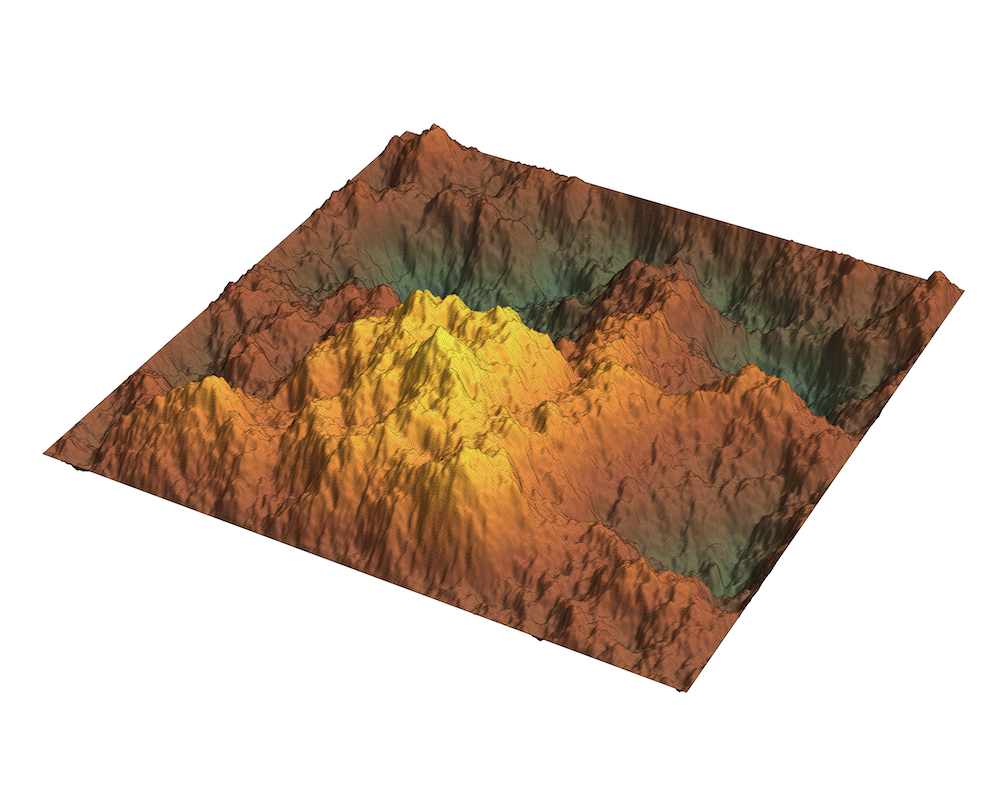}}
  \subfigure[$\FGF_s$ with $s=3$]{\includegraphics[width=0.45\textwidth]{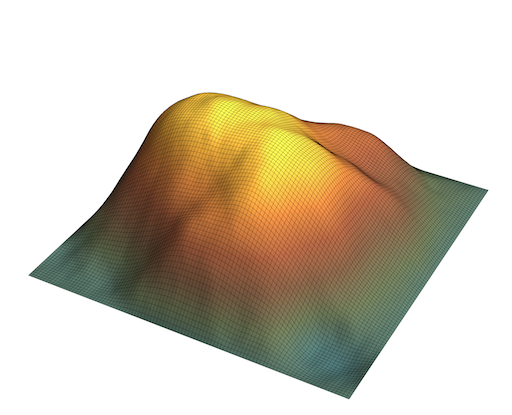}}
  \caption{\label{fig:pictures} Surface plots of discrete fractional
    Gaussian fields as defined on a bounded domain $D = [0,1]^2 \subset \R^2$ with zero boundary conditions, where $s=0$, 1, 2, and 3
    respectively. These discrete random functions are defined on a
    $500\times500$ grid and linearly interpolated. The corresponding
    continuum limit, $\FGF_s([0,1]^2)$, is not a function when $s=0$ or
    $s=1$, is $\alpha$-H\"older continuous for all $\alpha<1$ when $s=2$,
    and has $\alpha$-H\"older continuous first-order derivatives for all
    $\alpha<1$ when $s=3$.}
\end{figure}

\begin{figure}[hpt]
  \center
 \includegraphics[width=3in]{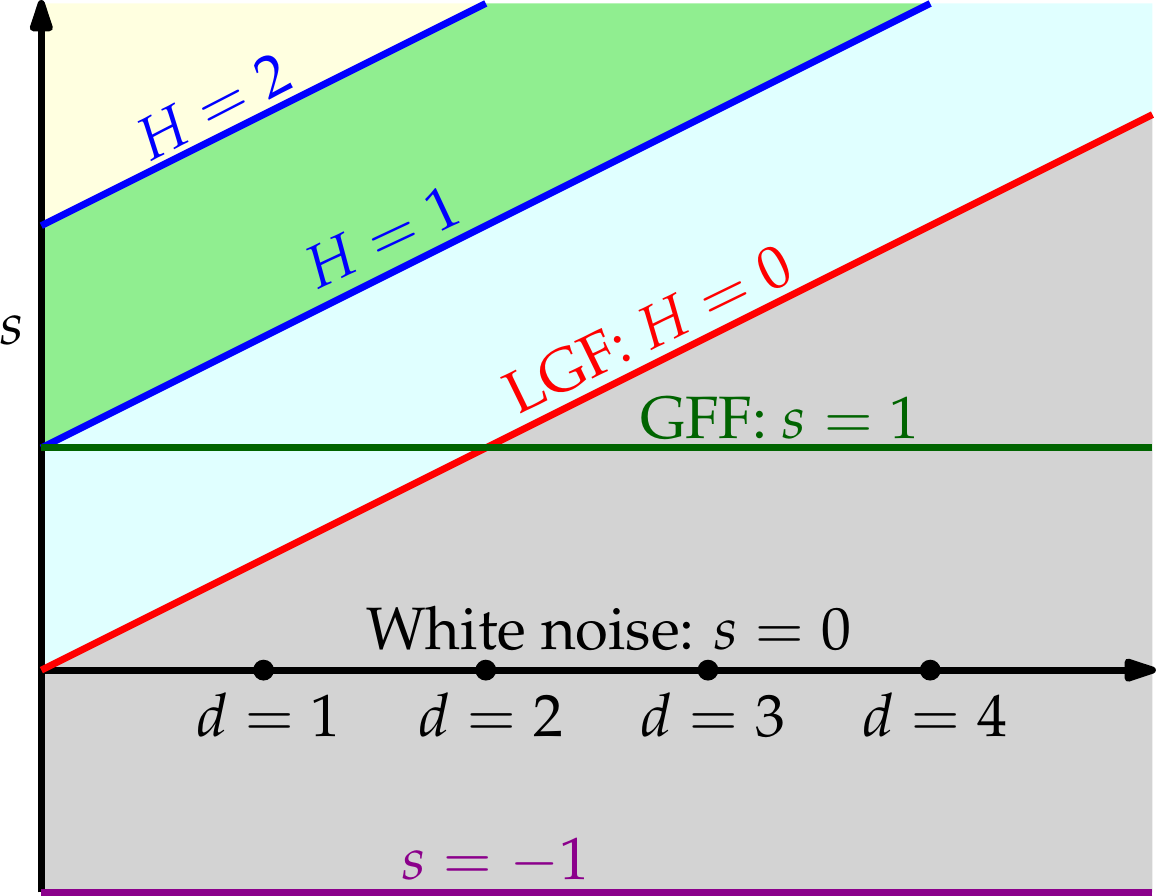}
 \caption{\label{fig:indexgraph} When $\Hu < 0$ (grey shaded region) the
   FGF is defined as a random tempered distribution, not a function.
   When $\Hu\in(0,1)$, the FGF is defined as a random
   continuous function modulo a global additive constant.  Generally, for
   integers $k>0$ and $\Hu\in(k,k+1)$, the FGF is a translation invariant random
   $k$-times-differentiable function defined modulo polynomials of
   degree $k$.  For integer $\Hu=k \geq 0$,
   the FGF is a random $(k-1)$-times-differentiable function (or distribution if $k=0$)
defined modulo polynomials of degree $k$.
             }
\end{figure}

The $d$-dimensional \textbf{fractional Gaussian field} $h$ on $\R^d$ with
index $s\in \R$ (abbreviated as $\FGF_s(\R^d)$) is given by
\begin{equation} \label{eqn:fgflaplaciandef}
  h := (-\Delta)^{-s/2} W,
\end{equation}
where $W$ is a real white noise on $\R^d$ and $ (-\Delta)^{-s/2}$ is the
fractional Laplacian on $\R^d$. In Sections \ref{sec:preliminaries} and
\ref{sec:whole-space}, we will review classical and recent literature on
the fractional Laplacian (see, e.g., \cite{landkof1972foundations,
  silvestre2007regularity, caffarelli2008regularity, chang2011fractional})
and show how to assign rigorous meaning to \eqref{eqn:fgflaplaciandef}. 

Our goal is to provide a mathematically rigorous, unified, and accessible
account of the $\FGF_s(\R^d)$ processes, treating the full range of
values $s\in \R$ and $d \in \N$. This paper is fundamentally a survey,
but we also present several basic facts that we have not found
articulated elsewhere in the literature.  Many of these are
generalizations of classical results that had previously only been
formulated for specific $d$ and $s$ values.

We hope that this survey will increase the circulation of basic information
about fractional Gaussian fields in the mathematical community.  For
example, the vocabulary and content of the following statements should
arguably be well known to probabilists, but the authors were unaware of
much of it until recently: 
\begin{itemize}
\item In dimension $3$, the Gaussian field with logarithmic correlations
  has been used as an approximate model for the gravitational potential of
  the early universe; its Laplacian is a $\FGF_{-1/2}(\R^3)$ and has been
  used to model the perturbation from uniformity of the mass/energy density
  of the early universe\footnote{An overview of this story appears in the
    reference text \cite{2003moco.book.....D} and a few additional notes
    and references appear in \cite{duplantier2013log}.}.
\item In dimension $4$, the so-called bi-Laplacian field has logarithmic
  correlations, and its Laplacian is white noise.
\item In any dimension, L\'evy Brownian motion can be defined as a random
  continuous function whose restriction to any line has the law of a
  Brownian motion (modulo additive constant). In dimension $5$, the
  Laplacian of L\'evy Brownian motion is the Gaussian free field.
\end{itemize}
We also hope that this text will be a useful reference for experts in the
study of Gaussian fields; to this end, we provide a robust account of the
regularity of FGF fields, the long and short range correlation formulae,
conditional expectations given field values outside of fixed domains,
the Fourier transforms and spherical coordinate decompositions of the FGF,
and various bounded-domain definitions of the FGF.

The family of fractional Gaussian fields includes several well-known Gaussian
fields such as Brownian motion ($d=1$ and $s=1$), white noise ($s=0$), the
Gaussian free field ($s=1$), and the log-correlated Gaussian field ($s =
\frac{d}{2}$).

Given $s\in \R$ and $d \geq 1$, the \textbf{Hurst parameter} $H$ is defined
by
\begin{equation}
  \label{eqn:hdef}
  \Hu \colonequals s-\frac{d}{2}.
\end{equation}
The Hurst parameter describes a scaling relation satisfied by $h\sim
\FGF_s(\R^d)$: for $a>0$, the field $x\mapsto h(a x)$ has the same law as
$x\mapsto a^\Hu h(x)$.\footnote{When $s$ and $d$ are such that $h$ is a
  random tempered distribution, but not a random function, we interpret
  $x\mapsto h(ax)$ as a distribution via $(x\mapsto h(ax),\phi) =
  a^{-d}(h,x\mapsto \phi(x/a))$.} Fields satisfying such a relation
are said to be \textbf{self-similar}, and they arise naturally in the
study of statistical physics models \cite{newman1980self}. The FGFs
belong to a more general class of translation-invariant self-similar
Gaussian random fields which were investigated and classified in
\cite{dobrushin1979gaussian}.  When $d=1$ and $\Hu \in (0,1)$, the
$\FGF_s(\R^d)$ process is commonly known as {\bf fractional Brownian motion
  with Hurst parameter $H$}, and is the subject of an extensive literature
(see the survey \cite{cohenfracfield}).  Brownian motion itself corresponds
to $H=1/2$ and $s=1$.

The law of a Brownian motion or fractional Brownian motion $B_t$, indexed
by $t \in \R$ and defined so that $B_0 = 0$, is not translation invariant.
However, the law of Brownian motion is translation invariant if we consider
Brownian motion as a random process defined only modulo a global additive
constant.  In other words, Brownian motion has stationary increments.
Similarly, the indefinite integral of a Brownian motion can be interpreted,
in a translation invariant way, as a random function defined modulo the
space of linear functions.  We generally interpret all of the $\FGF_s$
processes as translation invariant random distributions, but in some cases
they are defined modulo a space of polynomials. More precisely, when
$\Hu < 0$, $\FGF_s(\R^d)$ is a translation invariant random tempered
distribution (that is, a generalized function) on $\R^d$. When $\Hu > 0$,
$\FGF_s(\R^d)$ is a translation invariant random element of the space
$C^{\lceil \Hu \rceil - 1}(\R^d)$ modulo the space of polynomials on $\R^d$
of degree no greater than $\lfloor \Hu \rfloor$. This means that $h$ is
defined as a linear functional on the subspace of test functions $\phi$
satisfying $\int_{\R^d} \phi(x) L(x)dx = 0$ for all polynomials $L$ of
degree $\lfloor \Hu \rfloor$. Alternatively, at the cost of breaking
translation invariance, we may define $\FGF_s(\R^d)$ as a random element of
$C^{\lceil \Hu \rceil -1}(\R^d)$ by fixing the derivatives of $h$ at 0 up
to order $\lceil \Hu \rceil - 1$.  
The FGF covariance structure is described by the Hurst parameter $H$.  When
$\Hu$ is a positive non-integer, we have
$$\Cov[(h, \phi_1), (h, \phi_2)] = C(s,d)\int_{\R^d} \int_{\R^d}
|x-y|^{2\Hu} \phi_1(x) \phi_2(y) dxdy,$$ for some constant $C(s,d)$.  A
variant of this statement applies for negative and integer values of $\Hu$
(see Theorem \ref{KernelComp}).

Note that $\Hu$ is an affine function of $s$ and can be used instead of $s$
to parameterize the family of FGFs.  We use the parameter $s$ in part to
highlight the connection to the fractional Laplacian and white noise. With
our convention, white noise is $\FGF_0(\R^d)$ and the Gaussian free field
is $\FGF_1(\R^d)$.  However, in many of our formulas and theorems $\Hu$
will be the more natural parameter to use; thus, we fix the relationship
\eqref{eqn:hdef} and reference both $\Hu$ and $s$ throughout the
paper. We note that the fields $\{\FGF_s(\R^d)\,:\, s \in \R\}$ may be
  coupled with the same white noise so that \eqref{eqn:fgflaplaciandef}
  holds for all $s\in \R$ (Proposition~\ref{prop:FGF_coupling}).

\subsection{Examples}
The simplest example of a fractional Gaussian field is $\FGF_0(\R^d)$,
which is white noise. We denote by $\s(\R^d)$ the space of Schwartz
functions on $\R^d$, and we let $\s'(\R^d)$ be its dual, the space of
tempered distributions (see Section~\ref{sec:preliminaries} for
details). If $h\in \s'(\R^d)$ and $\phi\in \s(\R^d)$, we use the notation
$(h,\phi)$ for $h$ evaluated at $\phi$. White noise (surveyed in
\cite{kuo1996white}) is a random element of $\s'(\R^d)$ with the property
that for $\phi_1,\phi_2\in \s(\R^d)$, the random variables $(h,\phi_1)$ and
$(h,\phi_2)$ are centered Gaussians with covariance
\[
\Cov[(h,\phi_1),(h,\phi_2)] = \int_{\R^d}\phi_1(x)\phi_2(x)\,dx.
\]

Taking $d=1$ and $s=1$, we see that $(-\Delta)^{-s/2}$ is the
antiderivative operator. It follows that $\FGF_{1}(\R)$ is the
antiderivative of one dimensional white noise, which is a Brownian motion
interpreted as a real-valued function modulo constant. If we fix the
constant by setting the value at 0, we get ordinary Brownian motion.

If $s=1$ and $d\in \N$, then $\FGF_{1}(\R^d)$ is a $d$-dimensional
generalization of Brownian motion called the Gaussian free field
(GFF). As surveyed in \cite{sheffield2007gaussian}, the GFF is a random
tempered distribution on $\R^d$ (defined modulo additive constant if $d=2$)
with covariance given by
\[
\Cov[(h,\phi_1),(h,\phi_2)] =
\int_{\R^d}\int_{\R^d}\Phi(x-y)\phi_1(x)\phi_2(y)\,dx\,dy,
\]
where $\Phi$ is the fundamental solution of the Laplace equation in
$\R^d$. The two-dimensional GFF (which is the same as $\FGF_1(\R^2)$) has
been studied in a wide range of contexts in recent years. It can be
obtained as a scaling limit of random discrete models, such as domino
tilings \cite{kenyon2001dominos}, as well as continuum models, such as
those arising in random matrix theory \cite{rider2006noise}.  It is central
to conformal field theory and Liouville quantum gravity
\cite{sheffield2010conformal, duplantier2011liouville} and has many
connections to the Schramm Loewner evolution \cite{dubedat2009sle,
  schramm2010contour, millerimaginary1, millerimaginary2, millerimaginary3,
  millerimaginary4}. The 2D GFF is also known in the geostatistics
  literature as the \textit{de Wijs process} or the \textit{logarithmic
    variogram model}, where it was introduced in the early 1950s to
  describe ore deposits \cite{de1951statistics, de1953statistics,
    mondalapplying, chiles2009geostatistics}. More recently, variations in
  crop yields have been modeled using the GFF
  \cite{mccullagh2002statistical,mccullagh2006evidence}.

For all $d\in \N$, the $d$-dimensional GFF exhibits a certain Markov
property: For each fixed domain $D \subset \R^d$, if we are given the
restriction a GFF $h$ to $\R^d \setminus D$, then the conditional law of
$h$ restricted to $D$ is given by a conditionally deterministic function
(the harmonic extension of the field from $\partial D$ to
$D$)\footnote{Since the GFF is not defined pointwise, some care is needed
  to define the harmonic extension of the values of the GFF on $\R^d
  \setminus D$. Nevertheless, this can be made rigorous
  \cite{schramm2010contour}.}  plus an independent zero-boundary GFF
defined on $D$.

In Section \ref{sec:projection} we will establish an analogous property
that applies when $h$ is an $\FGF_s(\R^d)$ with $s \geq 0$.  Namely, if we
are given the restriction of $h$ to $\R^d \setminus D$, then the
conditional law of $h$ restricted to $D$ is given by a conditionally
deterministic function (the so-called {\bf $s$-harmonic extension} of the
field from $\R^d \setminus D$ to $D$) plus a random function (the so-called
{\bf zero-boundary-condition $\FGF_s$ on $D$}).  If $s \in \N$, then the
conditionally deterministic function depends on the restriction to
$\partial D$ of $h$ and its derivatives up to a certain order.  This
follows from the fact that $(-\Delta)^s$ is a local operator when $s \in
\N$.

As previously mentioned, another generalization of Brownian motion is the
fractional Brownian motion (FBM).  Fractional Brownian motion appears to
have been first introduced by Kolmogorov in 1940
\cite{kolmogorov1940wienersche}, and the term ``fractional Brownian
motion'' was introduced by Mandelbrot and Van Ness in 1968
\cite{mandelbrot1968fractional}.  As motivation, Mandelbrot and Van Ness
discuss various empirical studies of real world processes (the price of
wheat, water flowing through the Nile, etc.) that had been made by Hurst,
who found different scaling exponents in different settings\footnote{FBM is
  not the only model exhibiting the scaling behavior observed by Hurst. See
  \cite{bhattacharya1983hurst} for a model which uses drift rather than
  long-range dependence.}.

The definition of fractional Brownian motion can be extended to describe a
random function modulo additive constant on $\R^d$ when $d > 1$.  Given
$\Hu\in(0,1)$ we define the FBM (also called the fractional Brownian field)
on $\R^d$ as a mean-zero Gaussian process $(B^{\Hu}_t)_{t\in \R}$ with
covariance
\[
\Cov(B^\Hu_tB^\Hu_s) = \frac{1}{2}(|t|^{2\Hu}+|s|^{2\Hu}-|t-s|^{2\Hu}), 
\]
where $\Hu$ is the Hurst parameter of the field. We will prove in Section
\ref{sec:FBM} that the multidimensional fractional Brownian motion defined
this way is equivalent to $\FGF_s(\R^d)$, where $H=s-\frac{d}{2}\in(0,1)$. 

In the case $\Hu = 1/2$, this multidimensional process was introduced by
L\'evy in 1940 and is known as L\'evy Brownian motion
\cite{levy1940mouvement}.  General processes including multidimensional
fractional Brownian motion are discussed in Yaglom in 1957 and by Gangolli
in 1967 \cite{yaglom1957some, MR0215331}.  (Gangolli gives general analytic
arguments for positive definiteness of covariance kernels that apply in
this case.)  Fractional Brownian motion is studied in more
detail in works of Mandelbrot, as referenced in
\cite{mandelbrot1975geometry}.  More detailed and modern discussions of
fractional Brownian motion (including topics such as excursion set theory,
Hausdorff dimension, H\"older regularity, etc.) can be found in
\cite{adler2007random, adler2010geometry}.

The log-correlated Gaussian field (LGF) is a random element $h$ of the
space of tempered distributions modulo constants and has covariance given
by
\[
\Cov[(h,\phi_1),(h,\phi_2)] = -\int_{\R^d}\int_{\R^d}\log|x-y|\phi_1(x)\phi_2(y)\,dx\,dy,
\]
In two dimensions, the LGF coincides with the GFF (up to a constant
factor).
We will see in Section \ref{sec:whole-space} that the $d$-dimensional LGF
is a multiple of $\FGF_{d/2}(\R^d)$.  In recent years the log-correlated
Gaussian field has enjoyed renewed interest because of its relationship to
Gaussian multiplicative chaos. For a survey article of Gaussian
multiplicative chaos see \cite{rhodes2013gaussian}. Furthermore, the LGF in
$\R^3$ plays an important role in early universe cosmology, where it
approximately describes the gravitational potential function of the
universe at a fixed time shortly after the big bang; see
\cite{duplantier2013log} for more discussion and references.

Another noteworthy subclass of the fractional Gaussian fields is
$\FGF_2(\R^d)$, which is known as the bi-Laplacian Gaussian field. The
discrete counterpart of the bi-Laplacian Gaussian field is called the
membrane model in physics literature; for a mathematical point of view see
\cite{sakagawa2003entropic}, \cite{kurt2007entropic},
\cite{kurt2009maximum}, and \cite{sakagawa2012free}. In dimension at least
five, there is a natural discrete field associated with the uniform
spanning forest on $\Z^d$ whose scaling limit is $\FGF_2(\R^d)$
\cite{sun2013uniform}.

\subsection{Fractional Gaussian fields in one dimension}

The $\FGF_s(\R^d)$ processes are easiest to classify and explain when
$d=1$.  We first consider $\Hu = s-\frac{d}{2} \in (0,1)$ (so that $s \in
(1/2, 3/2)$), in which case the $\FGF_s(\R)$ is a Gaussian random function
$h: \R \to \R$ which we interpret as being defined modulo an additive constant.
This means that while the quantity $h(t)$ is not a well-defined random
variable for $t \in \R$, the quantity $h(t_1) - h(t_2)$ is a well-defined
random variable for $t_1,t_2 \in \R$. When $H\in (0,1)$, the $\FGF_s(\R)$
is the stationary-increment form of the {\bf fractional Brownian motion
  with Hurst parameter $\Hu$}. The law of the fractional Brownian motion is
determined by the variance formula $$\Var \bigl(h(t_1) - h(t_2)\bigr) =
|t_1 - t_2|^\Hu.$$
When $\Hu = 0$, so that $s=1/2$, the $\FGF_s(\R)$ is the log-correlated
Gaussian field (LGF), which is defined as a random tempered distribution modulo
additive constant.

When $d=1$ the weak derivative of an $\FGF_s(\R)$ is an $\FGF_{s-1}(\R)$.
Thus all $\FGF_s(\R)$ processes may be obtained by either integrating or
differentiating fractional Brownian motion (with $s \in (1/2, 3/2)$) or the
LGF ($s=1/2$) an integral number of times.  From this, it is clear that if
an $\FGF_s(\R)$, for $s \in (1/2,3/2]$, is defined modulo additive constant
in a translation invariant way, then the distributional derivatives
$\FGF_{s-1}(\R)$, $\FGF_{s-2}(\R)$, etc.\ are defined without an additive
constant.  Thus the $\FGF_s(\R)$ is defined as a random tempered
distribution without an additive constant when $s \leq 1/2$.  Similarly, if
the $\FGF_s(\R)$, for $s \in (1/2,3/2]$ is defined modulo additive constant
(in a translation invariant way), then the indefinite integrals
$\FGF_{s+1}(\R)$, $\FGF_{s+2}(\R)$, etc. are respectively defined modulo
linear polynomials, quadratic polynomials, etc.

The following proposition, rephrased and proved as Theorem~\ref{thm:
  restriction} in Section~\ref{sec:restriction}, is one reason that the
one-dimensional case is significant.

\begin{proposition}
  If $\Hu \geq 0$, then the restriction of the $d$-dimensional FGF with
  Hurst parameter $\Hu$ (i.e., with $s = \Hu + \frac{d}{2}$) to any fixed
  $k$-dimensional subspace (with $1 \leq k < d$) is a $k$-dimensional FGF
  with Hurst parameter $\Hu$ (up to multiplicative constant).
\end{proposition}

\subsection{Interpretation as a long range GFF}\label{sec:long-range-FGF}

The Gaussian free field $\FGF_1(\R^d)$ can be approximated by the discrete
Gaussian free field, which only has nearest neighbor interactions. This
discrete Markov property gives rise to the domain Markov property of the
Gaussian free field in the limit \cite{sheffield2007gaussian}. In Section
\ref{sec:dfgf}, we construct a discrete version of $\FGF_s$ for $s\in
(0,1)$ by introducing a discrete fractional gradient to play the role of
the discrete gradient in the definition of the discrete GFF. The fractional
gradient involves long range interactions, which may be viewed as the
reason that the Markov property fails for $\FGF_s$ when $s$ is not an
integer.

The comparison between the short range $\FGF_s(\R^d)$ (when $s\in \Z$) and
the long range $\FGF_s(\R^d)$ (when $s\notin \Z$) may also be seen from the
point of view of the corresponding potential theories. As an illustration,
consider GFF and $\FGF_s$ for $0<s<1$. The covariance kernel for the
Gaussian free field is given by the solution of the ordinary Laplace
equation $-\Delta f = \phi$. As we will see, the counterpart for $\FGF_s$
with $0<s<1$ is the fractional Laplacian equation $(-\Delta)^{s} f=\phi$.
The Laplacian is a local differential operator, while $(-\Delta)^{s}$ for
$s\in (0,1)$ is a non-local pseudo-differential operator and $(-\Delta)^s
f(x)$ depends on the values of $f(x)$ for all $x\in \R^d$. Another way to
see the distinction between the $s=1$ and $s\in(0,1)$ cases is to recall
that the Green's function for the Dirichlet Laplacian is given by the
density of the occupation measure of a Brownian motion (see
\cite{morters2010brownian}, for example), which is continuous. The
corresponding process when $s\in (0,1)$ is an isotropic $2s$-stable
L\'evy motion, which is a jump process.  

\section*{Acknowledgements}

The authors would like to thank David Jerison, Jason Miller, and Charles
Smart for helpful discussions.  Scott Sheffield thanks the astronomer John
M. Kovac for conversations about early universe cosmology and for the
corresponding references. Sheffield would also like to thank Bertrand
Duplantier, R\'emi Rhodes, and Vincent Vargas, his co-authors on a survey
of the log-correlated Gaussian field \cite{duplantier2013log} for a broad
range of useful insights.

\section{Preliminaries}\label{sec:preliminaries}

\makeatletter{}In this section we remind the reader of some definitions and facts
regarding tempered distributions and homogeneous Sobolev spaces. Some of
the following notation and ideas are from \cite{triebel1983theory}, to
which we refer the reader for more discussion on homogeneous spaces. We
will introduce and construct several linear spaces.  To aid the reader in
keeping track of the various definitions, we include a glossary of these
definitions in the \hyperref[notation]{appendix} on
page~\pageref{pg:notation}.

\subsection{Tempered Distributions and Sobolev spaces}
\label{subsec:tempdist}

Fix a positive integer $d$, and denote by $\s(\R^d)$ the real Schwartz
space, defined to be the set of real-valued functions on $\R^d$ whose
derivatives of all orders exist and decay faster than any \label{not:schw}
polynomial at infinity. A multi-index $\beta = (\beta_1,\ldots,\beta_d)$ is
an ordered $d$-tuple of nonnegative integers, and the order of $\beta$ is
defined to be $|\beta|\colonequals \sum_{j=1}^d \beta_j$. We equip
$\s(\R^d)$ with the topology generated by the family of seminorms
\[
\left\{\raisebox{2pt}{$\displaystyle\|f\|_{n,\beta} \colonequals \sup_{x\in
      \R^d} |x|^n |\partial^\beta f(x)| \,:\,n \geq 0, \: \beta\text{ is a
      multi-index}$} \right\}.
\]
The space $\s'(\R^d)$ of tempered distributions is defined to be
the space of continuous linear functionals on $\s(\R^d)$ \label{not:tdist}.

We take the convention that the Fourier transform $\mathcal F$ acting on a
Schwartz function $\phi$ on $\R^d$ is the function
\[
	\mathcal{F}[\phi](\xi) = \frac{1}{(2\pi)^{d/2}}\int_{\R^d} \phi(x)e^{-i\xi \cdot x} \,dx
\]
which we will often abbreviate as $\wh{\phi}(\xi)$.  The complex Schwartz
space (the space of functions whose real and imaginary parts are in
$\s(\R^d)$) is closed under the operation of taking the Fourier transform
\cite[Section~1.13]{tao-epsilon}, so the inverse Fourier transform
$\mathcal{F}^{-1}$ is well-defined on the complex Schwartz space and
satisfies the formula
\[
\mathcal{F}^{-1}[\phi](x) =
\frac{1}{(2\pi)^{d/2}}\int_{\R^d}\phi(\xi)e^{ix\cdot\xi}\,d\xi.
\]

We define the Fourier transform $\wh{f}$ of a tempered distribution $f$ by
setting $(\wh{f},\phi)\colonequals (f,\wh\phi)$, so that $\mathcal F$ and
$\mathcal F^{-1}$ may be interpreted as operators from $\s'(\R^d)$ to
$\s'(\R^d)$.  Regarding $\phi \in \s(\R^d)$ as a tempered distribution via
$\phi(\psi)\colonequals \int_{\R^d}\phi(x)\,\psi(x)\,dx$, we have the
continuous, dense inclusion $\s(\R^d) \subset \s'(\R^d)$. For the
fundamentals of the theory of distributions, we refer the reader to
\cite[~Appendix B]{lax} or \cite{tao-epsilon}.  For a more detailed
introduction to distribution theory we refer to \cite{friedlander} and
\cite{hormanderV1}.

For $r\in \R$, define $\s_r(\R^d) \subset \s(\R^d)$ to be the set of
Schwartz functions $\phi$ such that $(\partial^\alpha \wh{\phi})(0) = 0$
(or, equivalently, $\int_{\R^d} x^\alpha \phi(x)\,dx=0$) for all
multi-indices $\alpha$ such that $|\alpha|\leq r$ \label{not:schwM}. We
equip $\s_r(\R^d)$ with the subspace topology inherited from $\s(\R^d)$ and
denote by $\s_r'(\R^d)$ the topological dual \label{not:tdistM} of
$\s_r(\R^d)$. 
Observe that $\s_r'(\R^d)$ is canonically isomorphic to
$\s'(\R^d)/\mathcal{T}_r(\R^d)$, where $\mathcal{T}_r(\R^d)$ denotes the
space of polynomials of degree at most $r$ \label{not:poly} on
$\R^d$. Observe also that $\s_{r}(\R^d)=\s(\R^d)$ whenever $r$ is negative,
and that $\s_0(\R^d)=\{\phi \in \s(\R^d) \,:\, \int_{\R^d}\phi(x)\,dx
=0\}$.

Given $r\in \R$, we also consider the space
\[
\wt{\s}_r(\R^d) = \{\phi \in \s(\R^d)\,:\, (\partial^\alpha \phi)(0) = 0
\text{ for all }|\alpha| \leq r\},
\]
which is equal to the image of $\s_r(\R^d)$ under the inverse Fourier
transform operator\label{not:schwMFI}. We define the Fourier transform of an element of
$\s_r'(\R^d)$ as an element of $\wt{\s}_r'(\R^d)$ via $(\wh{f},\phi)
\colonequals (f,\wh{\phi})$ whenever $f\in \s_r'(\R^d)$ and $\phi \in
\wt{\s}_r(\R^d)$.

Define the space
\[
\mathring{H}^s(\R^d) \colonequals \left\{ f \in \s(\R^d)\,:\xi\mapsto |\xi|^s
	\wh{f}(\xi) \in L^2(\R^d)\right\} \label{not:uSob}
\]
and equip $\mathring{H}^s(\R^d)$ with the inner product
\[
(f,g)_{\dot H^s(\R^d)}\colonequals \left(\xi\mapsto |\xi|^s
  \wh{f}(\xi),\xi\mapsto |\xi|^s \overline{\wh{g}(\xi)}\right)_{L^2(\R^d)}.
\]
We define the Sobolev space $\dot H^s(\R^d)$ to be the Hilbert space
completion of $\mathring{H}^s(\R^d)$\label{not:Sob}, which we continuously embed in
$\s_{H}'(\R^d)$ as follows. If $\{f_n\}_{n\geq 1}$ is a Cauchy sequence in
$\mathring{H}^s(\R^d)$ and $\phi \in \s_{H}(\R^d)$, then by Plancherel and
Cauchy-Schwarz we have
\begin{align} \label{eq:cauchy} 
  \lefteqn{|(f_m - f_n,\phi)_{L^2(\R^d)}| \leq} \\ \nonumber &\left(\int
    |\wh{f}_m(\xi) - \wh{f}_n(\xi)|^2|\xi|^{2s}
    \,d\xi\right)^{1/2}\left(\int|\wh{\phi}(\xi)||\xi|^{-2s}\,d\xi\right)^{1/2}.
\end{align}
The first factor on the right-hand side tends to 0 as $\min(m,n) \to\infty$
and the second factor is finite since $\phi \in \s_{H}(\R^d)$. It follows
that $(f_m - f_n,\phi)_{L^2(\R^d)}$ is Cauchy in $\R$, which implies that
we can define a linear map $f:\s_{H}(\R^d) \to \R$ by $(f,\phi) \colonequals
\lim_{n\to\infty} (f_n,\phi)$ for all $\phi \in \s_{H}(\R^d)$. Observing that
$\phi_k \to 0$ in $\s_{H}(\R^d)$ implies
\begin{align*}
\limsup_{k\to\infty} |(f,\phi_k)|^2 &\leq \\
\limsup_{k\to\infty}&\limsup_{n\to\infty} \int|\wh{f}_n(\xi)|^2|\xi|^{2s}
  \,d\xi \times \int|\wh{\phi}_k(\xi)|^2|\xi|^{-2s}\,d\xi = 0,
\end{align*}
we conclude that $f$ is a continuous functional on
$\s_{H}(\R^d)$. Therefore, we may realize $\dot H^s(\R^d)$ as a subset of
$\s'_{H} (\R^d)$ by identifying each Cauchy sequence $\{f_n\}_{n\geq 1}$
with its $\dot H^s(\R^d)$-limit $f\in \s'_{H}(\R^d)$.

We can characterize $\dot H^s(\R^d)$ in another way which will be useful
for the following section. Note that if $(\phi_n)_{n \in \N}$ is an $\dot
H^s(\R^d)$-Cauchy sequence of Schwartz functions converging to $f$ in
$\s'_{H}(\R^d)$, then $(\wh{\phi}_n)_{n\in \N}$ is Cauchy in
$L^2(\R^d,|\xi|^{2s}\,d\xi)$, where $|\xi|^{2s}\,d\xi$ denotes the measure
whose density with respect to Lebesgue measure is $\xi\mapsto
|\xi|^{2s}$. Therefore, there exists $g\in L^2(\R^d,|\xi|^{2s}\,d\xi)$ to
which $\wh{\phi}_n$ converges with respect to the
$L^2(\R^d,|\xi|^{2s}\,d\xi)$ norm. Furthermore, it is straightforward to
verify using the Cauchy-Schwarz inequality that $g=\wh{f} \in
\wt{\s}_H'(\R^d)$. Therefore,
\[
\dot H^s(\R^d) =
\left\{
f \in \s_H'(\R^d) \,:\,
\wh{f} \in L^2(|\xi|^{2s}\,d\xi)
\right\}, \label{not:SobAlt}
\]
where $\wh{f} \in L^2(|\xi|^{2s}\,d\xi)$ means that there exists $g \in
L^2(|\xi|^{2s}\,d\xi)$ such that $(\wh{f},\phi) = \int_{\R^d} g(x) \phi(x)
\,dx$ for all $\phi\in \wt{\s}_H(\R^d)$.

\subsection{The Fractional Laplacian} \label{sec: fractional laplacian}
The fractional Laplacian generalizes the notion of a power $(-\La)^s$ of
the Laplacian from nonnegative integer values of $s$, for which it is
defined as a local operator by iterating the Laplacian, to all real values
of $s$. A standard reference for the fractional Laplacian is
\cite{landkof1972foundations}. Here we use ideas from Section 2 of
\cite{silvestre2007regularity}.
Let $k\in \{-1,0,1,2,\ldots\}$, and let $\phi \in \s_k(\R^d)$. If
$s>-\frac{1}{2}(d+k+1)$, then we set
\begin{equation} \label{eq:fraclapdef_smooth}
(-\La)^s\phi \colonequals \mathcal{F}^{-1}\left[ \xi\mapsto
  |\xi|^{2s}\wh{\phi}(\xi)\right],
\end{equation}
which is well-defined because $\xi\mapsto |\xi|^{2s}\wh{\phi}(\xi)$ is in
$L^1(\R^d)$. Note that \eqref{eq:fraclapdef_smooth} agrees with the local
definition of $-\Delta$ when $s=1$. Because of the singularity at the
origin in its Fourier transform, $(-\La)^s\phi$ is not necessarily
Schwartz. However, it is real-valued, smooth, and has polynomial decay at
infinity:
\begin{proposition} \label{prop:fraclapdecay} Let $k\in
  \{-1,0,1,2,\ldots\}$, $\phi \in \s_k(\R^d)$, and
  $s>-\frac{1}{2}(d+k+1)$. If $\alpha$ is a multi-index, then there exists
  a constant $C$ such that $\phi \in \s_k(\R^d)$ implies
  \begin{equation} \label{eq:polydecay} \sup_{x\in
      \R^d}(1+|x|^{d+2s+k+1})|\partial^\alpha(-\La)^s\phi(x)| \leq C
    \sup_{|\beta|\leq \max(k+1,|\alpha|)}\|\partial^\beta \phi\|_{L^\infty(\R^d)}.
  \end{equation}
  Furthermore, $(-\La)^s\phi$ is real-valued and smooth. 
\end{proposition}
\begin{proof}
  The proof of smoothness is routine: we write the inverse Fourier
  transform using its definition and differentiate under the integral
  sign. To show that $(-\La)^s\phi$ is real-valued, we note that a function
  $\psi\in L^1(\R^d)$ is the Fourier transform of a real-valued function in
  $L^1(\R^d)$ if and only if $\psi(-\xi)$ and $\psi(\xi)$ are complex
  conjugates for all $\xi \in \R^d$. The function $\xi\mapsto
  |\xi|^{-2s}\wh{\phi}(\xi)$ has this property whenever $\wh{\phi}$ does,
  so we may conclude that $(-\La)^s\phi$ is real-valued.

  To prove \eqref{eq:polydecay}, let $\{f_1,f_2\}$ be a partition of unity
  subordinate to the open cover $\left\{\R^d\setminus
  B\left(0,\frac{1}{|x|}\right),B\left(0,\frac{2}{|x|}\right)\right\}$ of $\R^d$,
  and define $\phi_i(\xi)= f_i(\xi) \wh{\phi}(\xi)/|\xi|^{k+1}$ for $i \in
  \{1,2\}$. We calculate
  \begin{align*}
    \partial^{\alpha}(-\La)^s \phi(x) &= C \int_{\R^d} e^{i x\cdot \xi}
    \xi^\alpha|\xi|^{2s+k+1}
    \wh{\phi}(\xi)/|\xi|^{k+1}\,d\xi \\
    &= C \int_{\R^d\setminus B(0,1/|x|)} e^{i x\cdot \xi}
    \xi^\alpha|\xi|^{2s+k+1}
    \phi_1(\xi)\,d\xi + \\
    &\hspace{2cm}C\int_{B(0,2/|x|)} e^{i x\cdot \xi}
    \xi^\alpha|\xi|^{2s+k+1} \phi_2(\xi)\,d\xi,
  \end{align*}
  where $C$ is some constant. To obtain the desired bound for the first
  integral, we write the integral in spherical form and apply
  integration-by-parts with respect to the radial coordinate. For the
  second integral, we bound $\phi_2(x)$ by a constant times
  $\sup_{|\beta|=k+1}|\partial^\beta\phi(0)||\xi|^{-k-1}$ near the origin, 
  using Taylor's theorem.
\end{proof}

For $s>-d/2$, we define\footnote{These spaces are denoted
  $\overline{\s}_s(\R^d)$ in \cite{silvestre2007regularity}.} the space
$\mathcal{U}_s(\R^d)$ to be the space of all functions $\phi\in
C^\infty(\R^d)$ such that
\[
x\mapsto
(1+|x|^{d+2s})(\partial^\alpha f)(x) \label{not:yU?}
\]
is bounded for all multi-indices $\alpha$. These spaces interpolate between
$C^\infty(\R^d)$ and $\s(\R^d)$, as the derivatives of their elements decay
polynomially at a rate indexed by $s$. In particular, $\s(\R^d) \subset
\mathcal{U}_s(\R^d) \subset \mathcal{U}_{s'}(\R^d)$ whenever $s>s'$. We
equip $\mathcal{U}_s(\R^d)$ with the topology induced by the family of
seminorms $f\mapsto \sup_{x\in \R^d} |(1+|x|^{d+2s})(\partial^\alpha
f)(x)|$. By Proposition~\ref{prop:fraclapdecay}, $(-\Delta)^s$ is a
continuous map from $\s_k(\R^d)$ to
$\mathcal{U}_{s+(k+1)/2}(\R^d)$ \label{not:thatsyU}. Furthermore, $(-\La)^s\phi=0$ for $\phi
\in \s_k(\R^d)$ implies that $\wh{\phi}$ vanishes except possibly at the
origin. This implies that $\phi$ is a polynomial, which in turn implies
that $\phi=0$. Therefore, $(-\Delta)^s$ is injective. For all $f$ in the
topological dual of the image $(-\Delta)^s \s_k(\R^d)\subset
\mathcal{U}_{s+(k+1)/2}$, we define $(-\Delta)^sf\in \s_k'(\R^d)$ by
\[
((-\Delta)^sf,\phi) = (f,(-\Delta)^s\phi),
\]
It is straightforward to verify that this definition agrees with
\eqref{eq:fraclapdef_smooth} when $f\in \s(\R^d)$. Observe that
$(-\La)^{s_1}(-\La)^{s_2} = (-\La)^{s_1+s_2}$ for all $s_1,s_2\in \R$.  We
will consider two important examples of elements of $((-\Delta)^s
\s_k(\R^d))'$:

(i) Elements of homogeneous Sobolev spaces. Let $s\in \R$ and $H=s-d/2$. It
is straightforward to verify that $f\in \dot{H}^s(\R^d)$ determines an
element of $((-\Delta)^s \s_H(\R^d))'$ via $(f,\phi)\colonequals (\wh{f},
\wh{\phi})$. Furthermore, the definition of $(-\Delta)^sf$ arising from
this correspondence satisfies $\wh{(-\La)^sf}(\xi)=|\xi|^{2s}\wh{f}(\xi)$.
It follows that $(-\La)^s$ is an isometric isomorphism from
$\dH^{s_0}(\R^d)$ to $\dH^{s_0-2s}(\R^d)$.

(ii) Measurable functions $f:\R^d \to \mathbb{C}$ satisfying
\begin{equation} \label{eq:finite_integral}
\int_{\R^d}|f(x)|(1+|x|^{d+2s+k+1})^{-1}\,dx < \infty.
\end{equation}
Interpreting $f$ as a linear functional on $(-\Delta)^s \s_k(\R^d)$ by
integration against a test function, the continuity of $f$ with respect the
$\mathcal{U}_{s+(k+1)/2}(\R^d)$ topology follows from
Proposition~\ref{prop:fraclapdecay}.

The following proposition gives an alternative representation of the
fractional Laplacian in the case $0<s<1$.
\begin{proposition} \label{prop:diff_quo} For all $f\in \s(\R^d)$, $x\in
  \R^d$, and $s\in(0,1)$, we have
  \begin{align*}
    (-\La)^s f(x) = -\frac{1}{2}C(d,s)\int_{\R^d} \frac{f(x+y)-2f(x)
      +f(x-y) }{|y|^{d+2s}}\,dy,
  \end{align*}
  where $1/C(d,s) = \int_{\R^d} (1-\cos x_1) |x|^{-d-2s} \,dx$.
\end{proposition}
\begin{proof}
  Combine Lemma 3.2 and Proposition 3.3 in \cite{dihitchhiker}.
\end{proof}

When $s\in \{0,1,2,\ldots\}$, the fractional Laplacian $(-\La)^s$ coincides
with the poly-Laplacian, a fundamental example of a higher order elliptic
operator, obtained by iterating the Laplacian operator. For $s\in (0,1)$,
$(-\La)^s$ is a classical example of a non-local pseudo-differential
operator. These two classes generate all the operators of the form
$(-\Delta)^s$ for $s \geq 0$ in the sense that $(-\Delta)^s$ can be written
as a composition of $(-\Delta)^{s-\lfloor s\rfloor}$ and
$(-\Delta)^{\lfloor s\rfloor}$.

For the properties of the poly-Laplacian, we refer the reader to
\cite{gazzola2010polyharmonic} and references therein.  For more properties
of $(-\La)^s$ where $s\in (0,1)$, see \cite{silvestre2007regularity} and
reference therein.

\subsection{White Noise}\label{subsec:white_noise}

On a finite dimensional Hilbert space $\mathcal{H}$ with inner product
$(\cdot,\cdot)_{\mathcal{H}}$, one characterization of standard Gaussian
$h$ on $\mathcal{H}$ is that $h$ is a standard Gaussian in $\mathcal{H}$ if
and only if for all $v\in \mathcal{H}$, $(h,v)_{\mathcal{H}}$ is a centered
Gaussian variable with variance $(v,v)_\mathcal{H}$. If $\mathcal{H}$ is
infinite dimensional, then it is not possible to define a random element of
$\mathcal{H}$ that satisfies this condition
\cite{janson1997gaussian,sheffield2007gaussian}. Nevertheless, we can still
say that a random functional (which we will denote by
$(h,\cdot)_{\mathcal{H}}$) is a standard Gaussian on $\mathcal{H}$ if for
all $v\in \mathcal{H}$, $(h,v)_{\mathcal{H}}$ is a centered Gaussian
variable with variance $(v,v)_{\mathcal{H}}$. Note that such a functional
cannot be almost surely continuous with respect to $\|\cdot
\|_{\mathcal{H}}$.

White noise on $\R^d$ can be regarded as a standard Gaussian on
$L^2(\R^d)$. We will define $W$ to be a random generalized function such
that $(W,f)$ is a centered Gaussian with variance $\|f\|^2_{L^2(\R^d)}$ for
all $f\in \s(\R^d)$.  However, it is not obvious that there exists a
measure on $\s'(\R^d)$ satisfying these conditions. Since we will
rigorously construct the FGF in Section \ref{Rd} in the same manner, we
will review a construction of white noise following
\cite{simon1979functional}.

    We say that a complex-valued function $\Phi$ on $\s(\R^d)$ is the
    characteristic function of a probability measure $\nu$ on $\s'(\R^d)$ if
    \begin{equation} \label{eq:char}
    \Phi(\phi) = \int_{\s'(\R^d)}e^{i(x,\phi)}\,d\nu(x),\quad \text{for all }\phi\in
    \s(\R^d).
    \end{equation}
      \begin{theorem}[Bochner-Minlos theorem for $\s'(\R^d)$]\label{minlos}
    A complex-valued function $\Phi$ on $\s(\R^d)$ is the characteristic
    function of a probability measure $\nu$ on $\s'(\R^d)$ if and only if
    $\Phi(0)=1$, $\Phi$ is continuous, and $\Phi$ is positive
    definite, that is,
    \[
    \sum_{j,k=1}^nz_j\overline{z_k}\Phi(\phi_j-\phi_k)\geq 0,
    \]
    for all $\phi_1,\dots,\phi_n\in \s(\R^d)$, and $z_1,\dots,z_n\in
    \mathbb{C}$. Furthermore, $\Phi$ determines $\nu$ uniquely.
  \end{theorem}

  \begin{proof}
    We briefly sketch the proof given in
    \cite[Theorem~2.3]{simon1979functional} for the case $d=1$; the case
    $d>1$ may be proved similarly. We introduce coordinates to the space
    $\s(\R)$ by writing each function $\phi \in \s(\R)$ as $\phi =
    \sum_{n=1}^\infty (\phi,\phi_n)_{L^2(\R)} \phi_n$, where
    $\{\phi_n\}_{n=0}^\infty$ is the Hermite basis of $L^2(\R)$ defined by
    \[
    \phi_n(x) = \frac{(-1)^n e^{\frac{x^2}{2}} \frac{d^n}{d
   x^n}[e^{-x^2}]}{ \pi^{1/4} \sqrt{2^n n!}}.
    \]
    Identifying $\phi \in \s(\R)$ with
    $\{(\phi,\phi_n)_{L^2(\R)}\}_{n=0}^\infty$ and using the fact that
    $\phi_n$ is an eigenfunction of $-\frac{d^2}{dx^2} + x^2$, we find that
    $\s(\R)$ is isomorphic to the sequence space
    \[
    s = \bigcap_{m\in \Z} \left\{x\in \R^{\N_0} \,: \, \sum_{n}(1+n^2)^m
      |x_n| \equalscolon \|x\|_m < \infty\right\},
    \]
    and the topology of $\s(\R)$ is equivalent to the one induced by the
    family of seminorms $\|\cdot \|_m$. Furthermore, $\s'(R^d)$ is
    isomorphic to
    \[
    s' = \bigcup_{m\in \Z} \left\{x\in \R^{\N_0} \,: \|x\|_m <
      \infty\right\}
    \] if we interpret a sequence $x$ as a linear functional $L_x$ via
    $L_x(y) = \sum_{n=0}^\infty x_n y_n$.

    Bochner's theorem states that characteristic functions of $\R^n$-valued
    random variables are in one-to-one correspondence with normalized,
    continuous, positive definite functions on $\R^n$. Using Bochner's
    theorem, we conclude for all $n \in \N_0$, there is a measure $\mu_n$
    on $\text{span}(\phi_1,\ldots,\phi_n)$ such that
    \[
    \Phi(\phi) = \int e^{i(x,\phi)}\,d\mu_n(x) \quad \text{ for all }\phi \in
    \text{span}(\phi_1,\ldots,\phi_n).
    \]
    By the uniqueness part of Bochner's theorem, these measure are
    consistent. By the Kolmogorov extension theorem, there exists a measure
    $\mu$ on $\R^{\N_0}$ such that \eqref{eq:char} holds for all $\phi$ in
    the linear span of $\{\phi_n\}_{n=0}^\infty$ (that is, when at most
    finitely many of $\phi$'s coordinates are nonzero). It may be shown
    using the continuity of $\Phi$ that $\mu(s')=1$ (see
    \cite{simon1979functional} for details), which allows us to restrict
    $\mu$ to obtain a probability measure on $s'$ and conclude that
    \eqref{eq:char} holds for all $\phi \in \s(\R)$.
  \end{proof}

  We will use Theorem~\ref{minlos} in conjunction with the following
  proposition, which gives sufficient conditions for a functional to be
  positive definite.

  \begin{proposition} \label{prop:posdef}
    Let $(S,(\cdot,\cdot))$ be an inner product space. Then the functional
    $\Phi:S\to \R$ defined by $\Phi(v)\colonequals
    \exp\left(-\frac{1}{2}(v,v)\right)$ is positive definite.
  \end{proposition}

  \begin{proof}
    Let $v_1,\ldots v_n$ be elements of $S$, and choose an orthonormal
    basis $e_1,\ldots,e_m$ of the span of $\{v_1,\ldots v_n\}$. Let $Z =
    (Z_1,\ldots, Z_m)$ be a vector of independent standard normal real
    random variables, and note that for all $u \in \R^m$, we have
    \[
    \Phi\left(\sum_{j=1}^m u_j e_j\right) = \exp\left(-\frac{1}{2}\sum_{i=1}^m
    u_i^2 \right) = \E[e^{iu\cdot Z}],
    \]
    which implies that
    \begin{align*}
    \sum_{j,k=1}^nz_j\overline{z_k}\Phi(v_j-v_k) &=
    \sum_{j,k=1}^nz_j\overline{z_k}\E[e^{i(v_j - v_k)\cdot Z}]
    = \E\left|\sum_{j=1}^n z_j e^{iv_j\cdot Z}\right|^2 \geq 0,
    \end{align*}
    as desired.
  \end{proof}

  We will apply Theorem~\ref{minlos} to construct a measure $\mu$ on
  $\s'(\R^d)$ which we will refer to as white noise $W$.  Recall that
  $\s(\mathbb R^d)$ is a nuclear space and let us define the functional
  \[
  \Phi_0(\phi) = \exp\left(-\frac{1}{2}\|\phi\|_{L^2(\R^d)}^2\right),\quad \text{for all
  }\phi\in \s(\mathbb R^d).
  \]
  By Proposition~\ref{prop:posdef}, this functional is positive
  definite. Since it is also continuous and satisfies $\Phi_0(0) = 1$,
  Theorem~\ref{minlos} implies that there is a unique probability measure
  on $\s'(\mathbb R)$ having $\Phi_0$ as its characteristic function, which we
  define as white noise $W$.  In particular we have the relation
  \[
  \int_{\s'(\mathbb R^d)}e^{i( x, \phi)}\,d\mu(x) =
  \exp\left(-\frac{1}{2}\|\phi\|_{L^2(\R^d)}^2\right),\quad \phi\in \s(\mathbb R^d),
  \]
  which implies for \label{p:white_noise_ext} every $f\in\s (\mathbb R^d)$
  the random variable $(W,f)$ is a centered Gaussian with variance
  $\|f\|^2_{L^2(\R^d)}$.

  An alternative to the preceding view of white noise as a random tempered
  distribution is to regard white noise as a collection of random variables
  $\{(W,f)\,:\,f\in \s(R^d)\}$. The advantage of this perspective is that
  we may extend this collection so that $(W,f)$ is a well-defined random
  variable for all $f\in L^2(\R^d)$. However, in this construction
  $f\mapsto (W,f)$ is no longer almost surely continuous. Recall the
  following definition from \cite{janson1997gaussian} or
    \cite{sheffield2007gaussian}.
  \begin{definition}
    A \textbf{Gaussian Hilbert space} is a collection of Gaussian random
    variables on a common probability space $(\Omega,\mathcal{F},\mu)$
    which is equipped with the $L^2(\Omega,\mathcal{F},\mu)$ inner product
    and is closed with respect to the norm of the
    $L^2(\Omega,\mathcal{F},\mu)$ inner product.
  \end{definition}

  To define a Gaussian Hilbert space $\{(W,f)\,:\,f \in L^2(\R^d)\}$ where
  $W$ is a white noise, we consider the map from $\s(\R^d)$ to
  $L^2(\Omega)$ which sends $\phi\in \s(\R^d)$ to the random variable
  $(W,\phi)$ (here $\Omega$ denotes the underlying probability
  space). Since $\E[(W,\phi)^2] = \|\phi\|^2_{L^2(\R^d)}$, this map is an
  isometry. Since $L^2(\Omega)$ is complete, we may extend this isometry to
  an operator from $L^2(\R^d)$ to $L^2(\Omega)$ by defining
  $(W,f)\colonequals \lim_{n\to\infty} (W,\phi_n)$ where $\phi_n \in
  \s(\R^d)$ and $\phi_n \to f$ in $L^2(R^d)$ as $n\to\infty$.
                    Since $\E[e^{i \xi (h,\phi_n)}] \to \E[e^{i \xi (h,\phi)}]$ by the
  bounded convergence theorem, we have $(W,f) \sim
  \mathcal{N}\left(0,\|f\|_{L^2(\R^d)}^2\right)$ for all $f\in
  L^2(\R^d)$. We call $\{(W,f)\,:\,f\in L^2(\R^d)\}$ a white noise
  Gaussian Hilbert space. Given $f,g\in L^2(\mathbb R^d)$ we may apply this
  fact to $(W,f+g)$ to see that
  \[
  \Cov[(W,f),(W,g)] = (f,g)_{L^2(\R^d)},
  \]
  so if $f$ and $g$ are orthogonal with respect to the $L^2(\mathbb R^d)$
  inner product, then $(W,f)$ and $(W,g)$ are independent.  We may rewrite the
  above expression as
  \[
    \Cov[(W,f),(W,g)] = \int_{\mathbb R^d}\int_{\mathbb R^d}\delta(x-y) f(x)g(y)\,dx\,dy,
  \]
  and say that $W$ has covariance kernel $\delta(x-y)$ (here
  $\delta(x)\,dx$ is notation for the Dirac measure which assigns unit mass
  to the origin). In Section \ref{sec:covariance kernel}, we will
  compute the covariance kernel of the $\FGF_s(\R^d)$ for general $s$
  and $d$.

\section{\texorpdfstring{The FGF on $\R^d$}{The FGF on Rd}}
\label{sec:whole-space}
\makeatletter{}We provide the construction of $\FGF_s(\R^d)$ following the same procedure
as used in Section \ref{subsec:white_noise} for white noise.  We also
compute the covariance kernel for the $\FGF_s(\R^d)$.

\subsection{\texorpdfstring{Definition of $\FGF_s(\R^d)$}{Definition
    of FGF(R^d)}}\label{Rd}

We begin with some heuristic motivation for the rigorous construction that
follows. We want to define $h$ to be a standard Gaussian on $\dot
H^s(\R^d)$. As a first guess, we might try to define a random element $h$ of
$\dot H^s(\R^d)$ so that for all $f\in \dot H^s(\R^d)$, we have
\begin{equation} \label{eq:want}
(h,f)_{\dot H^s(\R^d)} \sim \mathcal{N}\left(0,\|f\|^2 _{\dot H^s(\R^d)}\right).
\end{equation}
However, since $\dot H^s(\R^d)$ is infinite dimensional, no such random
element exists \cite{janson1997gaussian,sheffield2007gaussian}. However, we
note that when $h,f \in \s_{H}(\R^d)$, we have
\begin{equation} \label{eq:switch}
(h,f)_{\dot H^s(\R^d)} = (h,(-\Delta)^sf)_{L^2(\R^d)}.
\end{equation}
Therefore, substituting \eqref{eq:switch} into \eqref{eq:want} and defining
$\phi \colonequals (-\Delta)^s f$, we find that it is reasonable to change
the desired relation from \eqref{eq:want} to
\begin{equation} \label{eq:prelimvar} \E\left[(h,\phi)_{L^2(\R^d)}^2\right]
  = \E\left[(h,(-\Delta)^sf)^2_{\dot H^s(\R^d)}\right] =
  \|(-\Delta)^{-s}f\|^2_{\dot H^{s}(\R^d)}.
\end{equation}
The advantage of this formulation is that we may reinterpret it by
replacing the inner product $(h,\phi)_{L^2(\R^d)}^2$ with the evaluation of
a continuous linear functional $(h,\cdot)$ at $\phi \in \s_{H}(\R^d)$. The
norm on the right-hand side can be rewritten as
\[
\|(-\Delta)^{-s}\phi\|^2_{\dot H^{s}(\R^d)} = \int_{\R^d}
|\xi|^{2s}|\xi|^{-4s} |\wh{\phi}(\xi)|^2 \,d\xi = \|\phi\|^2_{\dot
  H^{-s}(\R^d)}.
\]
So, if $h$ is a random element of $\s_{H}'(\R^d)$ with the property that
\begin{equation} \label{eq:FGFdef}
(h,\phi) \sim \mathcal{N}\left(0,\|\phi\|_{\dot H^{-s}(\R^d)}^2\right) \quad \text{ for all }\phi \in \s_{H}(\R^d),
\end{equation}
then we say that $h$ is a \textbf{fractional Gaussian field with parameter
  $s$} on $\R^d$ and write $h\sim \FGF_s(\R^d)$; note that by abuse of
notation we refer to either $h$ or its law as $\FGF_s(\R^d)$. We note that
when $h \sim \FGF_s(\R^d)$ and $a>0$, the scaling relation
\[
x\mapsto h(ax) \stackrel{d}{=} a^{s-d/2} h
\]
follows from \eqref{eq:FGFdef} (here we are interpreting $x\mapsto h(ax)$
as a distribution via $(x\mapsto h(ax),\phi) = a^{-d}(h,x\mapsto
\phi(x/a))$). For more discussion of FGF scaling and its relationship
  to the scaling properties of statistical physics models, see
  \cite{newman1980self,dobrushin1979gaussian}.

We now provide a construction establishing the existence of fractional
Gaussian fields. We would like to apply the Bochner-Minlos theorem with the
functional $\phi\mapsto
\exp\left(-\tfrac{1}{2}\|\phi\|^2_{H^{-s}(\R^d)}\right)$, but this
functional is only finite when $\phi \in \s_H(\R^d)$, not for all $\phi \in
\s(\R^d)$. Therefore, we define a functional \eqref{eq:charfun} which is
finite for all Schwartz functions and which reduces to $\phi\mapsto
\exp\left(-\tfrac{1}{2}\|\phi\|^2_{H^{-s}(\R^d)}\right)$ whenever $\phi \in
\s_H(\R^d)$.

Let $\{\phi_\alpha:\alpha\text{ is a multi-index}\}$ be a collection
Schwartz functions such that $\int_{\R^d} x^\alpha \phi_\beta(x)\,dx =
\one_{\{\alpha = \beta\}}$. Such a collection may be obtained via a
Gram-Schmidt procedure. Define the functional $C_s:\s_{H}(\R^d) \to \R$ by
\begin{equation} \label{eq:charfun}
C_s(\phi) = \exp\left(-\frac{1}{2}\left\|\phi - \sum_{|\alpha| \leq \lfloor H
    \rfloor}\phi_\alpha \int_{\R^d}x^\alpha \phi(x) \,dx\right\|_{\dot H^{-s}(\R^d)}^2\right).
\end{equation}
By Proposition~\ref{prop:posdef}, $C_s$ is positive definite. Since $C_s$
is also continuous and satisfies $C_s(0) = 1$, we may apply the
\hyperref[minlos]{Bochner-Minlos theorem} to conclude that there is a
random tempered distribution $h$ such that $\E[e^{i(h,\phi)}]=C_s(\phi)$
for all $\phi \in \s(R^d)$. Considering $h$ as a random element of
$\s_{H}'(\R^d)$ by restricting its domain to $\s_{H}(\R^d)$, we obtain a
random element of $\s_{H}'(\R^d)$ which satisfies \eqref{eq:FGFdef} (note
that this restriction is necessary so that the definition does not depend
on the arbitrary choice of functions $\phi_\alpha$).

As we did for white noise (see page~\pageref{p:white_noise_ext}), we may
define a Gaussian Hilbert space $\{(h,\phi)\,:\,\phi \in T_s(\R^d)\}$ for a
class $T_s(\R^d)$ of test functions larger than $\s_{H}(\R^d)$. In
particular, we \label{not:tsrd} define $T_s(\R^d)$ to be the closure of
$\s_\Hu(\R^d)$ in $\dot H^{-s}(\R^d)$. Consider the isometry from
$T_s(\R^d)$ to $L^2(\Omega)$ which sends $\phi\in \s_\Hu(\R^d)$ to the
random variable $(h,\phi)$; we extend this isometry to an operator from
$T_s(\R^d)$ to $L^2(\Omega)$. Writing $\phi \in T_s(\R^d)$ as a limit of
functions in $\s_{H}(\R^d)$ and considering the limit of the corresponding
characteristic functions, we conclude that
\[
(h,\phi) \sim
\mathcal{N}\left(0,\|\phi\|_{\dot H^{-s}(\R^d)}^2\right)\text{ for all
 } \phi\in T_s(\R^d).
\]
We call $\{(h,\phi)\,:\,\phi \in T_s(\R^d)\}$ an $\FGF_s(\R^d)$ Gaussian
Hilbert space.

We now make sense of the expression $h=(-\La)^{-s/2}W$ (see
\eqref{eqn:fgflaplaciandef}). Let $W$ be a white noise on $\R^d$. Observe
that $(-\Delta)^{-s/2} \phi \in L^2(\R^d)$ for all $\phi \in
T_s(\R^d)$. Therefore, we may define for all $\phi\in T_s(\R^d)$ the random
variable $(h,\phi) = (W,(-\Delta)^{-s/2}\phi)$. In this way, we have
constructed a coupling between an $\FGF_s(\R^d)$ Gaussian Hilbert space
$\{(h,\phi)\,:\,\phi \in T_s(\R^d)\}$ and a white noise Gaussian Hilbert
space $\{(W,\phi)\,:\,\phi \in L^2(\R^d)\}$ so that $(h,\phi) =
(W,(-\Delta)^{-s/2}\phi)$. In this sense we can say that
$h=(-\La)^{-s/2}W$. For a coupling in which this equation holds almost
surely, see Proposition~\ref{prop:FGF_coupling}.

\begin{remark}
  Computing $||\phi||_{\dot H^{-s}(\R^d)}^2$ amounts to computing the
  covariance kernel of the $\FGF_s(\R^d)$, which will be done in Section
  \ref{sec:covariance kernel}.
\end{remark}

\begin{remark}
  Since $C_c^{\infty}(\R^d)$ is dense in $\s(\R^d)$, the $\FGF_s(\R^d)$ is
  uniquely determined by the random variables $\{(h,\phi_n)\}_{n\geq1}$
  where $\phi_n$ is a dense (in $\s(\R^d)$) sequence of $C_c^\infty(\R^d)$
  functions.
\end{remark}

\subsection{The FGF covariance kernel}\label{sec:covariance kernel}

Given $h\sim \FGF_s(\R^d)$ with Hurst parameter $\Hu = s-d/2$, let
$G^s(x,y)$ be a function (or generalized function) such that for
$\phi_1,\phi_2\in C^\infty_c(\R^d)\cap T_s(\R^d)$ we have
\begin{equation}\label{eq:correlation}
  \Cov[(h,\phi_1),(h,\phi_2)] = (\phi_1,\phi_2)_{\dot H^{-s}(\R^d)}  = \int_{\mathbb{R}^d}\int_{\mathbb{R}^d}
  G^s(x,y)\phi_1(x)\phi_2(y)\,dx\,dy.
\end{equation}
We call $G^s(x,y)$ a \textbf{covariance kernel} of the $\FGF_s(\R^d)$.
We point out that there can be more than one function $G^s$ satisfying
(\ref{eq:correlation}).  For example, if $\Hu\geq 0$ and $G^s(x,y)$
satisfies (\ref{eq:correlation}), then so does $G^s(x,y)+g(x,y)$ for any
polynomial $g$ in $x$ or in $y$ of degree no greater than $\lfloor
\Hu\rfloor$.

In this section we compute covariance kernels for the fractional Gaussian
fields on $\R^d$. For most positive values of $s$, we find that $G^s(x,y) =
C(s,d)|x-y|^{2\Hu}$ for some constant $C(s,d)$. When $s < 0$ the formula is
similar but involves some derivatives of the delta function, and when $H$ is
a nonnegative integer there is a logarithmic correction. The constant
$C(s,d)$, and therefore also the correlation of $\FGF_s(\R^d)$, is positive
when $s\in (0,d/2)$, is $(-1)^{\lfloor s \rfloor}$ when $s$ is a negative
non-integer, and is $(-1)^{1+\lfloor H \rfloor}$ when $H$ is a positive
non-integer. The statement and proof of the following theorem are adapted
from \cite[Chapter 1, \S 1]{landkof1972foundations}.

\begin{theorem} \label{KernelComp} Each of the following holds.
  \begin{enumerate}[label=(\roman*),leftmargin=*,widest=5]
  \item \label{itm:(i)}  If $\Hu\in\left(-\frac{d}{2},\infty\right)$ (that is, $s>0$) and $\Hu$ is not a nonnegative integer, then
    \[
    G^s(x,y) = C(s,d)|x-y|^{2\Hu}
    \]
    satisfies \eqref{eq:correlation}, where
    \[
    C(s,d) =
    \frac{2^{-2s}\pi^{-d/2}\Gamma\left(\frac{d}{2}-s\right)}{\Gamma(s)}.
    \]
  \item \label{itm:(ii)} If $s<0$ (that is, $\Hu<-d/2$) and $s\in(-k-1,-k)$
    where $k$ is a nonnegative integer, then $ \Cov[(h,\phi_1),(h,\phi_2)]
    $ is given by
    \[
    \int_{\R^d}\int_{\R^d}C(s,d)|x-y|^{2\Hu}\left[\phi_1(x)\phi_2(y) -
      \sum_{j=0}^{k} \phi_1(x)H_j\La^j\phi_2(x)|x-y|^{2j}\right],
    \]
    where
    \[
    H_j = \frac{\Omega_d}{2^j j! d(d+2)\cdots (d+2j-2)},
    \]
    and $\Omega_d = \frac{2\pi^{d/2}}{\Gamma\left(d/2\right)}$ is the
    surface area of the unit sphere in $\R^d$.
                          \item \label{itm:(iii)} If $s=-k$ where $k$ is a nonnegative integer, then
    $\Cov[(h,\phi_1),(h,\phi_2)] $ is given by
    \[
    \int_{\R^d} \phi_1(x)(-\La)^k \phi_2(x)\,dx.
    \]
  \item \label{itm:(iv)} If $\Hu$ is a nonnegative integer $k$, then
    \[
    G^s(x,y) =  2\,c_{-1}^{(\frac{d}{2}+k)}|x-y|^{2\Hu}\log|x-y|,
    \]
    satisfies \eqref{eq:correlation}, where $c_{-1}^{(\frac{d}{2}+k)}$ is
    the residue at $\frac{d}{2}+k$ of $s\mapsto C(s,d)$: 
    \[
    c_{-1}^{(\frac{d}{2}+k)} = \frac{(-1)^{k+1}2^{-2k-d} \pi ^{-d/2}}{k!\,\Gamma(\frac{d}{2}+k)}.
    \]
  \end{enumerate}
\end{theorem}

\begin{remark}\label{rmk:covariance kernel}
  In case \ref{itm:(ii)} above, we can also write
  \[
  G^s(x,y)= C(s,d)|x-y|^{2\Hu}\left[1 -\sum_{j=0}^{k} |x-y|^{2j}H_j\La^{j}\delta(x-y)\right],
  \]
  Similarly, in case \ref{itm:(iii)},
  \[
  G^s(x,y)=(-\La)^k \delta(x-y).
  \]
\end{remark}

\begin{proof}[Proof of Theorem \ref{KernelComp}]
	\ref{itm:(i)} We first assume $\Hu\in\left(-\frac{d}{2},0\right)$ and let
  $h\sim \FGF_s(\R^d)$.  Let $\phi_1,\phi_2\in \s(\R^d)$. Then we may compute the covariance:
  \begin{align*}
	  \Cov[(h,\phi_1),(h,\phi_2)] &= \int_{\mathbb{R}^d}|\xi|^{-2s}\hat{\phi}_1(\xi)\overline{\hat{\phi}_2(\xi)}\, d\xi,\\
    &= (|\xi|^{-2s} \hat{\phi}_1, \hat{\phi}_2)_{L^2(\R^d)},\\
    &= \left(\mathcal{F}^{-1}(|\xi|^{-2s})*\phi_1, \phi_2\right)_{L^2(\R^d)},\\
    &= \int_{\mathbb{R}^d} \int_{\mathbb{R}^d}C(s,d)|x-y|^{2\Hu} \phi_1(x)\phi_2(y)\,dx\,dy,
  \end{align*}
  where in the third line we used the Plancherel theorem, and in the last
  line we used the following Fourier transform formula given in
  \cite[Chapter 1, \S 1]{landkof1972foundations}: 
  \begin{equation}
	  \mathcal{F}\left[C(s,d)|x|^{2\Hu}\right] = |\xi|^{-2s}. \label{eqn:fourCor}
  \end{equation}
  It is important to note that \eqref{eqn:fourCor} is only valid for
  $0<s<d/2$ when the class of test functions is taken to be $\s(\R^d)$.
  Indeed, $|\xi|^{-2s}$ is not a tempered distribution when $s\geq d/2$
  (due to the singularity at the origin), and $C(s,d)|x|^{2\Hu}$ is not a
  tempered distribution when $s\leq0$. 
      Therefore, we extend the Fourier transform formula \eqref{eqn:fourCor}
  outside of the region $\Hu\in\left(-\frac{d}{2},0\right)$.  Now for
  $\Hu\geq 0$ and non-integral, since $\phi_2(y) \in \s_{H}(\R^d)$ it
  follows that for all $N\geq 0$, $\phi_2(y)=O(|y|^{-N})$ as
  $|y|\to\infty$, thus
  \[
  \psi(x,s) \colonequals C(s,d) \int_{\R^d} |x-y|^{2\Hu} \phi_2(y)\,dy
  \]
  is a smooth function of $x$ and an analytic function of $s$ for all $\Hu$
  in the range under consideration \cite[p.\ 48]{landkof1972foundations}.
  Furthermore, as $|x|\to\infty$ we have $\psi(x,s) = O(|x|^{2\Hu})$ so
  that $\phi_1(x)\psi(x,s)$ is integrable in $x$ and analytic in $s$ for
  all $\Hu$ in the range under consideration.  By an analytic continuation
  argument as in \cite[Chapter 1, \S 1]{landkof1972foundations},
  \ref{itm:(i)} follows.

  Formulas (ii) and \ref{itm:(iii)} follow directly from equation (1.1.10) in
  \cite{landkof1972foundations}: 
  \[
  \psi(x,s) = C(s,d) \int_{\mathbb R^d}\left[\phi_2(y) - \sum_{j=0}^k H_j \La^j \phi_2(x) |x-y|^{2j} \right]|x-y|^{2\Hu}\,dy,
  \]
  where $\psi(x,s)$ is an analytic continuation from $0<s<d/2$ to $s\in
  (-k-1,-k]$.  The result for \ref{itm:(iii)} follows from the equality 
  \[
  \psi(x,-k) = (-1)^k \La^k \phi_2(x).
  \]

  Finally, to obtain \ref{itm:(iv)} we will take a limit as $t\to s$ of
  both sides of 
  \begin{equation} \label{eq:iv}
    \| \phi \|^2_{\dot H^{-t}(\R^d)} = \int_{\R^d}\int_{\R^d} C(t,d)
    |x-y|^{2t-d} \phi(x)\phi(y)\,dx\,dy; 
  \end{equation}
  see \cite[p.\ 50]{landkof1972foundations} for more details. Since
  $\phi_1$ and $\phi_2$ are in $\s_k(\R^d)$, we have $\int_{\R^d} x^j
  \phi_1(x)\,dx = \int_{\R^d}y^j\phi_2(y)\,dy = 0$ for all $0\leq j \leq
  k$. This implies
  \begin{multline*}
    \int_{\R^d}\int_{\R^d}|x-y|^{2t-d}\phi_1(x)\phi_2(y)\,dx\,dy =\\
    \int_{\R^d}\int_{\R^d}(|x-y|^{2t-d}-|x-y|^{2s-d})\phi_1(x)\phi_2(y)\,dx\,dy.
  \end{multline*}
  We use Taylor's theorem to write  
  \begin{align*}
    \lefteqn{|x-y|^{2t-d}-|x-y|^{2s-d}} \\ &= 2(t-s)|x-y|^{2k}\ln|x-y| +
    O\left(\left((t-s)|x-y|^{2k}\ln|x-y|\right)^2\right),  
  \end{align*}
  and substitute into \eqref{eq:iv}. Taking $t\to s$ and using $\lim_{t\to
    s}(t-s)C(t,d)=\Res_{t=s}C(t,d)$,
              we obtain \ref{itm:(iv)}.  The formula for $c_{-1}^{(\frac{d}{2}+k)}$ follows from
  the fact that the residue of $\Gamma$ at a negative integer $-n$ is
  $(-1)^n/n!$.
            \end{proof}

\section{The FGF on a domain}\label{sec:domain}
\makeatletter{}\subsection{The space \texorpdfstring{$\dH^s_0(D)$}{dot H\^{}s\_0(D)}}\label{H0D}
Let $s\geq 0$, and let $D\subset \R^d$ be a domain. Recall that
$C^\infty_c(D)$ denotes the set of smooth functions supported on a compact
subset of $D$. \label{not:Ccinfty} We have $C^\infty_c(D) \subset
\dH^{s}(\R^d)$ from the definition of $\dH^{s}(\R^d)$ and the closure of
the complex Schwartz functions under the Fourier transform (see
Section~\ref{subsec:tempdist}). We may therefore define the set
$\dH^s_0(D)$ to be the closure of $C^\infty_c(D)$ in
$\dH^{s}(\R^d)$ \label{not:hs0d} and equip it with the $\dH^{s}(\R^d)$
inner product.

\begin{definition}\label{Defallow}
  We call a domain $D\subset \R^d$ \textbf{allowable} for all $\phi\in
  \s(\R^d)$ there exists $C=C(D,d,\phi)<\infty$ such that for all $g\in
  C_c^\infty(D)$, we have
    \[
    |(\phi,g)_{L^2(\R^d)}|\leq C\|g\|_{\dot H^s(\R^d)}.
    \] 
\end{definition}

We will construct a fractional Gaussian field $\FGF_s(D)$ for all allowable
domains $D\subset \R^d$ (see Remark~\ref{rmk:zero boundary FGF}). The
following lemma gives sufficient conditions for a domain to be allowable.

\begin{lemma}\label{lem:allowable-domain}
  Let $s\geq 0$. If $H = s - d/2$ is not a nonnegative integer, then every
  proper subdomain of $\R^d$ is allowable. If $H = s-d/2$ is a nonnegative
  integer, then a domain $D$ is allowable if $\R^d \setminus D$ contains an
  open set.
\end{lemma}
\begin{proof}
  Let $D\subset \R^d$ be a domain, and let $\phi \in \s(\R^d)$ and $g\in
  C_c^\infty(D)$. We have
  \begin{align*}
    |(\phi,g)_{L^2(\R^d)}| &= \left| \int_{\R^d}
      |\xi|^{-s}\wh{\phi}(\xi)|\xi|^{s}\wh{g}(\xi)\,d\xi \right| \\
    &\leq \left\|\phi\right\|_{\dot H^{-s}(\R^d)} \left\|g\right\|_{\dot H^s(\R^d)},
  \end{align*}
  by the Plancherel formula and Cauchy-Schwarz. If $0\leq s<d/2$, we
  conclude that $\left\|\phi\right\|_{\dot H^{-s}(\R^d)}$ is finite and
  therefore that $D$ is allowable.

  If $H = s - d/2\in\{0,1,\ldots\}$ and $\R^d\setminus D$ contains an open
  set, then let $B$ be a ball contained in $\R^d\setminus D$, and let
  $\eta\in C_c^\infty(\R^d)$ be supported on $B$ and satisfy $\int_{\R^d}
  \eta(x)x^\alpha\,dx=\int_{\R^d} \phi(x)x^\alpha\,dx$ for every
  multi-index $\alpha$ satisfying $|\alpha|\leq H$ (such a function may be
  constructed via a Gram-Schmidt procedure). Since $\eta g = 0$, we have
  \begin{align*}
    (\phi,g)_{L^2(\R^d)}=(\phi-\eta,g)_{L^2(\R^d)} \leq \left\|\phi-\eta\right\|_{\dot H^{-s}(\R^d)} \left\|g\right\|_{\dot
      H^s(\R^d)}
  \end{align*}
  By our choice of $\eta$, the Fourier transform of $\phi - \eta$ vanishes
  to order $H$ at the origin, so $\left\|\phi-\eta\right\|_{\dot
    H^{-s}(\R^d)}$ is finite. Therefore $D$ is allowable. 

  Suppose that $s-d/2 > 0$ is not an integer and that $D \subsetneq \R^d$.
  Without loss of generality, we may assume $D$ does not contain the
  origin.  Let $P^{\phi}$ be the unique polynomial of degree $\lfloor
  H\rfloor$ such that all the derivatives up to order $\lfloor H\rfloor$ of
  $\mathcal{F}(\phi -P^\phi(D)\delta )$ are zero, where $P(D)$ denotes the
  differential operator corresponding to a polynomial $P$ and $\delta$
  denotes a unit Dirac mass at 0. Then
  \[
  |(\phi,g)_{L^2(\R^d)}|=|(\phi-P^{\phi}(D)\delta,g)_{L^2(\R^d)}|\leq
  \left\|\phi-P^{\phi}(D)\delta\right\|_{\dot H^{-s}(\R^d)}
  \left\|g\right\|_{\dot H^s(\R^d)}.
  \] 
  The expression $\left\|\phi-P^{\phi}(D)\delta\right\|_{\dot
    H^{-s}(\R^d)}$ is finite since $\mathcal{F}(\phi-P^{\phi}(D)\delta)$ is
  bounded by a constant times $|\xi|^{\lfloor H \rfloor+1}$ near the origin
  and by a constant times $|\xi|^{\lfloor H \rfloor}$ as $\xi \to\infty$.
\end{proof}

Let $\phi\in \s(\R^d)$, and let $D$ be an allowable domain. By the
definition of allowability, $(\phi,\cdot)_{L^2(\R^d)}$ is a continuous
linear functional on $\dH^s_0(D)$. Therefore, by the Riesz representation
theorem for Hilbert spaces, there exists a unique $f\in \dH^s_0(D)$ such
that $(\phi,g)_{L^2(\R^d)}=(f,g)_{\dot H^s(\R^d)}$ for all $g\in
\dH^s_0(D)$. Writing out the definition of $(f,g)_{\dot H^s(\R^d)}$ and
using the Plancherel formula, we see that this implies that $f$ is the
unique solution of the distributional equation
\begin{equation}\label{eq:fracLap}
(-\La)^s f=\phi, \quad \quad f\in \dH^s_0(D). 
\end{equation}
For $s>0$, we define the semi-norm $||\phi||_{\dot H^{-s}(D)} \colonequals
\|f\|_{\dot H^s(\R^d)}$, where $f$ is determined by $\phi$ via
\eqref{eq:fracLap}. 

Denote by $\s(D)$ the space of functions on $D$ which can be realized as
the restriction of a Schwartz function to $D$. Then
$d(\phi,\psi)\colonequals ||\phi -\psi||_{\dot H^{-s}(D)}$ defines a metric
on $\s(D)$. Taking the completion under this metric as we did at the end of
Section~\ref{subsec:tempdist}, we get a Hilbert space $T_s(D)\subset
\s'(\R^d)$ which will serve as a space of test functions for
$\FGF_s(D)$. \label{not:TsD} 

\subsection{The zero-boundary FGF in a domain}\label{defDomain}
Let $D \subsetneq \R^d$ be an allowable domain, let $s\geq 0$, and define
the functional 
\[
C_{D,\,s}(\phi) \colonequals \exp\left( -\frac{1}{2}\|\phi \|^2_{\dot H^{-s}(D)} \right)
\] 
for $\phi \in \s(\R^d)$. Since $C_{D,s}$ is continuous by the definition of
allowability, we may use Proposition~\ref{prop:posdef} and the
\hyperref[minlos]{Bochner-Minlos Theorem} to $C_{D,\,s}$ to conclude that
there is a unique random element $h_D$ of $\s'(\R^d)$ such that
$(h_D,\phi)$ is a mean-zero Gaussian with variance $||\phi||^2_{\dot
  H^{-s}(D)}$. Since $\E[(h_D,\phi)^2] = 0$ whenever $\phi$ is supported in
$\R^d \setminus D$, the support of $h_D$ is almost surely contained in
$\overline{D}$. We call\footnote{We use the word \textit{boundary} instead
  of \textit{complement} for consistency with the GFF terminology. Note,
  however, that due to the nonlocal nature of the fractional Laplacian, the
  relevant boundary data include the values on $\R^d\setminus \overline{D}$.} $h_D$
the \textit{zero-boundary FGF} on $D$, abbreviated as $\FGF_s(D)$.

\begin{remark} \label{rmk:zero boundary FGF} We construct $h_D\sim
  \FGF_s(D)$ only when $D$ is allowable because we want to ensure that
  $h_D$ is a tempered distribution (rather than a tempered distribution
  modulo a space of polynomials).
                          \end{remark}

We can also define a Gaussian Hilbert space version of $\FGF_s(D)$,
following the corresponding discussion $\FGF_s(\R^d)$ in
Section~\ref{Rd}. In this way we obtain a collection of random variables
$\{(h_D,f)\,:\,f\in T_s(D)\}$ so that $(h_D,f)$ is a centered Gaussian with
variance $\|f||^2_{\dot H^{-s}(D)}$.

If $s$ is an even positive integer, then $\|f\|_{\dH^s_0(D)}=
\|(-\La)^{\Hs}f\|_{L^2(\R^d)}$ for all $f\in C^\infty_0(D)$. If
$s$ is an odd positive integer, then
\[
\|f\|_{\dH^s_0(D)}=\|(-\La)^{\frac{s-1}{2}}f\|_{\dot H^1_0(D)}
\]
for all $f\in C^\infty_0(D)$.  Therefore, if $s=0$ then $h_D$ is white
noise on $D$, and if $s=1$ then $h_D$ is the GFF on $D$.  Thus $\FGF_s(D)$
generalizes the domain versions of white noise and the Gaussian free field.

\subsection{Covariance kernel for the FGF on the unit ball}

Let $s\geq 0$, and let $D$ be an allowable domain. As usual, we say that a
function $G^s_D:D\times D\to \R$ is the $\FGF_s(D)$ covariance kernel if it
satisfies
\begin{equation} \label{eq:green_D}
  \Cov[(h_D,\phi_1),(h_D,\phi_2)] = \int_{\mathbb{R}^d}\int_{\mathbb{R}^d}
  G_D^s(x,y)\phi_1(x)\phi_2(y)\,dx\,dy.
\end{equation} 
for $h_D\sim \FGF_s(D)$ and for all $\phi_1,\phi_2\in C_c^\infty(D)$. We
treat each of the cases 

(i) $s$ is an integer, \\
(ii) $s\in (0,1)$, and \\
(iii) $s$ is a non-integer greater than 1.

Suppose that $s$ is a positive integer, and let $\phi \in \s(B)$. By (2.65)
in Chapter 2 of \cite{gazzola2010polyharmonic}, the unique solution of
\eqref{eq:fracLap} is $f(x)=\int G^s_B(x,y)\phi(y)dy$, where
\begin{equation}\label{Green}
	G^s_B(x,y)= k_{s,d} | x-y|^{2\Hu} \int_1^{\frac{\left||x|y-\frac{x}{|x|}\right|}{|x-y|}}(v^2-1)^{s-1}v^{1-d}dv, \quad x,y\in B
\end{equation}
and 
\[
k_{s,d} = \frac{\Gamma(1+d/2)}{d\pi^{d/2}4^{d-1}((s-1)!)^2}. 
\]
It follows that for all $\phi \in C_c^\infty(D)$, we have 
\begin{equation} \label{eq:different_Green} 
\E[(h_D,\phi)^2] = \|\phi\|_{\dot H^{-s}(D)}^2 = \|f\|_{\dot H_0^s(D)}^2 = 
\iint G_B(x,y)\phi(x)\phi(y) \,dx\,dy,
\end{equation} 
which shows that $G_B^s$ is the $\FGF_s(D)$ covariance kernel. 

Suppose that $0<s<1$. Let $X_t$ denote a $2s$-stable symmetric L\'evy
process, and let $\tau_B$ be the first time $X$ exits $B$. Recall the
definition of the constant $C_{d,s}$ in Theorem~\ref{KernelComp}, and
define $u(x,y)=(2/\pi)^{2s}C_{d,s}|x-y|^{2H}$.  By the potential theory of
$2s$-symmetric stable processes, (see, for example,
\cite{chen1998estimates}), the function
\[
G^s_B(x,y)=u(x,y)-\mathbb{E}^x[u(X_{\tau_B},y)]
\]
is the $\FGF_s(D)$ covariance kernel. The following explicit formula for
$G_s^B$ is given as Corollary 4 in \cite{blumenthal1961distribution}: 
\begin{equation} \label{eq:green01}
G^s_B(x,y)= \tilde{k}_{s,d} | x-y|^{2\Hu}
\int_0^{\frac{(1-|x|^2)(1-|y|^2)}{|x-y|^2}}(v+1)^{-d/2}v^{s-1}dv, \quad
x,y\in B, 
\end{equation} 
where 
\[
\tilde{k}_{s,d} = \frac{\Gamma(d/2)}{4^{s}\pi^{d/2}\Gamma(s)^2}.
\]

Suppose that $s>1$ is not an integer and $\phi \in \s(\R^d)$. We claim that
$G^s_B(x,y)=\int_{B}G^{\lfloor s\rfloor}(x,u)G^{s-\lfloor
  s\rfloor}(u,y)du$ is the covariance kernel for $\FGF_s(\D)$. Indeed, we
may write $(-\Delta)^s = (-\Delta)^{\lfloor s \rfloor}(-\Delta)^{s-\lfloor
  s \rfloor}$ and calculate 
\begin{align*}
(-\Delta)^s\iint G_B^{s-\lfloor s \rfloor}(x,u) &G_B^{\lfloor s
  \rfloor}(u,y) \phi(y) \,dy \\ &= (-\Delta)^{\lfloor s \rfloor} \int G_B^{\lfloor s
  \rfloor}(u,y) \phi(y) \,dy \\ &= \phi(x),
\end{align*}
which implies that $G_B^s$ is the $\FGF_s(D)$ covariance kernel by
\eqref{eq:different_Green}.

\begin{remark}
  Similar results may be obtained for a more general class of domains
  $D$. The ingredients are the corresponding potential theory of the
  poly-Laplacian and fractional Laplacian for $s\in(0,1)$.
\end{remark}

\section{Projections of the FGF}\label{sec:projection}
\makeatletter{}Given a domain $D\subset \R^d$ and a distribution $f$ defined on
$\R^d\setminus D$, if a distribution $g:\R^d\rightarrow \R$ satisfies the
condition
\begin{align*}
\left.f\right|_{\R^d\setminus D}&=\left.g\right|_{\R^d\setminus D} \\
 \left.((-\La)^s g)\right|_{D} &=0, 
\end{align*}
then we call $g$ the \textbf{\textit{s}-harmonic extension} of $f$. In this
section we decompose $h\sim \FGF_s(\R^d)$ as a sum of two random fields, one
of which is supported on $D$ and the other of which may be interpreted as
the $s$-harmonic extension of the values of $h$ on $\R^d\setminus D$.

Let $s>0$, let $D\subsetneq \R^d$ be an allowable domain, and define
\[
\Har_s(D)=\{f\in \dH^s(\R^d) \,:\, \left.((-\La)^sf)\right|_{D}=0 \}.
\]
\begin{proposition} \label{prop:directsum} 
  $\dH^s(\R^d)=\Har_s(D)\oplus\dH^s_0(D)$.
\end{proposition}

\begin{proof}
If $f\in \Har_s(D)$ and $g\in \dH^s_0(D)$, then $(f,g)_{\dot H^s(\R^d)} =
((-\Delta)^s f,g) = 0$. Therefore, $\Har_s(D)$ and $\dH^s_0(D)$
are orthogonal subspaces of $\dH^s(\R^d)$.  

Let $f\in \dH^s(\R^d)$. Since $D$ is allowable, $((-\La
)^sf,\cdot)_{L^2(\R^d)}$ is a continuous functional on
$\dH^s_0(D)$. Therefore, there exists $f_D\in \dH^s_0(D)$ such that for all
$g\in \dH^s_0(D)$, we have $(f,g)_{\dot H^{s}(\R^d)}=(f_D,g)_{\dot H^s(\R^d)}$. In
particular, this implies that $((-\Delta)^s(f-f_D),g) = 0$ for all $g\in
C_c^\infty(D)$, which means that 
\[
\left.(-\Delta)^s(f-f_D)\right|_D=0.
\] 
Thus we can write $f$ as a sum of elements of $\Har(D)$ and $\dH^s_0(D)$ as
$f=(f-f_D)+f_D$.
\end{proof}
Observe that Proposition~\ref{prop:directsum} implies that $\Har_s(D)$ is a
closed subspace of $\dot H^s(\R^d)$. We define the projection operators
$P_Df=f_D$ and $P_D^{\Har} f = f - f_D$. We will make sense of $P_Dh$ and
$P_D^{\Har} h$ almost surely, although these are defined a priori
only for $h \in \dot H^s(\R^d)$ and not for arbitrary elements of
$\s'_H(\R^d)$.

We begin by observing that the solution $f$ of \eqref{eq:fracLap} is given
by $f = P_D(-\Delta)^{-s} \phi$. Indeed, $P_D(-\Delta)^{-s} \phi\in \dot
H_0^s(\R^d)$, and
\[
(P_D(-\Delta)^{-s}\phi,g)_{\dot H^s(\R^d)} = ((-\Delta)^{-s}\phi,g)_{\dot
  H^s(\R^d)} = (\phi,g)_{L^2(\R^d)}, 
\]
since $(P_D^{\Har}(-\Delta)^{-s}\phi,g) = 0$ for all $g\in
C_c^\infty(D)$. Therefore, we may apply the Bochner-Minlos theorem to the
functional 
\[
\Phi(\phi) = \exp\left(-\frac{1}{2}\|P(-\La)^{-s}\phi\|^2_{\dot H^s(\R^d)}
\right)
\]
for $P=P_D$ and for $P=P_D^{\Har}$ to obtain random tempered distributions
$h_D$ and $h_D^{\Har}$, respectively. We call $h^{\Har}_D$ the $s$-harmonic
extension of $h$ restricted to $\R^d\setminus D$. In
Section~\ref{sec:regularity}, we will show that $h^{\Har}_D$
is smooth in $D$ almost surely.

\begin{remark}\label{rmk:harmonic part}
  Like the fractional Gaussian field in $\R^d$, $h^{\Har}_D$ is a random
  element of $\s'_H(\R^d)$. But $h_D$ is a random element of $\s'(\R^d)$, as
  mentioned in Remark~\ref{rmk:zero boundary FGF}.
\end{remark}

Now sample $h^{\Har}_D$ and $h_D$ independently and define
$h=h^{\Har}_D+h_D$. By the uniqueness part of the Bochner-Minlos theorem,
$h$ is an $\FGF_s(\R^d)$. For all $f\in \dH^s_0(D)$, we have 
$(h,f)_{\dH^s(\R^d)}=(h_D,f)_{\dH^s(\R^d)}$ almost surely. Therefore, $h_D$
is almost surely determined by $h$. Thus $h_D^{\Har} = h - h_D$ is also
almost surely determined by $h$. So we can define measurable maps $P_D$ and
$P^{\Har}_D$ on $\s'_H(\R^d)$ such that $h_D=P^Dh \sim
\FGF_s(\R^d)$ and $h^{\Har}_D=P^{\Har}_Dh$ is the harmonic extension of $h$
restricted to $\R^d \setminus D$.

\begin{remark} \label{rem:cond_exp}
  We will sometimes describe the relationship between $h_D$ and
  $h^{\Har}_D$ by saying that $h^{\Har}_D$ is the conditional expectation
  of $h\sim \FGF_s(\R^d)$ given the values of $h$ on $\R^d\setminus D$.
\end{remark}

Because $(-\La)$ commutes with $P_D$ and $P^{\Har}_D$, by the
Bochner-Minlos theorem, we have
\begin{equation}\label{eq:laplacian_action}
  (-\La )h^s_D\stackrel{d}{=}h^{s-2}_D \quad \text{and} \quad (-\La )h^{D,s}_{\Har }\stackrel{d}{=}h^{D,s-2}_{\Har } 
\end{equation}
where $\stackrel{d}{=}$ denotes equality in distribution. 

Suppose $U\subset D$ is another allowable domain. Since projection
operators in $\dH^s(\R^d)$ commute, 
\[ 
P_UP_Dh=P_DP_Uh=P_Uh, 
\] 
\[  
h_D=P_Dh=P_D(P_U^{\Har}h+P_Uh)=P_U^{\Har}h_D+P_Uh 
\] 
almost surely. Moreover, $P^U_{\Har }h_D$ and $P_Uh$ are
independent. As discussed above, $h_U$ and $P_U^{\Har}h_D$ are
determined by $h_D$ almost surely.  Thus we have the following
proposition.
\begin{proposition}
  Given allowable domains $U$ and $D$ such that $U\subset D$, there is a
  coupling $(h_D,h^{\Har}_{U,D},h_U)$ such that 
  \begin{enumerate}[label=(\roman*)]
  \item $h_D=h^{\Har}_{U,D}+h_U$, 
  \item $h_D$ is a zero boundary $\FGF$ on $D$, 
  \item $h_U$ is zero boundary $\FGF$ on $U$, and 
  \item $h^{\Har}_{U,D}$ and $h_U$ are independent and both determined by
    $h_D$ almost surely. 
  \end{enumerate} 
  We call $h^{\Har}_{U,D}$ the harmonic extension of $h_D$ given
  its values on $D\slash U$.
\end{proposition}

By the definition of $h^{\Har }_D$, given $\phi\in C_c^\infty(D) \cap
\s_H(\R^d)$ and $f=(-\La )^{-s}\phi\in \dH^s(\R^d)$, we have $(h^{\Har}_D,
\phi)=(h,(-\La )^sf^{\Har}_D)$. Since $\mathrm{supp}((-\La
)^sf^{\Har}_D))\subset \R^d \setminus D$, we can say that the value of
$h_{\Har }^D$ on $D$ modulo a polynomial of degree at most $\lfloor
H\rfloor$ is determined by values of $h$ on $\R^d \setminus D$. More
precisely, the random variable $\left.h^{\Har}_D \right|_D$ is determined
by $\{(h,\phi) \,: \, \phi\in T_s(\R^d), \mathrm{supp}(\phi)\subset \R^d
\setminus D\}$.

When $s$ is a positive integer, the operator $(-\Delta)^s$ is local, in
that case we have a stronger result: $\left.h^{\Har}_D \right|_D$ is
measurable with respect to the $\sigma$-algebra generated by the
intersection of the value of $h$ on every neighborhood of the boundary
(that is, the action of $h$ on test functions supported on a neighborhood
of the boundary).  This is a generalization of the corresponding Markov
property for the Gaussian free field \cite{sheffield2007gaussian}.

\section{Fractional Brownian motion and the FGF}\label{sec:FBM}
\makeatletter{}

The $d$-dimensional fractional Brownian motion $B$ with Hurst parameter
$H>0$ is defined to be the centered Gaussian process on $\R^d$ with
\begin{equation} \label{eq:FBF} 
\E[B(x)B(y)] = |x-y|^{2H} - |x|^{2H} - |y|^{2H} \quad \text{for all }x,y\in \R^d. 
\end{equation} 
The existence of such a process is guaranteed by the general theory of
Gaussian processes (for example, see Theorem 12.1.3 in
\cite{dudley2002real}), because the right-hand side of \eqref{eq:FBF}
  is positive definite \cite{ossiander1989certain}. The special case
$H=\frac{1}{2}$ is called L\'evy Brownian motion \cite{levy1940mouvement},
\cite{levy1945mouvement}.

\begin{proposition} \label{prop:FBFFGF} 
  If $s \in (d/2,d/2+1)$ (that is, $H\in (0,1)$) and $h\sim\FGF_s(\R^d)$,
  then the process defined by $\wt{h}(x) = (h,\delta_x-\delta_0)$ has the
  same distribution as the fractional Brownian motion with Hurst parameter
  $H$ (up to multiplicative constant).
\end{proposition}

\begin{proof} 
  Let $x\in \R^d$. Since the Fourier transform of $\delta_x$ is $\xi\mapsto
  e^{2\pi i x\cdot \xi}$, one may verify from the definition of the $\dot
  H^{-s}(\R^d)$ norm that $\delta_x - \delta_0$ is an element of $\dot
  H^{-s}(\R^d)$ and therefore an element of $T_s(\R^d)$. So if $h\sim
  \FGF_s(\R^d)$, then we may define $\wt{h}(x) =
  (h,\delta_x-\delta_0)$. Then by Theorem~\ref{KernelComp}(i) we have
\begin{equation} \label{eq:FBFFGF}
\E[\wt{h}(x)\wt{h}(y)]=G^s(x,y)-G^s(0,y)-G^s(x,0), 
\end{equation} 
where $G^s(x,y) = C(s,d)|x-y|^{2H}$. Combining \eqref{eq:FBF} and
\eqref{eq:FBFFGF}, we see that $C(s,d)B$ and $\wt{h}$ have the same
covariance structure. Since both are centered Gaussian processes, this
implies that they have the same law.
\end{proof} 

Since \eqref{eq:FBF} and \eqref{eq:FBFFGF} show that $\wt{h}(0) = B(0) = 0$
almost surely, Proposition~\ref{prop:FBFFGF} establishes that the
$\FGF_s(\R^d)$ can be identified as (a constant multiple of) the fractional
Brownian motion by fixing its value to be zero at the origin.

Denote by $C^{k,\alpha}(\R^d)$ the space of functions on $\R^d$ all of
whose derivatives of order up to $k$ exist and are $\alpha$-H\"older
continuous.\label{not:holder} Note that the differentiability and H\"older continuity of a
function-modulo-polynomials is well-defined, because adding a polynomial to
a function does not affect its regularity properties. 

\begin{proposition} \label{prop:reg} 
  Let $h$ be an FGF on $\R^d$ with Hurst parameter $H>0$, and define
  $k=\lceil H\rceil-1$. Then $h\in C^{k,\alpha}(\R^d)$ almost surely for
  all $0<\alpha<H-\lceil H\rceil$.
\end{proposition} 

\begin{proof} We consider several cases: 

  (i) Suppose that $0<H<1$. By Theorem 8.3.2 in \cite{adler2010geometry},
  fractional Brownian motion is $\alpha$-H\"older continuous for all
  $\alpha < H$. The result then follows from Proposition \ref{prop:FBFFGF}.

  (ii) Suppose that $1<H<2$, and let $s=d/2+H$. As in the case $H\in
  (0,1)$, it is straightforward to verify that $\partial^\alpha \delta_x
  - \partial^\alpha \delta_0\in T_s(\R^d)$ when $|\alpha|\leq 1$ and $x\in
  \R^d$. Therefore, if $h\sim \FGF_s(\R^d)$, we may fix all derivatives of
  $h$ of order up to 1 to vanish at the origin. In this way we obtain a
  scale-invariant function $h_0$ whose restriction to $\s_1(\R^d)$
  coincides with $h$. Since $|h_0(x)|$ has the same law as $|x|^H h_0(1)$ by
  scale invariance, we have $\mathbb{E}|h_0(x)|=c|x|^H$ for all $x \in \R^d$,
  where $c = \E[|h_0(1)|]$. Thus
  \[ 
  \mathbb{E}\left[\int_{|x|>1}
    \frac{|h_0(x)|}{|x|^{d+2}}\,dx\right]=\int_{|x|>1}
  \frac{\mathbb{E}|h_0(x)|}{|x|^{d+2}} \,dx = \int_{|x|>1}
  \frac{c}{|x|^{d+2-H}} <\infty,
  \] 
  which implies that $h_0$ satisfies condition \eqref{eq:finite_integral}
  with $s={1/2}$ almost surely (see Section~\ref{sec: fractional
    laplacian}). Therefore, $\wt{h} \colonequals (-\Delta)^{1/2}h_0$ is
  well-defined as a random element of $\s_0'(\R^d)$. Furthermore, since
  \begin{align*}
  (\wt{h},\phi) = (h,(-\Delta)^{1/2}\phi) &\sim
  \mathcal{N}\left(0,\|(-\Delta)^{1/2}\phi\|_{\dot H^{-s}(\R^d)}\right) \\
  &= \mathcal{N}\left(0,\|\phi\|_{\dot H^{-(s-1)}(\R^d)}\right)
  \end{align*}
  for all $\phi \in \s_0(\R^d)$, we see that $\wt{h}\sim
  \FGF_{s-1}(\R^d)$. Thus $\wt{h}$ is $\alpha$-H\"older continuous for all
  $\alpha < s-1$ by the preceding case. By the proof\footnote{Proposition
    2.8 in \cite{silvestre2007regularity} includes a boundedness hypothesis
    which does not hold here. However, that hypothesis is only used for a
    norm bound also given in the proposition statement. The regularity
    assertion follows from the other hypotheses.}  of
  \cite[Proposition~2.8]{silvestre2007regularity}, $h$ is almost surely in
  $C^{1,\alpha}(\R^d)$.
  
  (iii) If $H=1$, we may apply the same argument with
  $(-\Delta)^{(1-\alpha)/2}$ in place of $(-\Delta)^{1/2}$, which means
  that $\wt{h}\sim \FGF_{(1+\alpha)/2}(\R^d)$.

  (iv) If $H=2$, then we may apply the same reasoning we applied in case
  (ii), leveraging the $H=1$ case. 

  (v) For $H>2$, we note that $f \in C^{k+2,\alpha}(\R^d)$ whenever $\Delta
  f \in C^{k,\alpha}(\R^d)$
  \cite[Theorem~2.28]{folland1999real}. Therefore, the result follows from
  the case $H\in (0,2]$ by induction.
\end{proof}

As an application of the ideas presented in this section, we construct a
coupling of all the fractional Gaussian fields on $\R^d$.

\begin{proposition} \label{prop:FGF_coupling} There exists a coupling of the
  random fields $\{h_s\,:\,s\in \R\}$ such that $h_s\sim \FGF_s(\R^d)$ and
  $h_s = (-\Delta)^{\frac{s'-s}{2}}h_{s'}$ for all $s,s'\in
  \R$. Furthermore, in this coupling $h_s$ determines $h_{s'}$ for all
  $s,s' \in \R$.
\end{proposition}

\begin{proof}
  We will start with an FGF with Hurst parameter 2 and apply the fractional
  Laplacian to obtain FGFs with Hurst parameters in $(0,2)$. The remaining
  FGFs are then obtained by applying integer powers of the Laplacian to
  FGFs with Hurst parameter in $(0,2]$.
  
  Let $h_{2+d/2}\sim\FGF_{2+d/2}(\R^d)$.  As discussed in the proof
  of Proposition~\ref{prop:reg} case (ii), we can fix the values and
  first-order derivatives of $h$ to vanish at the origin to obtain a
  scale-invariant random function $h_0$ whose restriction to $\s_1(\R^d)$
  agrees with $h$. Furthermore, we have
  \[
  \mathbb{E}\left[\int_{|x|>1}
    \frac{|h_0(x)|}{|x|^{d+2s+k+1}}\,dx\right]=\int_{|x|>1}
  \frac{\mathbb{E}|h_0(x)|}{|x|^{d+2s+k+1}} \,dx = \int_{|x|>1}
  \frac{c|x|^2}{|x|^{d+2s+k+1}} <\infty,
  \] 
  whenever $s\in (0,1/2]$ and $k=1$ or when $s\in (1/2,1)$ and
  $k=0$. Therefore, we may define $h_{s'+d/2}=(-\Delta)^{1-s'/2}h_{2+d/2}$
  for all $s'\in (0,2)$. If $s'+d/2\in \R\setminus (0,2]$, define
  $h_{s'+d/2}=(-\Delta)^{\frac{s-s'}{2}}h_{s+d/2}$, where $s$ is the unique
  real number in $(0,2]$ for which $s-s'$ is an even integer.  
    
  It follows from the construction that $h_s\sim \FGF_s(\R^d)$ for all
  $s\in \R$ and that $h_s = (-\Delta)^{\frac{s'-s}{2}}h_{s'}$ for all
  $s,s'\in \R$, which in turn implies that $h_s$ determines $h_{s'}$ for
  all $s,s'\in \R$.
\end{proof}

\section{Restricting FGFs}\label{sec:restriction}
\makeatletter{}In this section we study how fractional Gaussian fields behave when
restricted to a lower dimensional subspace. 

We regard $\R^{d-1}$ as a subspace of $\R^d$ by associating
$(x_1,\ldots,x_{d-1}) \in \R^{d-1}$ with $(x_1,\ldots,x_{d-1},0)\in \R^d$.
For all $\phi \in \s_H(\R^{d-1})$, we define $\phi^{\uparrow} \in
\s'(\R^d)$ by
\[
(\phi^\uparrow,f)\colonequals \int_{\R^{d-1}} f(x)\phi(x) \, dx
\] 
for all $f\in \mathcal{S}(\R^d)$.

\begin{theorem}\label{thm: restriction}
  Fix $s>\frac{1}{2}$, suppose $h^d\sim \FGF_s(\R^d)$. Then
  $\phi^\uparrow\in T_s(\R^d)$ for all $\phi \in \s_H(\R^{d-1})$, which
  means that $(h,\phi^\uparrow)$ is a well-defined random variable almost
  surely (see Section~\ref{Rd}). Moreover, $h^d$ almost surely determines a
  random distribution $h^{d-1}\sim \FGF_{s-1/2}(\R^{d-1})$ such that for
  all $\phi \in C^{\infty}_c(\R^{d-1})\cap \s_H(\R^d)$ fixed, the relation
  \begin{equation}\label{eq:restriction}
    (h^d,\phi^\uparrow)=C (h^{d-1},\phi), 
  \end{equation}
  holds almost surely, where $C$ is a constant depending only on $d$ and
  $s$.
\end{theorem}

We refer to $h^{d-1}$ as the \textbf{restriction} of $h^d$ to $\R^{d-1}$.

\begin{proof}
  Let $\{\eta_k\}_{k \in \N}$ be an approximation to the identity, which
  means that  
  \begin{enumerate}[label=(\roman*)]
  \item $\eta_k$ is smooth for all $k\in\N$, 
  \item $\eta_k\geq 0$, 
  \item $\mathrm{supp}(\eta_k)\subset B(0,1/k)$, and
  \item $\int_{\R^d} \eta_k(x)\,dx = 1$. 
  \end{enumerate} 
  Then $\phi^\uparrow_k\colonequals \eta_k*\phi^\uparrow \in
  \s_H(\R^{d})$, because applying the definition of a convolution and
  making a substitution $w=x-y$ yields  
  \[
  \int_{\R^d} x^\alpha (\eta_k * \phi^\uparrow)(x)\,dx = \int_{\R^d} \eta_k(w)
  \overbrace{\int_{\R^{d-1}} (w + y)^\alpha \phi(y) \,dy}^{0}\,dw = 0,
  \]
  since $\int x^\alpha \phi(x) \,dx = 0$ whenever $|\alpha|\leq
  H$. Moreover we can use Theorem~\ref{KernelComp} to check that
  $\{\phi_k^\uparrow\}_{k \in \N}$ is a Cauchy sequence in $\dot
  H^{-s}(\R^d)$. Since $\phi_k^\uparrow\rightarrow \phi^\uparrow$ in
  $\mathcal{S}'(\R^d)$, we have $\phi_k \rightarrow \phi^\uparrow$ in
  $\dot H^{-s}(\R^d)$ and therefore $\phi^\uparrow \in T_s(\R^d)$.

  Since $\phi_k^\uparrow \rightarrow \phi^\uparrow$ in $\dot H^{-s}(\R^d)$,
  we have $\Var [(h^d,\phi)]=\lim_{k\rightarrow \infty}\|\phi_k\|^2_{\dot
    H^{-s}(\R^d)}$. By definition of $\{ \eta_k \}$, $\Var [(h^d,\phi)]$
  satisfies the formula in Theorem~\ref{KernelComp} where we replace $\R^d$
  by $\R^{d-1}$ and set $\phi_1=\phi_2=\phi$. This is the covariance
  structure of FGF on $\R^{d-1}$ with the same Hurst parameter as $h^d$, up
  multiplicative constant. In other words, there is a constant $C$ so that
  if we define $(h^{d-1},\phi)\colonequals C^{-1} (h^d,\phi^\uparrow)$ for
  all $\phi$ in a countable dense subset $\Phi \subset C_c^\infty(\R^d)$,
  then $h^{d-1}$ has the law of an $\FGF_{s-1/2}(\R^{d-1})$ restricted to
  $\Phi$. Therefore, $h^{d-1}$ extends uniquely to a tempered distribution
  on $\R^{d-1}$, and it satisfies \eqref{eq:restriction} for all $\phi\in
  C_c^\infty(\R^d) \cap \s_H(\R^d)$ by continuity. 
\end{proof}

Since $h^{d-1}$ is a function of $h^d$ almost surely, we can define a
measurable function $\resop$ on $\mathcal{S}'$ such that $h^{d-1}=\resop
h^d$. We call $\resop$ the restriction operator. We can see that $\resop$
maps an FGF to a lower dimensional FGF with the same Hurst parameter. By
applying $\resop$ repeatedly, we can restrict an $\FGF(\R^d)$ to an
$\FGF(R^{d'})$ with the same Hurst parameter, as long as $d'>-2H$.

When the Hurst parameter is positive, $\FGF_s(\R^d)$ is a pointwise-defined
random function, so $\resop$ agrees with the usual restriction of
functions. In particular, we note that the restriction of a
multidimensional fractional Brownian motion with Hurst parameter $H$ to a
line through the origin is a linear fractional Brownian motion with Hurst
parameter $H$.

\section{Regularity of FGF(D)} \label{sec:regularity}
\makeatletter{}Let $s\geq 0$, and let $h\sim \FGF_s(\R^d)$ be coupled with $h_D\sim
\FGF_s(D)$ and $h_D^{\Har}$ as in Section~\ref{sec:projection}, so that $h =
h_D + h_D^{\Har}$. In this section, we will show that $h^{\Har}_D$ is
smooth in $D$ almost surely.  First, we record some results on the
fractional Laplacian following Section 2 of \cite{silvestre2007regularity}.

\begin{lemma}\label{harmonic regularity1}
  If $s$ is a positive integer and $g$ is a distribution on $D$
  such that $(-\La)^sg$ is smooth on $D$, then $g$ is smooth on $D$.
\end{lemma}
\begin{proof}
  Let $s=1$; the case $s>1$ follows by induction. If $\La g=0$, the desired
  result is Weyl's lemma (see Appendix B of \cite{lax}). If $\La g$ is not
  zero, suppose that $U \subset D$ is an arbitrary ball, and let $g_1$ be a
  function which is smooth on $U$ such that
  $\left.(-\La)g_1\right|_U=\left.(-\La) g\right|_U$
  \cite[Corollary~2.20]{folland1999real}. Applying the result for the case
  $\La g = 0$, we see that $g-g_1$ is smooth on $U$, and hence so is
  $g$. Since $U$ was arbitrary, $g$ is smooth on $D$.
\end{proof}

\begin{lemma}\label{harmonic regularity}
  Let $0<s<1$, and let $B\subset \R^d$ be an open ball. If $(-\La)^sf$ is
  smooth in $B$, then $f$ is smooth in $B$.
\end{lemma}

\begin{proof}
  By \cite[(1.6.11), p.\ 121]{landkof1972foundations} (see also Section 5.1
  in \cite{silvestre2007regularity}) the solution $u$ to $\left.(-\La)^s
    u \right|_B = 0$ and $\left. f \right|_{\R^d \setminus B} = \left. f
  \right|_{\R^d \setminus B}$ is given by the convolution $u(y) =
  \int_{\R^d \setminus B} f(x)P(x,y)\,dy$ where $P(x,y)$, the Poisson
  kernel of $(-\La)^s$, is proportional to
  \[
  \frac{(1-|x|^2)^s}{(|y|^2-1)^s|x-y|^d}. 
  \]
  Since $P$ is smooth, we see that 
  \begin{equation} \label{eq:smoothsol} 
  g \text{ is smooth in } B \text{ whenever }
  (-\La)^sg=0 \text{ in } B. 
  \end{equation} 

  By convolving with the Green's function \eqref{eq:green01} for the
  fractional Laplacian on $B$, we see that there exists a continuous
  solution $g$ of the equation $(-\La)^sg=(-\La)^sf$ on $B$ and $g=0$ on
  $\R^d\setminus B$ which is smooth in $B$. Since $f-g$ is also smooth in
  $B$ by \eqref{eq:smoothsol}, we conclude that $f$ is smooth in $B$.
\end{proof}

We can now prove the main result in this section.
\begin{theorem}\label{regularity}
  If $D\subsetneq \R^d$ is an allowable domain, then $h^{\Har}_D$ is
  smooth on $D$ almost surely. If $U\subset D$ is a domain, then
  $h_{U,D}^{\Har}$ is smooth on $U$ almost surely.
\end{theorem}
\begin{proof} We consider several cases: 

  (i) We first suppose $d$ is even, $0<H<1$, and $D$ is a ball. The argument
  in case (ii) of Proposition~\ref{prop:reg} shows that $h$ satisfies
  condition \eqref{eq:finite_integral} with $s=H$ almost surely. Since
  $h_D^{\Har} = h - h_D$ and $h_D$ is supported in $D$, we see that
  $h_D^{\Har}$ also satisfies \eqref{eq:finite_integral} with
  $k=-1$. Therefore, $\Lap^H h_D^{\Har}$ is tempered distribution. By the
  definition of $h_D^{\Har}$ as a random field with $(h,\phi)\sim
  \mathcal{N}\left(0,\|P_D^{\Har} (-\La)^{-s} \phi\|_{\dot
      H^s(\R^d)}\right)$, we have for all $\phi \in C_c^\infty(D)$,
  \[
  (\Lap^{\frac{d}{2}}\Lap^H h_D^{\Har},\phi) = (\Lap^{s} h_D^{\Har},\phi) = 0
  \] 
  almost surely. Considering a countable dense subset of $C_c^\infty(D)$, we
  conclude that $\left.(-\La)^{\frac{d}{2}}(-\La)^H h_D^{\Har}\right|_D = 0$
  almost surely. Thus by Lemma \ref{harmonic regularity1}, $(-\La)^H
  h_D^{\Har}$ is smooth in $ D $ almost surely. By Lemma \ref{harmonic
    regularity}, $h_D^{\Har}$ is smooth in $D$ almost surely.
  
  (ii) Suppose that $d$ is even, $1<H<2$ and $D$ is a ball whose
  closure does not contain the origin. By the scale invariance of $h$, there
  exists $c>0$ so that have $\mathbb{E}|\nabla h(x)|=c|x|^{H-1}$ for all $x
  \in \R^d$. Thus 
  \[ \mathbb{E}\left[\int_{|x|>1} \frac{|\nabla
      h(x)|}{|x|^{d+2H-2}}\,dx\right]=\int_{|x|>1} \frac{\mathbb{E}|\nabla
    h(x)|}{|x|^{d+2H-2}}\,dx <\infty,
  \] 
  which implies that $|\nabla h|$ satisfies condition
  \eqref{eq:finite_integral} with $s=H-1$ almost surely. Since
  $h\in C^1(\R^d)$ and $h_D \in C^1(\R^d)$, we have $h^{\Har}_D\in
  C^1(\R^d)$. Therefore, $|\nabla h^{\Har}_D |$ satisfies
  \eqref{eq:finite_integral} almost surely. By the same argument as in Case
  (i) above, $\partial_{x_i} h^{\Har}_D$ is smooth in $D$ for all $1\leq
  i\leq d$. Therefore $h^{\Har}_D$ is smooth in $D$.
  
  (iii) If $d$ is even, $H \in \{0,1\}$ and $D$ is a ball, then $s$ is an
  integer. So $h^{\Har}_D$ is smooth by Lemma \ref{harmonic regularity1}.
  
  (iv) If $D\subsetneq \R^d$ is allowable, suppose that $U\subset D$ is an
  arbitrary ball. Since $\E[h^2_D(x)]\leq \E[h^2(x)]$ and
  $\frac{\E[(h_D(x)^2]}{(\E|h_D(x)|)^2}=\frac{\E[(h(x))^2]}{(\E|h(x)|)^2}$,
  the arguments for the preceding cases imply that $h_{U,D}^{\Har}$ is
  smooth on $U$. By the formula
  \[
  h^{\Har}_U=P_{\Har}^U(h_D+h^{\Har}_D )=h_{U,D}^{\Har}+h^{\Har}_D,
  \]
  we see that $h^{\Har}_D$ is also smooth on $U$. Since $U$ is arbitrary
  ball contained in $D$, this implies that $h^{\Har}_D$ is smooth in $D$. If
  $U$ is an arbitrary allowable domain in $D$, again by
  $h^{\Har}_U=h_{U,D}^{\Har}+h^{\Har}_D$, we see that $h_{U,D}^{\Har}$ is
  smooth on $U$. 
  
  (v) Suppose that $d$ is even and $s>0$. In the preceding cases we have
  established the result for $H\in [0,2)$. Since
  $\Lap h_D^{\Har}\stackrel{d}{=}\wt{h}_D^{\Har}$ when $s>2$, $h\sim
  \FGF_s(D)$, and $\wt{h} \sim \FGF_{s-2}(D)$, Lemma \ref{harmonic
    regularity1} establishes the result for all $s>0$. 
  
  (vi) Suppose that $d$ is odd, $H>0$, and $H$ is not an even
  integer. Suppose $D$ is an allowable domain in $\R^{d}$ and regard $\R^d$
  as a subspace of $\R^{d+1}$ by mapping $x\in \R^{d-1}$ to
  $(x_1,x_2,\ldots,x_{d-1},0)$. Since $h=h_D+h_D^{\Har}$ where $h_D$ and
  $h_D^{\Har}$ are independent, $h_D^{\Har}$ is the conditional expectation
  of $h$ given $h$ on $\R^d\setminus D$.
    So if we regard $\R^d \setminus D$ as a closed set in $\R^{d+1}$, the
  restriction of $h_{\R^{d+1}\setminus (\R^d \setminus D)}^{\Har}$ has the
  same law as $h_D^{\Har}$ on $\R^d$. Since the restriction of a smooth
  function is smooth, we conclude that $h_D^{\Har}$ is a smooth function in
  $D$ almost surely. 
  
  (vii) If $d$ is odd and $s>0$, we apply the argument in Case (v) to the
  result from Case (vi).
\end{proof}

Since $h=h_D+h^{\Har}_D$ and $h^{\Har}_D $ is smooth in $D$, the regularity
of $\FGF_s(D)$ is the same as the regularity of $\FGF_s(\R^d)$. In other
words, $h_D$ is has $\alpha$-H\"older derivatives of order up to $k$, where
$k = \lceil H \rceil -1$ and $\alpha < H - \lceil H \rceil$
(Proposition~\ref{prop:reg}).

\section{The eigenfunction FGF}
\makeatletter{}
Let $D\subset \R^d$ be a bounded domain, and let $s\in (0,1)$. In this
section we discuss an different notion of a fractional Gaussian field on
$D$, which we call the \textbf{eigenfunction FGF} and denote
$\EFGF_s(D)$. 

The eigenfunction FGF is based on the following definition of a fractional
Laplacian operator on $D$.  Following Section 2.3 in
\cite{sheffield2007gaussian}, we let $\{f_n\}_{n\in \N}$ be an orthonormal
basis of eigenfunctions of the Dirichlet Laplacian on $D$, arranged in
increasing order of their corresponding eigenvalues $\lambda_n>0$. We
define for all $\phi = \sum (f_n,\phi)_{L^2(\R^d)} f_n \in L^2(D)$ the
formal sum
\[
(-\Delta)_D^s \phi = \sum_{n \in \N} \lambda_n^s (\phi,f_n)_{L^2(D)}
f_n,
\]
which converges if $\phi \in C_c^\infty(D)$
\cite{sheffield2007gaussian}. We call $(-\Delta)_D^s$ the eigenfunction fractional
Laplacian operator on $D$. This fractional Laplacian operator determines
a Hilbert space, analogous to $\dot H_0^s(D)$, with inner product given by
\[
\sum_{n \in \N} \lambda_n^s
(\phi_1,f_n)_{L^2(D)}(\phi_2,f_n)_{L^2(D)}.
\]
Note that $\{\lambda_n^{-s/2}f_n\}_{ n \in \N}$ defines an orthonormal
basis with respect to this inner product. We define $\EFGF_s(D)$ to be a
standard Gaussian on this space; more precisely, let $\{Z_n\}_{n \in \N}$
be an i.i.d.\ sequence of standard normal random variables and set for all
$\phi \in C_c^\infty(D)$,
\[
(h,\phi) \colonequals \sum_{n\in \N} Z_n \lambda_n^{-s/2} (f_n,\phi).
\]
By Weyl's law, $\lambda_n = \Theta(n^{2/d})$ as $n\to\infty$, so the sum on
the right-hand side converges almost surely for each $\phi$. Furthermore,
the functional $h$ defined this way is a continuous functional by the same
argument given for the GFF case in \cite{sheffield2007gaussian}.  We define
$\EFGF_s(D)$ to be the law of $h$.

Both the fractional Laplacian and the eigenfunction fractional Laplacian
can be understood in terms of a local operator in $d+1$ dimensions. In
\cite{caffarelli2007extension}, the fractional Laplacian is realized as a
boundary derivative for an extension problem in $\R^d \times [0,\infty)$. A
corresponding analysis for the eigenfunction Laplacian is developed in
\cite{cabre2010positive} and \cite{capella2011regularity} by considering a
similar extension problem in $D\times [0,\infty)$. We carry out an
analogous comparison between $\FGF_s(\R^d)$, $\FGF_s(D)$, and $\EFGF_s(D)$
by realizing each as a restriction of a higher-dimensional random field
that can be understood as a Gaussian free field with spatially varying
resistance (see Propositions~\ref{prop:restrictionRd}, \ref{prop:domain1}
and \ref{prop:domain2} below).

Let $s\in (0,1)$, and define $\alpha = \frac{1-2s}{1-s} \in
(-\infty,1)$. For simplicity, we will assume $d \geq 2$. We introduce
the coordinates $(x_1,\ldots,x_d,z)$ for $\R^{d+1}$, and we define the
following variant of the gradient operator. For $\phi\in \s(\R^{d+1})$, we
set
\[
\nablaa \phi \colonequals \left(\frac{\partial \phi}{\partial x_1},
  \frac{\partial \phi}{\partial x_2},\ldots,\frac{\partial \phi}{\partial x_d},
  |z|^{\alpha/2} \frac{\partial \phi}{\partial z}\right).
\]
We will use
\[
\s_{\text{sym}}(\R^{d+1})\colonequals \{\phi \in
\s_0(\R^{d+1})\,:\,\phi(x,z) = \phi(x,-z) \text{ for all }x,z\}
\]
as a space of test functions. Integrating by parts (see
\cite[Chapter~7]{bass1998diffusions} for more details), we find that for
all $\phi \in \s_{\text{sym}}(\R^d)$, we have
\[
\int_{\R^{d+1}} |\nablaa\phi|^2 =
-\int_{\R^{d+1}}\phi(L_\alpha \phi),
\]
where the operator $L_\alpha$ is defined by
\[
L_\alpha=\frac{\partial^2}{\partial x_1^2}+\frac{\partial^2}{\partial
  x_1^2}+\cdots +\frac{\partial^2}{\partial x_d^2} +
\frac{\partial}{\partial z}\left(|z|^\alpha\frac{\partial}{\partial
    z}\right).
\]

By the Bochner-Minlos theorem, we can define a random tempered distribution
$h_\alpha$ for which
\begin{align} \label{eq:char_mod_GFF}
\E[\exp(i(h_\alpha,\phi))] &= \exp\left(-\frac{1}{2}\int_{\R^{d+1}} \wt{\phi}
  (-L_\alpha)^{-1}\wt{\phi}\right)  \\ \nonumber &=
\exp\left(-\frac{1}{2}\int_{\R^{d+1}} |\nablaa
  (-L_\alpha)^{-1}\wt{\phi}|^2\right),
\end{align}
where $$\wt{\phi}(x,z) = \tfrac{1}{2}(\phi(x,z)+\phi(x,-z)),$$ and
$(-L_\alpha)^{-1}\phi$ satisfies $-L_\alpha(-L_\alpha)^{-1}\phi=\phi$ and
vanishes at infinity---see the proof of
Proposition~\ref{prop:restrictionRd} for the existence of such a
function. We then restrict the domain of $h_\alpha$ to
$\s_{\text{sym}}(\R^{d+1})$, so that \eqref{eq:char_mod_GFF} holds with
$\phi$ in place of $\wt{\phi}$.

Since the right-hand side of \eqref{eq:char_mod_GFF} reduces when
$\alpha=0$ to the GFF characteristic function evaluated at $\wt{\phi}$, we
may think of $h$ as a symmetrized and re-weighted\footnote{We are using the
  term \textit{weight} here in sense described for the disrete GFF in
  Section~4 of \cite{sheffield2007gaussian}.} version of the Gaussian free
field on $\R^{d+1}$.
We define restriction of $h_\alpha$ to $\R^d\times \{0\}$ by
\[
(\left.h_\alpha\right|_{\R^d\times \{0\}},\phi) \colonequals
(h_\alpha,(x,z)\mapsto \phi(x)\delta_0(z)) \quad \text{for }\phi \in \s(\R^d),
\]
where $\delta_0$ denotes the unit Dirac mass at $z=0$. See the proof of
Theorem~\ref{thm: restriction} for an explanation of why the random
variable on the right-hand side is well-defined. More precisely, we will
show that the covariance kernel of $\left.h_\alpha\right|_{\R^d\times
  \{0\}}$ is that of an $\FGF_s(\R^d)$. It follows from continuity of
$\FGF_s(\R^d)$ (as a functional on $\s(\R^d)$) that this restriction can
be defined on a countable dense subset of $\s(\R^d)$ and continuously
extended to obtain a random tempered distribution.

\begin{proposition} \label{prop:restrictionRd} The restriction of
  $h_\alpha$ to $\R^d\times \{0\}$ is an $\FGF_s(\R^d)$, up to
  multiplicative constant.
\end{proposition}

\begin{proof}
  Let $(X_t)_{t \geq 0}$ be a diffusion in $\R^{d+1}$ with a standard
  Brownian motion $B$ in the first $d$ coordinates and the process
  $Z_t\colonequals\left(\frac{\sqrt{2}}{\delta}Y_t\right)^{\delta}$ in the
  last coordinate, where $\delta=2(1-s)$ and $Y$ is $\delta$-dimensional
  Bessel process reflected symmetrically at $0$. An application of It\=o's
  formula reveals that $L_\alpha$ is the generator of $X$. We define the
  Green's function
  \[
  G_\alpha(x_1,x_2) \colonequals \lim_{\eps \to 0} (2\eps)^{-d-1}
  \E^{x_1}\left[\int_0^\infty \one_{\{X_t \in Q(x_2,\eps)\}}\,dt\right],
  \quad x_1,x_2\in \R^{d+1},
  \]
  where $Q(x,\eps) \colonequals \{y\in \R^{d+1}\,:|x_k - y_k|<\eps\text{
    for all }1\leq k \leq d+1\}$.  
  Since $d+1\geq 3$ implies that $X$ is transient, the limit on the
  right-hand side is well-defined---the proof is similar to the proof for
  the case $\alpha = 0$ \cite[Section~3.3]{morters2010brownian}. Since
  $L_\alpha$ is the generator of $X$, we have
  \[
  ((-L_\alpha)^{-1} \phi)(x_1) = \int_{\R^{d+1}}
  G_\alpha(x_1,x_2)\,\phi(x_2)\,dx_2 
  \]  
  for all $\phi \in \s_{\text{sym}}(\R^{d+1})$ (see Chapter~II in
  \cite{bass1998diffusions}).  Therefore,
  \[
  \E[(h_\alpha,\phi)^2] = \int_{\R^{d+1}}\int_{\R^{d+1}}
  G_\alpha(x_1,x_2)\,\phi(x_1)\phi(x_2)\,dx_1\,dx_2.
  \]
                  We denote by $(\ell_s)_{s\geq 0}$ the local time of $Z$ at 0 and by
  $(\tau_t)_{t\geq 0}$ the inverse function of $\ell$
  \cite{revuz1999continuous}. Then $(X_{\tau_t})_{t\geq 0}$ is a
  $2s$-stable L\'evy process in $\R^d$, and the integral $\int_0^s
  \tfrac{1}{2\eps}\one_{\{Z_u \in (-\eps,\eps)\}}\,du$ converges almost
  surely to $\ell_s$ as $\eps \to 0$
  \cite{molchanov1969symmetric}. Therefore, the Green's function of the a
  $2s$-stable L\'evy process in $\R^d$ evaluated at $x_1,x_2\in \R^d\times
  \{0\}$ equals 
  \begin{align*}
    \lefteqn{\lim_{\eps \to 0}(2\eps)^{-d}\E^{x_1}\left[\int_0^\infty \one_{\{B_{\tau_t}\in
          Q(x_2,\eps)\}} \, dt\right]}
    \\    
    &=\lim_{\eps \to 0} (2\eps)^{-d}\E^{x_1}\left[\int_0^\infty \one_{\{B_s\in
        Q(x_2,\eps)\}} \frac{d\ell_s}{ds}\,ds \right]
    \\
    &= \lim_{\eps \to 0} (2\eps)^{-d-1}\E^{x_1}\left[\int_0^\infty \one_{\{B_s\in
        Q(x_2,\eps)\}}\one_{\{Z_s\in
        Q(0,\eps)\}}ds\right] = G_\alpha(x_1,x_2). 
  \end{align*}
  In other words, the restriction to $\{z=0\}$ of the Green's function of
  $X$ is equal to the Green's function of a $2s$-stable L\'evy process in
  $\R^d$. The latter is proportional to $|x_1-x_2|^{2s-d}$
  \cite[(1.1)]{chen1998estimates}, and the covariance kernel of
  $\FGF_s(\R^d)$ is also proportional to $|x_1-x_2|^{2s-d}$ by
  Theorem~\ref{KernelComp}. Since the law of a centered Gaussian process is
  determined by its covariance kernel, this concludes the proof.
      \end{proof}

\begin{figure}[ht]
  \center
 \includegraphics[width=5in]{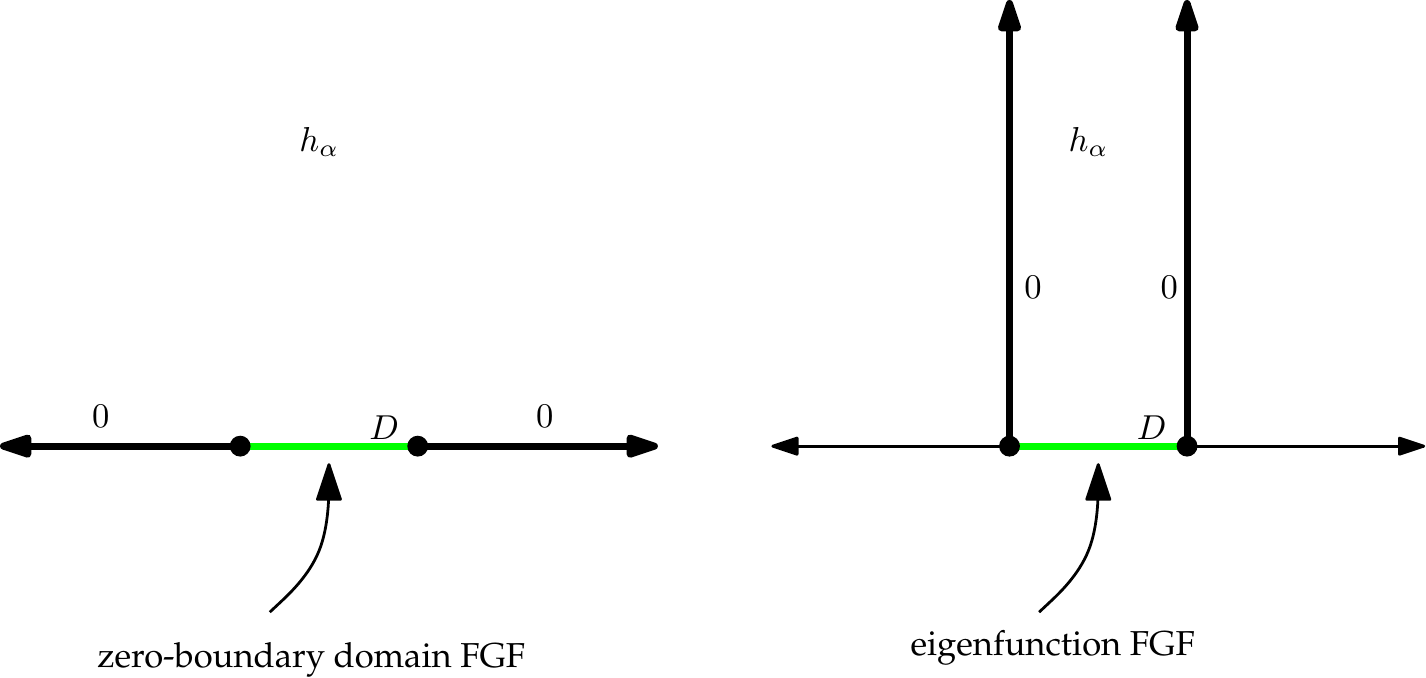}
 \caption{\label{fig:eigenfunctionfgf} Propositions~\ref{prop:domain1} and
   \ref{prop:domain2} describe the relationship between $\FGF_s(D)$ and
   $\EFGF_s(D)$. We obtain an $\FGF_s(D)$ by subtracting from $h_\alpha$ its
   conditional expectation given its values on $(\R^d\setminus D)\times
   \{0\}$ and restricting to $D\times \{0\}$, and we obtain an $\EFGF_s(D)$
   by subtracting from $h_\alpha$ its conditional expectation given its
   values on $\partial D\times \R$ and restricting to $D\times \{0\}$.}
\end{figure}

In Propositions~\ref{prop:domain1} and \ref{prop:domain2} below, we discuss
projections of $h_\alpha$ onto certain subdomains of $\R^{d+1}$. These
projections are analogous to those discussed for the FGF in
Section~\ref{sec:projection}. We state these propositions using terminology
described in Remark~\ref{rem:cond_exp}, to which we refer the reader for a
rigorous interpretation. 

\begin{proposition} \label{prop:domain1} Let $D\subset \R^d$ and define
  $h_D$ to be the restriction to $D\times \{0\}$ of $h_\alpha$ minus the
  conditional expectation of $h_\alpha$ given its values on $(\R^d\setminus
  D)\times \{0\}$. Then $h_D \sim \FGF_s(D)$, up to multiplicative
  constant.
\end{proposition}

\begin{proof}
  This result follows immediately from Proposition~\ref{prop:restrictionRd}
  and the fact that $h\sim \FGF_s(\R^d)$ minus its conditional expectation
  given its values on $\R^d \setminus D$ has the law of an $\FGF_s(D)$. 
\end{proof}

\begin{proposition} \label{prop:domain2} Let $D\subset \R^d$ and define
  $\wt{h}_D$ to be the restriction to $D\times \{0\}$ of $h_\alpha$ minus
  the conditional expectation of $h_\alpha$ given its values on $\partial D
  \times \R$. Then $\wt{h}_D \sim \EFGF_s(D)$, up to multiplicative
  constant.
\end{proposition}

\begin{proof}
  By the definition of the eigenfunction FGF, it suffices to show that if
  $f_1$ and $f_2$ are $L^2(D)$-normalized eigenfunctions of the
  Dirichlet Laplacian on $D$ with eigenvalues $\lambda_1$ and $\lambda_2$,
  then
  \[
  \E[(\wt{h}_D,f_1)(\wt{h}_D,f_2)]=C \lambda_1^{-s} \one_{\{\lambda_1=\lambda_2\}}
  \]
  for some constant $C$. (We will use $C$ to denote a generic constant
  whose value may change throughout the proof.)

  It is straightforward to verify that $h_\alpha$ minus the conditional
  expectation of $h_\alpha$ given its values on $\partial D \times \R$ is
  equal in law to the field $h^{\text{cyl}}_\alpha$ whose covariance kernel
  is given by the Green's function of the diffusion $X$ defined in the
  proof of Proposition~\ref{prop:restrictionRd} stopped upon hitting the
  cylinder $\partial D \times \R$. Equivalently, the covariances of
  $h^{\text{cyl}}_\alpha$ are given in terms of the inverse $L_\alpha^{-1}$ of the operator
  $L_\alpha$ with zero boundary conditions on $\partial D \times \R$ via
  $\E[(h^{\text{cyl}}_\alpha,\phi)^2] = \int_{D\times \R} \phi
  \wt{L}_\alpha^{-1}\phi$.

  Let $w_\lambda(z)$ be the function on $\R$ which satisfies
  $w_\lambda(0)=1$, $w_\lambda(\infty) = 0$,
  \[
  -\lambda w_\lambda(z) + \frac{\partial}{\partial
    z}\left(z^\alpha\frac{\partial w_\lambda }{\partial z}\right) = 0 \quad
  \text{ for all }z\in (0,\infty),
  \]
  and $w_\lambda(z)=w_\lambda(-z)$ for all $z \in \R$. A symbolic ODE
  solver may be used to express $w_\lambda$ in terms of the modified Bessel
  function of the second kind $K_s$ as
  \[
  w_\lambda(z) = C \lambda ^{s/2} z^{s/(2-2 s)} K_s\left(2 (1-s) z^{\frac{1}{2-2 s}}
   \sqrt{\lambda }\right).
  \]
  We define the operator $L_{\alpha,\lambda} = -\lambda +
  \frac{\partial}{\partial z}\left(z^\alpha\frac{\partial }{\partial
      z}\right)$. Integration by parts reveals that for all $\phi \in
  C_c^\infty(\R)$, we have
  \begin{equation} \label{eq:delta0}
    (-L_{\alpha,\lambda}w_\lambda,\phi)\colonequals
    (w_\lambda,-L_{\alpha,\lambda}\phi) = \lim_{z\to 0}z^\alpha
    w_\lambda'(z)\phi(z) = C \lambda^{s} \phi(0)
  \end{equation}
  for some constant $C$, where in the last step we have used the expansion
  \begin{align*}
    K_s(t) &= 2^{s-1} \Gamma (s) t^{-s} + 2^{-s-1} \Gamma (-s)t^s
      -\frac{2^{s-3} \Gamma (s) t^{2-s}}{s-1} + O(t^{2+s})
  \end{align*}
  as $t\to 0^+$. We may restate \eqref{eq:delta0} by writing $\delta_0 =
  C\lambda^{-s} (-L_{\alpha,\lambda})w_\lambda$, where $\delta_0$ denotes
  the unit Dirac mass at the origin. Therefore, using the relation
      \[
  (\wt{h}_D,\phi)\colonequals \lim_{k\to\infty}({h}_\alpha^{\text{cyl}},(x,z)\mapsto
  \psi(x)\eta_k(z)),
  \]
  where $\{\eta_k\}_{n\in \N}$ is an approximation
  to the identity, we have
  \begin{align*}
    \lefteqn{\E[(\wt{h}_D,f_1)(\wt{h}_D,f_2)]} \\ &= C
    \lambda_1^{-s}\lambda_2^{-s} \int_{D\times \R} f_1(x)
    (L_{\alpha,{\lambda_1}} w_{\lambda_1})(z)
    (-L_\alpha)^{-1} [f_2(x) L_{\alpha,\lambda_2} w_{\lambda_2}(z)] \,dx
    \,dz \\
    &= C
    \lambda_1^{-s}\lambda_2^{-s} \int_{D\times \R} f_1(x)
    (L_{\alpha,{\lambda_1}} w_{\lambda_1})(z)
    f_2(x) w_{\lambda_2}(z) \,dx \,dz \\
    &= C\lambda_1^{-s}\lambda_2^{-s}\left(\int_D f_1(x)f_2(x) dx\right)
    \left(\int_\R
      (L_{\alpha,\lambda_1} w_{\lambda_1})(z) w_{\lambda_2}(z)  dz\right) \\
    &= C \, \one_{\{\lambda_1=\lambda_2}\}\lambda_1^{-2s}\lambda_1^s = C \,
    \one_{\{\lambda_1=\lambda_2\}}\lambda_1^{-s},
  \end{align*}
  as desired.
                                                                              \end{proof}

\section{FGF local sets}\label{sec:localset}
\makeatletter{}\subsection{FGF with Boundary Values}
In Section~\ref{sec:domain}, we defined the FGF on a domain with zero
boundary conditions. It is also natural to consider other boundary
conditions to give rigorous meaning to the idea that the conditional law of
the FGF in $D$ given the values of $h$ outside $D$ is an FGF on $D$ with
boundary value $\left. h \right|_{\R^d\setminus D}$. For simplicity, we
only consider the case where $D$ is bounded and the boundary values are
Schwartz.

\begin{definition}\label{def:boundary FGF}
  Given a bounded domain $D$ and a Schwartz function $f$ which is
  $s$-harmonic in $D$, the random distribution $f+h_D$ is called the FGF on
  $D$ with boundary values $\left.f\right|_{\R^d\setminus D}$.
\end{definition}

\subsection{Local Sets of the FGF on a Bounded Domain}
The concept of a \emph{local set} of the Gaussian free field is developed
in \cite{schramm2010contour}. It turns out to be an important concept and
tool in the study of couplings between the GFF and random closed sets such
as SLE (\cite{schramm2010contour}, \cite{millerimaginary1},
\cite{millerimaginary2}, \cite{millerimaginary3}, \cite{millerimaginary4}).
The theory of local sets of the Gaussian free field carries over to the FGF
setting with minimal modification. 

Let $\Gamma_D$ be the space of all closed non-empty subsets of
$\overline{D}$.  We endow $\Gamma$ with the Hausdorff metric induced by
Euclidean distance: the distance between sets $S_1, S_2 \in \Gamma$ is
 \[
 d_{\text{Haus}}(S_1,S_2):=\max\Bigl \{ \sup_{x \in S_1} \text{dist}(x, S_2), \sup_{y
   \in S_2} \text{dist}(y, S_1)\Bigr \},
\] 
where $\text{dist}(x,S)\colonequals \inf_{y\in S}|x-y|$.  Note that
$\Gamma$ is naturally equipped with the Borel $\sigma$-algebra induced by
this metric. Furthermore, $\Gamma_D$ is a compact metric space \cite[pp.\
280-281]{Munkres1999}. Note that the elements of $\Gamma$ are themselves
compact.
 
Given $A \subset \Gamma$, let $A_\delta$ denote the closed set containing
all points in $\Gamma$ whose distance from $A$ is at most $\delta$.  Let
$\mathcal A_\delta$ be the smallest $\sigma$-algebra in which $A$ and the
restriction of $h$ (as a distribution) to the interior of $A_\delta$ are
measurable.  Let $\mathcal A = \bigcap_{\delta \in \mathbb Q, \delta > 0}
\mathcal A_\delta$.  Intuitively, this is the smallest $\sigma$-algebra in
which $A$ and the values of $h$ in an infinitesimal neighbourhood of $A$
are measurable.
 
Given a random closed set $A\subset D$ and deterministic open subset
$B\subset D$, we define the event $S=\{A \cap B = \emptyset\}$ and the
random set $\tilde{A} \colonequals A$ if $S$ occurs and $\emptyset$
otherwise.
 
 \begin{lemma} \label{l.localdefequivalence}
 Let $D$ be a bounded domain, suppose that $(h, A)$ is
 a random variable which is a coupling of an instance $h$ of the FGF with
 a random element $A$ of $\Gamma$.  Then the following are equivalent:
   \begin{enumerate}[label=(\roman*)]
 \item \label{i.Bcond} For each deterministic open $B \subset D$, the event
   $A \cap B = \emptyset$ is conditionally independent, given the projection
   of $h$ onto $\Har_s(B)$, of the projection of $h$ onto $\dH^s_0(B)$.  In
   other words, the conditional probability that $A \cap B = \emptyset$
   given $h$ is a measurable function of the projection of $h$ onto
   $\Har_s(B)$.
 
\item \label{i.B2cond} For each deterministic open $B \subset D$, we have
  that {\em given} the projection of $h$ onto $\Har_s(B)$, the pair $(S,
  \tilde A)$ is independent of the projection of $h$ onto $\dH^s_0(B)$.
 
 \item \label{i.Acond} Conditioned on $\mathcal A$, (a regular
 version of) the conditional law of $h$ is that of $h_1 + h_2$ where
 $h_2$ is a zero boundary FGF on $D\setminus A$
 (extended to all of $D$ by setting $ h_{1}|_A =0 $  )
 and $h_1$ is an $\mathcal A$-measurable random distribution
 (i.e., as a distribution-valued function on
 the space of distribution-set pairs $(h, A)$, $h_1$ is $\mathcal A$-measurable)
 which is almost surely $s$-harmonic on $D \setminus A$.
 \item \label{i.twostep} A sample with the law of $(h,A)$ can be produced as
 follows.
   First choose the pair $(h_1, A)$
 according to some law where $h_1$ is almost surely $ s $-harmonic on $D \setminus A$.  Then
 sample an instance $h_2$ of zero boundary FGF on $D \setminus A$ and set $h=h_1+h_2$.
 \end{enumerate}
 \end{lemma}
 
 Lemma \ref{l.localdefequivalence} may be proved by making minor
 modifications to the proof of Lemma 3.9 in \cite{schramm2010contour} to
 generalize from the setting $s=1, d=2$ to arbitrary $s\in \R$ and $d\geq
 1$.
 
 We say a random closed set $A$ coupled with an instance $h$ of the FGF, is
 \textbf{local} if one of the equivalent conditions in Lemma
 \ref{l.localdefequivalence} holds.  For any coupling of $A$ and $h$, we
 use the notation $C_A$ to describe the conditional expectation of the
 distribution $h$ given $\mathcal{A}$.  When $A$ is local, $C_A$ is the
 distribution $h_1$ described in \ref{i.Acond} above.

 Given two distinct random sets $A_1$ and $A_2$ (each coupled with a FGF
 $h$), we can construct a coupling $(h, A_1, A_2)$ such that the marginal
 law of $(h, A_i)$ (for $i \in \{1,2\}$) is the given one, and conditioned
 on $h$, the sets $A_1$ and $A_2$ are independent of one another.  This can
 be done by first sampling $h$ and then sampling $A_1$ and $A_2$
 independently from the regular conditional probabilities.  The union of
 $A_1$ and $A_2$ is then a new random set coupled with $h$.  We denote this
 new random set by $A_1 \check{\cup} A_2$ and refer to it as the {\textbf
   conditionally independent union} of $A_1$ and $A_2$. The following lemma
 is analogous to \cite[Lemma~3.6]{schramm2010contour},
 
 \begin{lemma} \label{continuumlocalunion} If $A_1$ and $A_2$ are local
   sets coupled with the GFF $h$ on $D$, then their conditionally
   independent union $A = A_1 \check{\cup} A_2$ is also local. Moreover,
   given $\mathcal A$ and the pair $(A_1, A_2)$, the conditional law of $h$
   is given by $C_A$ plus an instance of the FGF on $D \setminus
   A$.
 \end{lemma}

 \subsection{An example of a local set}

 Certain level lines of the Gaussian free field are studied in
 \cite{schramm2010contour} and shown to be local sets. We will show that
 certain level sets of fractional Gaussian fields with positive Hurst
 parameter are also local sets.

 Let $c_1,c_2>0$, let $s>d/2$ and let $h$ be the $\FGF_s$ on the unit ball
 $B$ in $\R^d$ with boundary values $c_1$ on the upper hemisphere, $-c_2$
 on the lower hemisphere, and zero outside a compact set. Then there is a
 unique surface whose boundary equals between the boundary of the upper
 hemisphere and on which $h=0$. This surface separates a region where $h$
 is positive and a region where $h$ is negative. We call this interface the
 level set of $h$ and denote it by $L$. To see that $L$ is a local set, fix
 $\delta>0$ and let $L_\delta$ be the intersection of $D$ with the union of
 all closed boxes of the grid $\delta \mathbb{Z}^d$ that intersect $L$. For
 each fixed closed set $C$, the event $\{L_\delta=C\}$ is determined by
 $h|_C$. Given a deterministic open set $U\cap C=\emptyset$, the projection
 of $h$ to $\dH^s_0(U)$ is independent of $h|_C$. Thus $L_\delta$ is
 local. Letting $\delta\rightarrow 0$, we see that $L$ is local.

\section{Spherical decomposition}\label{sec:spherical}
\makeatletter{}Since the fractional Gaussian field on $\R^d$ is isotropic (that is,
invariant under rotations), it is natural to consider its decomposition
under spherical coordinates. There is a general theorem \cite[Chapter
7]{wongRF} decomposing any isotropic Gaussian random field into a countable
number of mutually uncorrelated single-parameter stochastic
processes. However, since the $\FGF$ is a tempered distribution modulo a
space of polynomials (rather than a tempered distribution) and since it has
a special form, we will give the spherical decomposition directly. 

\subsection{FGF spherical average processes}
Let $S^{d-1}$ denote the unit sphere in $\R^d$, and define $\Omega_d$
to be the area of $S^{d-1}$.  If $f$ is a continuous function on $\R^d$,
then we define the spherical average process $\overline{f}:(0,\infty)\to
\R$ by $\overline{f}(r) =
\frac{1}{\Omega_d}\int_{S^{d-1}}f(r\sigma)\,d\sigma$, where $\,d\sigma$
denotes $(d-1)$-dimensional Lebesgue measure on $S^{d-1}$.  We calculate
that for all $\phi \in C_c^\infty((0,\infty))$,
\begin{equation} \label{eq:average} \int_0^\infty \overline{f}(r) \phi(r)
  \,dr = \frac{1}{\Omega_d}\int_{\R^d} f(x)
  \frac{\phi(|x|)}{|x|^{d-1}} \,dx.
\end{equation} 
Let $s\geq 0$, and let $h\sim \FGF_s(\R^d)$. Motivated by
\eqref{eq:average}, we define the \textbf{spherical average process}
$\overline{h}$ of $h$ by 
\[
(\overline{h},\phi) \colonequals \frac{1}{\Omega_d}\left(h,x\mapsto
  \frac{\phi(|x|)}{|x|^{d-1}}\right) \quad \text{for all } \phi \in
C_c^\infty((0,\infty)) \cap \s_H(\R). 
\]
Note that if $\phi \in C_c^\infty((0,\infty)) \cap \s_H(\R)$, then
$x\mapsto \phi(|x|)/|x|^{d-1}$ is in $\s_H(\R^d)$, so this definition makes
sense. 

The sphere average process of an FGF is a random distribution, since $h$ is
a random tempered distribution and $\phi_n \to 0$ in $C_c^\infty((0,\infty))$
implies $x\mapsto \frac{\phi_n(|x|)}{|x|^{d-1}}$ converges to 0 in
$\s(\R^d)$. To find the covariance kernel of $\overline{h}$, we calculate
\begin{align*} 
  \E[(\overline{h},\phi)^2] &= \frac{1}{\Omega_d^{2}}\E\left[\left(h,x\mapsto
      \frac{\phi(|x|)}{|x|^{d-1}}\right)^2\right] \\
  &= \int_{\R^d}\int_{\R^d} G^s(x,y)
  \frac{\phi(|x|)\phi(|y|)}{\Omega_d^2|x|^{d-1}|y|^{d-1}} \,dx \, dy \\
  &= \int_{\R} \int_{\R} \left(\frac{1}{\Omega_d^2} \int_{S^{d-1}}
    \int_{S^{d-1}}G^s(r_1\omega, r_2 \sigma) \,d\omega d\sigma \right) \phi(r_1)\phi(r_2)\,dr_1
  \,dr_2,
\end{align*} 
where $G^s$ is the covariance kernel of $h$, given in
Theorem~\ref{KernelComp}. Therefore, the covariance kernel of
$\overline{h}$ is 
\[
\overline{G}^s(r_1,r_2) \colonequals \frac{1}{\Omega_d^{2}} \int_{S^{d-1}}
    \int_{S^{d-1}}G^s(r_1\omega, r_2 \sigma) \,d\omega d\sigma. 
\]
Applying spherical symmetries to simplify this integral, we obtain 
\begin{align*}
\overline{G}^s&(r_1,r_2)= \\
&2\,C\int_0^\pi (\tfrac{1}{2}\log(r_1^2 + r_2^2 -
2r_1r_2 \cos\theta))^{\mathbf{1}_{\{H \in \Z_{+}\}}} \times \\ &\hspace{2cm
 } (r_1^2 + r_2^2 -
2r_1r_2 \cos\theta)^H (\sin \theta)^{d-2} \,d\theta,
\end{align*}
where $C$ is a constant described in Theorem~\ref{KernelComp} and $\Z_+$ is
the set of nonnegative integers. In the case $H\notin \Z^+$, we make a
substitution to obtain an integral in Euler form whose solution may be
expressed in terms of the Gauss hypergeometric function
${_2}F_1(a,b,c;z)$. In particular, we get
\begin{align*}
  \lefteqn{\overline{G}^s(r_1,r_2) = C\, 2^{d-1}\pi^{-1/2} \frac{\Gamma
      \left(\frac{d-1}{2}\right) \Gamma
      \left(\frac{d}{2}\right)}{\Gamma(d-1)}\times } \\
  &\hspace{1cm} (r_1+r_2)^{2H} \, _2{F}_1\left(\frac{d-1}{2},-H,d-1;\frac{4 r_1
      r_2}{(r_1+r_2)^2}\right).
\end{align*}

The hypergeometric function $_2F_1(a,b,c;z)$ satisfies
\[
|_2F_1(a,b,c;z)-{_2}F_1(a,b,c;1)|\asymp |z-1|^{c-a-b}
\] 
as $z\to 1$ whenever $c-a-b \in (0,1)$, because the indicial polynomial at
$z=1$ of the hypergeometric equation satisfied by $_2F_1$ has roots $0$ and
$c-a-b$ (see \cite{kristensson2010second} for details). When $c-a-b>1$,
$_2F_1(a,b,c;z)$ is differentiable at $z=1$. Since $\frac{4 r_1
  r_2}{(r_1+r_2)^2} = 1+O(|r_1-r_2|^2)$ as $r_1\to r_2$, it follows that
when $s>1/2$, we have $\overline{G}^s(r_1,r_2)-\overline{G}^s(r_1',r_2)\asymp
|r_1-r_1'|^{\min(1,\,2s-1)}$ as $r_1'$ approaches $r_1$. 

When $r_1$ and $r_2$ are far apart, $\overline{G}^s(r_1,r_2)$ is
approximately a constant times $(r_1+r_2)^{2H}$ since $_2F_1(a,b,c;z)$
approaches a constant as $z\to 0$. So we see that long-range covariances of
$h$ are determined by $H$, while local covariances are dictated by the
parameter $s-1/2$. In the following proposition, we show that in fact
$s-1/2$ also governs the almost-sure regularity of sample paths of
$\overline{h}$. We prove such a statement only for $s-\tfrac{1}{2} \in
(0,1)$, but we remark that in general the spherical average process is
differentiable $\lceil s-\tfrac{1}{2} \rceil - 1$ times, and those
derivatives are $\alpha$-H\"older continuous for all $\alpha$ less than the
fractional part of $s-\tfrac{1}{2}$.

\begin{proposition}
  When $s\in (1/2,3/2)$, there exists a version of the spherical average
  process $\overline{h}$ which is $\alpha$-H\"older continuous for all
  $\alpha < s-1/2$. 
\end{proposition}

\begin{proof}
  Since the spherical average covariance kernel $\overline{G}^s(r_1,r_2)$
  is finite for all $r_1$ and $r_2$ when $s\in (1/2,3/2)$, there exists a
  pointwise defined Gaussian process $\wt{h}$ on $(0,\infty)$ which agrees
  in law with $\overline{h}$ \cite[Theorem
  12.1.3]{dudley2002real}. Furthermore, the regularity of the covariance
  kernel implies that for $m=1$, we have 
  \begin{equation} \label{eq:kolmo}
  \E[|\wt{h}(r_1)-\wt{h}(r_2)|^{2m}] \leq C_m |r_1-r_2|^{m(2s-1)}, 
  \end{equation} 
  where $C_1$ is some constant. Since $\wt{h}(r_1)-\wt{h}(r_2)$ is
  Gaussian, \eqref{eq:kolmo} holds for all $m\in \N$, for some constants
  $C_m$. Applying the Kolmogorov-Chentsov continuity theorem with suitably
  large $m$, we conclude that $\wt{h}$, and therefore also $\overline{h}$,
  has a version which is almost surely $\alpha$-H\"older continuous for all
  $\alpha < s - 1/2$.
\end{proof}

\subsection{Background on spherical harmonic functions}

We write the Laplacian in spherical coordinates as
\begin{equation}\label{eq:sphecial Laplacian}
\Delta = r^{1-d}   \frac{ \partial  }{\partial r} r^{d-1} \frac{\partial
}{\partial r} + \frac{1}{r^2} \Delta_{S^{d-1} }, 
\end{equation}
where $\Delta_{S^{d-1}}$ is the Laplacian on the unit sphere
$S^{d-1}\subset \R^d$. A polynomial $\phi \in \mathbb{R} [x_1, x_2,\cdots,
x_d]$ is said to be harmonic if $\Delta \phi=0$. Suppose that $\phi$ is
harmonic and homogeneous of degree $k$. Let
$f=\left.\phi\right|_{S^{d-1}}$, and note that we have $\phi(ru)= f(u)r^k$
for all $u\in S^{d-1}$ and $r\geq 0$. Writing $\Delta \phi=0$, using
(\ref{eq:sphecial Laplacian}), and setting $r=1$ yields
\begin{equation}\label{eq: eigenfunction}
\Delta_{S^{d-1}} f = -k(k+d-2)f.
\end{equation}
In other words, $f$ is an eigenfunction of $\Delta_{S^{d-1}}$ with
eigenvalue $-k(k + d - 2)$.

We mention a few basic results about spherical harmonics that appear, for
example, in \cite[Chapter IV, \S 2]{stein1971introduction}. Assume $d\geq
2$, let $A_k$ be the set of homogeneous degree $k$ harmonic polynomials on
$\mathbb{R}^d$ and let $H_k$ be the space of functions on $S^{d-1}$
obtained by restricting functions in $A_k$\label{not:sphP}. An important property is that
the spaces $H_k$ are pairwise orthogonal (for the $L^2(S^{d-1})$ inner
product) and their union is dense in $L^2(S^{d-1})$. This means that we can
define, for each fixed $k$, an orthonormal basis $\{\phi_{ k,j }: 1\leq j
\leq \dim(H_k)\}$ of $H_k$ which is the restriction of the harmonic
polynomials $\{P_{ k,j }: 1 \leq j \leq \dim(H_k)\}\subset A_k$, so that
the collection of all $\phi_{k,j}$ is an orthonormal basis of
$L^2(S^{d-1})$ \label{not:sphH}.

We will need the following important theorem concerning the behaviour of
harmonic polynomials under the Fourier transform
\cite[pg. 72]{stein1970singular}. We say that a function $f:\R^d \to
\mathbb{C}$ is radial if $f(x)=f(y)$ whenever $|x|=|y|$. We 
occasionally abuse notation and write $f(r)$ where $f$ is radial and $r
\geq 0$, with the understanding that we mean $f((r,0,\ldots,0))$. 

\begin{theorem}\label{thm:Hecke}
  Let $P_k(x)$ be a homogeneous harmonic polynomial of degree $k$ in
  $\R^d$. Suppose that $f$ is radial and that $P_kf \in L^2(\R^d)$. Then the
  Fourier transform of $P_k f$ is of the form $P_kg$, where $g$ is a radial
  function. Moreover, the induced transform $\mathcal{F}_{d,k}
  (f)\colonequals g$ depends only on $d+2k$. More precisely, we have
  $\mathcal{F}_{d,k}=i^k\mathcal{F}_{d+2k,0}$.
\end{theorem}
\begin{remark}
  If $P_{k,j}f\in \dH^s(\R^d)$, then
  \[	
  \mathcal{F}\left[(-\Delta)^{s/2}(P_{k,j}f)\right](\xi) = |\xi|^{s}i^k\mathcal{F}_{d+2k,0}[f](\xi)P_{k,j}(\xi),
  \]
  Applying the Fourier transform on both sides (which is the inverse
  Fourier transform evaluated at $-x$) and using the
  theorem again, we obtain 
  \[
  (-\Delta)^{s/2}(P_{k,j}f) = [(-\Delta)^{s/2}_{\R^{d+2k}}f] P_{k,j},
  \]
  where $(-\Delta)^{s/2}_{\R^{d+2k}}f$ is the fractional Laplacian on
  $\R^{d+2k}$ acting on $f$ interpreted as a function on $\R^{d+2k}$ (that
  is, we define $f(x)$ for $x\in \R^{d+2k}$ to be $f(x')$ where $x'$ is any
  point in $\R^d$ satisfying $|x|_{\R^{d+2k}} = |x'|_{\R^d}$).
\end{remark}
\begin{remark}
  Let $P_{k,j}f_1$ and $P_{k',j'}f_2\in \dH^s(\R^d)$. Then
  \begin{multline}
    \label{eqn:innerproduct}
    \left\langle P_{k,j}f_1,P_{k',j'}f_2\right\rangle_{\dH^s(\R^d)} = \\
    \begin{cases}
      \int_0^\infty r^{2s+2k+d-1}g_1(r)\overline{g_2(r)}\,dr  \quad & (k,j)=(k',j'), \\
      0 \quad &(k,j)\neq(k',j').
    \end{cases}
  \end{multline}
  by orthonormality of $\phi_{k,j}$, where $g_i = \mathcal{F}_{d,k}[f_i] =
  i^k\mathcal{F}_{d+2k,0}[f_i]$ for $i\in \{1,2\}$. 
\end{remark}

We see that the right hand side of \eqref{eqn:innerproduct} (for
$(k,j)=(k',j')$) can be rewritten as $\left\langle f_1,
  f_2\right\rangle_{\dH^s(\R^{d+2k})}$ (since $\phi_{0,1} =
\Omega_d^{-1/2}$), where the radial functions $f_i$ are treated as
functions defined on $\R^{d+2k}$ (as described in the remark above).  We
thus have a unitary correspondence between elements $x\mapsto
f(|x|_{\R^{d+2k}})\in \dH^s(\R^{d+2k})$ and elements
$x\mapsto f(|x|_{\R^d})P_{k,j}(x)\in \dH^s(\R^d)$.

For $k\in \N$ and $1\leq j \leq \dim(H_k)$, we define the Hilbert space
$\dH_{k,j}^s(\R^d)$ to be the space of all functions of the form $P_{k,j}f$
where $x\mapsto f(|x|_{\R^{d+2k}})\in \dH^s(\R^{d+2k})$ is radial and
$P_{k,j}f \in \dH^s(\R^d)$ \label{not:SobSph}.  By \ref{eqn:innerproduct}, we see that
$\dH_{k,j}^s(\R^d)$ are orthogonal. In fact, they also span $\dH^s(\R^d)$:

\begin{lemma} \label{lemma: orthogonal}
$\dH^s_{k,j}(\R^d)$ are orthogonal subspaces spanning $\dH^s(\R^d)$.
\end{lemma}
\begin{proof}
  We only need to check the spanning condition.  Since $\s(\R^d)$ is dense
  in $\dH^s(\R^d)$, it suffices to show that all $g\in \s(\R^d)$ can be
  written as a linear combination of terms in $\dH^s_{k,j}(\R^d)$.  To do
  this, we use the stated fact that
  $\{\omega\mapsto\phi_{k,j}(\omega)\,:\,k\in \N,\,1\leq j \leq \dim H_k\}$
  a basis for $L^2(S^{d-1})$.  We compute for every sphere of radius $|x|$:
\[
\langle \omega\mapsto g(|x|\omega),\phi_{k,j}\rangle_{L^2(S^{d-1})} =
\int_{S^{d-1}}g(|x|\omega)\phi_{k,j}(\omega)\,d\omega \equalscolon
\rho_{k,j}(|x|),
\]
and see that $g(x) = \sum_{k,j} \rho_{k,j}(|x|)\phi_{k,j}(x/|x|) =
\sum_{k,j}|x|^{-k} \rho_{k,j}(|x|)P_{k,j}(x)$.  Define $g_{k,j}(x) =
|x|^{-k}\rho_{k,j}(|x|)P_{k,j}(x)$ and let $\chi_R(x)$ be the
characteristic function of an annulus of radii $1/R$ and $R$, where $R>1$.
It is clear that $g_{k,j}(x)\chi_R(x)$ is an element of $L^2(\R^d)$, since
$\|g_{k,j}(x)\chi_R(x)\|_{L^2(\R^d)}\leq \|g\|_{L^2(\R^d)}$ (by
orthogonality of $\phi_{k,j}(x/|x|)$), thus by Fatou's Lemma $g_{k,j}\in
L^2(\R^d)$.  Hence, it follows that the Fourier transform $\wh{g}_{k,j}$
exists and is in $L^2(\R^d)$.  Following the same reasoning as above with
$\xi\mapsto |\xi|^{2s}\widehat{g}_{k,j}(\xi)$, we have that
$x\mapsto \rho_{k,j}(|x|)P_{k,j}(x)\in \dH^s_{k,j}(\R^d)$ as required.
\end{proof}
\subsection{Spherical decomposition of the FGF}
We now study the spherical decomposition of the $\FGF_s(\R^d)$, 
which we denote by $h^d$. From the completeness and orthogonality
of $\dH^s_{k,j}(\R^d)$,
\begin{equation}\label{eq: decompostion}
h^d=\displaystyle{\sum_{k=0}^{\infty}}\displaystyle{\sum_{j=1} ^{\mathrm{dim} H_k}}h^d_{k,j},
\end{equation}
where the $h^d_{k,j}$ are independent standard Gaussians on the space of
$\dH^s_{k,j}(\R^d)$ (this follows from the same reasoning as in
Section~\ref{sec:projection}).

We note that $\dH^s_{k,j}(\R^d)$ is unitarily isomorphic to the Hilbert
space $\mathcal{R}^s_{d,k}$ consisting of radial functions $f_1,f_2\in\dH^s(\R^{d+2k})$ with inner product given in
\eqref{eqn:innerproduct}:
\[
\langle f_1,f_2 \rangle_{\mathcal{R}^s_{d,k}} = \int_0^\infty r^{2s+2k+d-1}g_1(r)\overline{g_2(r)}\,dr,
\label{not:sphHil}
\]
where $g_i = \mathcal{F}_{d+2k,0}[f_i]$ for $i\in \{1,2\}$. Thus, it follows
that we can construct a standard Gaussian on $\mathcal{R}^s_{d,k}$, which
we call $\wt{h}^d_{k,j}$ that corresponds to a standard Gaussian
$h^d_{k,j}$ on $\dH^s_{k,j}(\R^d)$.

The key observation is that the inner product on the Hilbert space
$\mathcal{R}^s_{d,k}$ above only depends on $d+2k$ (and $s$).  This means
that $\wt{h}^d_{k,j}$ has the same distribution as $\wt{h}^{d+2k}_{0,1}$
(equivalently, $\mathcal{R}^s_{d,k}$ is unitarily equivalent to
$\mathcal{R}^s_{d+2k,0}$).
Averaging both sides of \eqref{eq: decompostion} over $S^{d-1}_r
\colonequals rS^{d-1}$, we have
\[
	h^d_{0,1}=\frac{1}{r^{d-1}\Omega_d} \int_{S^{d-1}_r}  h^d(x)dx  =  \frac{1}{\Omega_d} \int_{S^{d}}  h^d(r\theta)d\theta.
\]
Note that we have used that $P_{0,1}(x) = \Omega_d^{-1/2}$, so that
$\dH^s_{0,1}(\R^d)$ is the set of radial functions $f\in\dH^s(\R^d)$. This
implies that $h^d_{0,1}$ averaged over a sphere is $h^d_{0,1}$.  By the
same observation, we have
\[
\wt{h}^d_{0,1} (r)= \frac{1}{\sqrt{\Omega_{d}}}\int_{S^{d-1}}  h^d(r\theta)d\theta,
\]
a constant multiple of the spherical average of $h^d$.
We collect these results in the following theorem.
\begin{theorem}\label{thm:spherical_decomposition}
  In the decompostion of $h^d=\FGF_s(\R^d )$ in \eqref{eq: decompostion},
  the coefficient processes $\wt{h}^d_{k,j}$ with respect to the normalized
  harmonic polynomials $\{P_{k,j}\}$ are independent processes with the
  same distribution as
  \[
	  r\mapsto \frac{1}{\sqrt{\Omega_{d+2k}}}\int_{S^{d+2k-1}} h^{d+2k}(r\theta)\,d\theta,
  \] 
  where $h^{d+2k}$ is an $\FGF_s(\R^{d+2k})$.
\end{theorem}
\begin{remark}
  We notice that since $\wt{h}^d_{k,j}$ is the average process of
  $\FGF_s(\R^{d+2k})$, it is defined modulo degree $\lfloor
  s-\frac{d}{2}-k\rfloor$ polynomials.  Since $\wt{h}_{d,k,j}$ is the
  coefficient of $P_{k,j}$, which is a polynomial of degree $k$, this is
  consistent with the fact that $h^d$ itself is defined up to polynomials
  of degree $\lfloor s-\frac{d}{2}\rfloor$.
\end{remark}

\begin{remark}
  From Theorem \ref{thm:spherical_decomposition}, one can analyze the
  average process in an arbitrary dimension by understanding the whole
  spherical decomposition of the $\FGF$ in dimensions 2 and 3 with the same
  index $s$. We remark that the distribution of the coefficient processes of
  $\FGF_{\frac{3}{2}}(\R^2) $ and $\FGF_{2}(\R^3) $ have been
  explicitly computed in \cite{mckean1963brownian}. Furthermore,
  \cite{mckean1963brownian} computes the coefficient processes for L\'evy
  Brownian motion ($\FGF$ with Hurst parameter $H=1/2$) in any dimension
  and gives the explicit covariance structure for $d\in \{2,3\}$.  In
  principle, we can also represent the covariance kernel for other values
  of $s$ with an integral involving a 2- or 3-dimensional harmonic polynomial
  and the covariance kernel of $\FGF$. If $d=2$, it involves
  trigonometric functions. If $d=3$, it will further involve associated
  Legendre polynomials; see Chapter 14 of \cite{olver2010nist}.
\end{remark}

When $s$ is a positive integer, we have $(-\Delta_{\R^{d+2k} })^s f
=(-L_{d,k})^s (f)$, where $L_{d,k} f = f''+(d+2k-1)r^{-1}f'$.  In this
case, the inner product of $\mathcal{R}^s_{d,k}$ is given by
$\int_0^{\infty} (-L_{d,k})^s (f)(r)g(r)\,dr$.  Since this inner product is
defined by a differential operator, $\wt{h}^d_{k,j}$ shares the same kind
of Markov property as the $\FGF_s$ when $s$ is an integer, which we
described at the end of Section \ref{sec:projection}: given the values of
$\wt{h}_{d,k,j}$ in the interval $[0,a]$, the conditional law of
$\wt{h}_{d,k,j}$ on the interval $(a,\infty)$ depends only on
$\left\{\wt{h}_{d,k,j} (a), \wt{h}'_{d,k,j} (a),\cdots, \wt{h}^{(s-1)
  }_{d,k,j} (a)\right\}$.

\section{The discrete fractional Gaussian field} \label{sec:dfgf}
\subsection{Fractional gradient} \label{sec:frac_grad}
\makeatletter{}
Recall that if $f:\R^d \to \R$ is differentiable, then the gradient $\nabla
f$ is a vector-valued function on $\R^d$ with the property that for all
$f,g\in \mathcal{S}(\R^d)$, 
\begin{equation} \label{eq:int_by_parts}
\int_{\R^d} \nabla f(x) \cdot \nabla g(x) \,dx = \int_{\R^d} (-\Delta f(x))g(x)\,dx.
\end{equation} 
For $0<s<1$, we will define the \textit{fractional gradient} $\nabla^s f$
so that an analogue of \eqref{eq:int_by_parts} holds with the fractional
Laplacian in place of the usual Laplacian. Rather than a vector-valued
function, however, we define $\nabla^s f$ to be a \textit{function}-valued
function on $\R^d$.  More precisely, if $f:\R^d \to \R$ is measurable, then
we define
\begin{equation} \label{eq:def_frac_grad}
\nabla^s f(x) = \left(y\mapsto \frac{f(x+y)-f(x)}{|y|^{\frac{d}{2} + s}}\right), 
\end{equation} 
where the domain of the function on the right-hand side is $\R^d\setminus
\{0\}$. We will establish the following analogue of the
integration-by-parts formula \eqref{eq:int_by_parts} for the fractional
gradient.
\begin{proposition}
  For all $d\geq 1$, $s\in (0,1)$, and $f,g\in \mathcal{S}(\R^d)$, 
  \begin{equation} \label{eq:frac_int_by_parts} \int_{\R^d} (\nabla^s
    f(x), \nabla^s g(x))_{L^2(\R^d)} \,dx = \int_{\R^d} ((-\Delta)^s f(x))
    g(x) \,dx
  \end{equation}
\end{proposition}

Note that we have replaced the gradient and Laplacian with their fractional
counterparts, and we replaced the dot product with an $L^2(\R^d)$ inner
product. 
\begin{proof}
  Since each side of \eqref{eq:frac_int_by_parts} is a bilinear form in
  $f$ and $g$, it suffices to show that the formula holds with $f=g$. We
  simplify the left-hand side of \eqref{eq:frac_int_by_parts} to obtain
  \begin{align*}
    \int_{\R^d}\int_{\R^d} &|(\nabla^s f(x))(y)|^2 \,dy \,dx \\ 
    &= \int_{\R^d}\int_{\R^d} \frac{|f(x+y)-f(x)|^2}{|y|^{d+2s}} \,dy \,dx \\ 
    &= \int_{\R^d}\int_{\R^d} \frac{f(x)^2 -2f(x)f(x+y)+f(x+y)^2}{|y|^{d+2s}}
    \,dy \,dx \\ 
    &=\int_{\R^d}\int_{\R^d} \frac{[2f(x)^2 -2f(x)f(x+y)]+[f(x+y)^2-f(x)]}{|y|^{d+2s}}
    \,dy \,dx.
  \end{align*}
  Changing variables for $x+y$ in the second square-bracketed expression
  shows that the left-hand side of \eqref{eq:frac_int_by_parts} is equal
  to 
  \[
  \int_{\R^d}\int_{\R^d} f(x) \frac{f(x+y) - 2f(x) + f(x-y)}{|y|^{d+2s}}
  \,dy \,dx, 
  \]
which equals the right-hand side of \eqref{eq:frac_int_by_parts} by
Proposition~\ref{prop:diff_quo}. 
\end{proof}

\makeatletter{}
\subsection{The discrete fractional Gaussian field}

In this section we define a sequence of discrete random distributions
converging in law to the fractional Gaussian field $\FGF_s(D)$, where
$s\in(0,1)$ and $D\subset \R^d$ is a sufficiently regular bounded
domain. We follow the strategy of \cite{caputo2000harmonic} and prove
convergence using a random walk representation of the field
covariances. This method was introduced by Dynkin
\cite{dynkin1980markov}. 

Suppose that $D\subset \R^d$ is a bounded domain and $s\in(0,1)$. For
$\delta > 0$, define $V^\delta\colonequals \delta \Z^d \cap D$. Recall that
the zero-boundary discrete Gaussian free field (DGFF) is defined to be the
mean-zero Gaussian field with density at $f\in \R^{V^\delta}$ proportional
to
\begin{equation} \label{eq:DGFF_density} 
\exp\left(-\frac{1}{2}\sum_{(x,y)\in (\delta \Z^d)\times(\delta
    \Z^d)}C_{d}\one_{|x-y|=\delta}|f(x)-f(y)|^2\delta^{d}\right),
\end{equation} 
where $C_d$ is a constant and where we interpret the expression in
parentheses as a quadratic form in the variables $\{f(x)\,:x \in \delta
\Z^d \cap D\}$ by substituting zero for each instance of the variable
$f(x)$ for all $x\notin D$. Observing that the sum in
\eqref{eq:DGFF_density} is a rescaled discretized version of the $L^2$ norm
of the gradient of $f$, we define the zero-boundary discrete fractional
Gaussian field $\DFGF_s(D)$ by replacing this expression with a rescaled
discretized $L^2$ norm of the fractional gradient of $f$. More precisely,
we let
\[
C_{d,s} = \left(\int_{\R^d} (1-\cos x_1)\,|x|^{-d-2s}\,dx\right)^{-1}, \quad
\text{where }x = (x_1,\ldots,x_d), 
\] 
and define
 $h^\delta \sim \DFGF^\delta_s(D)$ to be a Gaussian function
$h^\delta$ with density at $f\in \R^{V^\delta}$ proportional to
\[
\exp\left(-\frac{1}{2}\sum_{(x,y) \in (\delta \Z^d)^2, \: x\neq y}C_{d,s}\frac{|f(x)-
    f(y)|^2}{|x-y|^{d+2s}}\delta^{d}\right), 
\]
where we interpret the expression in parentheses as a quadratic form in the
variables $\{f(x)\,:x \in \delta \Z^d \cap D\}$ (as we did for the
DGFF). Observe that this quadratic form includes long-range interactions,
unlike the quadratic form for the GFF which includes only nearest-neighbor
interactions. The constant $C_{d,s}$ is chosen so that the discrete FGF
converges to the FGF with no further normalization--see
\eqref{eq:onepointconv} below to understand the role that this constant
plays in the calculation.

We interpret $h^\delta$ as a linear functional on $C_c^\infty(D)$ by
setting 
\begin{equation}\label{eq:discrete_def}(h^\delta,\phi)\colonequals \sum_{x
    \in
  V^\delta}h^\delta(x)\phi(x)\delta^d, \quad \text{for all }\phi\in
C_c^\infty(D).
\end{equation} 
To motivate \eqref{eq:discrete_def}, we note that the right-hand side is
approximately the same as the integral of an interpolation of $h^\delta$
against $\phi$. The following theorem is a rigorous formulation of the idea
that the DFGF converges to the FGF as $\delta \to 0$ when $D$ is
sufficiently regular. The idea of its proof is to compare a random walk
describing the covariance structure of the DFGF to the $2s$-stable L\'evy
process describing FGF covariances. Recall that $D$ is said to be $C^{1,1}$
if for every $z\in \partial D$, there exists $r>0$ such that $B(z,r)
\cap \partial D$ is the graph of a function whose first derivatives are
Lipschitz \cite{chen1998estimates}.

\begin{proposition}
  Let $D\subset \R^d$ be a bounded $C^{1,1}$ domain, and let $s\in (0,1)$. The
  discrete fractional Gaussian field $h^\delta\sim \DFGF_s(D)$ converges to
  the fractional Gaussian field $h\sim \FGF_s(D)$ in the sense that for
  any finite collection of test functions $\phi_1,\ldots,\phi_n\in
  C_c^\infty(D)$, we have
  \begin{equation} \label{eq:discrete_convergence} 
  ((h^\delta,\phi_1),\ldots,(h^\delta,\phi_n)) \to ((h,\phi_1),\ldots,(h,\phi_n))
  \end{equation} 
  in distribution as $\delta \to 0$. 
\end{proposition}

\begin{proof}
  Because both sides of \eqref{eq:discrete_convergence} are multivariate
  Gaussians and since $h^\delta$ and $h$ are linear, it suffices to show
  that $\E[(h^\delta,\phi)^2] \to \E[(h,\phi)^2]$ for all $\phi\in
  C_c^\infty(D)$. From \eqref{eq:discrete_def} we calculate
  \begin{equation} \label{eq:varsum} 
  \E\left[(h^\delta,\phi)^2\right] = \sum_{(x,y)\in V^\delta \times
    V^\delta} \E[h^\delta(x)h^\delta(y)]
  \phi(x)\phi(y)\,\delta^{2d}. 
  \end{equation} 
  Define an independent family of exponential clocks indexed by edges 
  \[
  \{(w,z)\,: w \in \delta \Z^d , z \in \delta \Z^d, \text{ and }w\neq z \}
  \] 
  such that the intensity of the clock corresponding to $(w,z)$ is
  $C_{d,s}\delta^{d}|w-z|^{-d-2s}$. Define a continuous-time process
  $(X^\delta_t)_{t\geq 0}$ which starts at $x \in V^\delta$ and moves
  from its current vertex $w$ to a new vertex $z\in \delta \Z^d$ whenever
  the clock associated with $(w,z)$ rings. 
                    Then 
  \begin{equation} \label{eq:walktime}
  \E[h^\delta(x)h^\delta(y)] = \E\left[\int_0^T \one_{\{X^\delta_t = y\}} \,dt\right],
  \end{equation} 
  where $T$ is the exit time from $D$ \cite[Section
  4.1]{sheffield2007gaussian}.

  We define a discrete-time version $(\wt{Y}^\delta_n)_{n\geq0}$ of the
  process $(X^\delta_t)_{t\geq 0}$ which tracks the sequence of vertices
  visited by $X^\delta$. That is, $\wt{Y}^\delta_n$ is the vertex at which
  $X^\delta_t$ is located after its $n$th jump. Let
  \[
  \gamma_{d,s} \colonequals C_{d,s}^{-1} \sum_{z\in \Z^d\setminus\{0\}} |z|^{-d-2s}. 
  \] 
  From $\wt{Y}^\delta$ we define the continuous-time process
  $(Y_t^\delta)_{t\geq 0}$ by $Y_t^\delta=\wt{Y}^\delta_{\lfloor
    \gamma_{d,s}^{-1}\delta^{-2s} t\rfloor}$.  Since the minimum of a
  collection of exponential random variables with intensities
  $(\lambda_i)_{i \in I}$ is exponential random variable with intensity
  $\sum_{i\in I} \lambda_i$, \eqref{eq:walktime} implies that
  \begin{align*}
    \E[h^\delta(x)h^\delta(y)] &= \E^x[\#\{n\,:\,\wt{Y}^\delta_n =
    y\}]\left(\sum_{z\in \delta \Z^d}
      C_{d,s}\delta^d|y-z|^{-d-2s}\right)^{-1} \\
    &= \E^x\left[\int_0^\infty \one_{\{Y^\delta_t = y\}}\,dt\right] \times \frac{\delta^{-2s} \gamma_{d,s}^{-1}
    C_{d,s}^{-1} \delta^{-d+d+2s}}{\sum_{z \in \Z^d \setminus
         \{0\}}|z|^{-d-2s}} \\ 
     &= \E^x\left[\int_0^\infty \one_{\{Y^\delta_t = y\}}\,dt\right],
  \end{align*}
  by our choice of $\gamma_{s,d}$ and $C_{s,d}$. If $Z$ is a Markov
  process, we denote by $p_t(x,y)\,dy = p^Z_t(x,y)\,dy$ the density of the
  law of $Z_t$ given $Z_0 = x$ (assuming that this law is absolutely
  continuous with respect to Lebesgue measure). Recall that the symmetric
  $2s$-stable process $(Y_{t\geq 0} )$ is the L\'evy process on $\R^d$
  whose transition kernel density $p_t$ has Fourier transform
  $\xi\mapsto\exp(-t|\xi|^{2s})$.

  By calculating the characteristic function of the step distribution of
  $\wt{Y}^\delta$, (see Remark 5.1 in \cite{caputo2000harmonic} for
  details), we see that 
  \begin{equation} \label{eq:onepointconv}
  Y_1^\delta \stackrel{\text{law}}{\to} Y_1 
  \end{equation} 
  By Theorem 2.7 in \cite{skorokhod1957limit}, this implies that
  $(Y^\delta_t)_{t\geq 0}$ converges in distribution to $(Y_t)_{t\geq 0}$
  with respect to the Skorokhod $J_1$ metric \cite{skorokhod1956limit},
  which is defined as follows. For an interval $I\subset [0,\infty)$, we
  denote by $\mathcal{D}(I,\R^d)$ the set of functions from $I$ to $\R^d$
  which are right-continuous with left limits, and for $t>0$ we denote by
  $\Lambda_t$ the set of increasing homeomorphisms from $[0,t]$ to itself.
  For $f,g\in \mathcal{D}([0,t],\R^d)$, we define the metric $d_{J_1(t)}$
  by
  \[
  d_{J_1(t)}(f,g) = \inf_{\lambda \in \Lambda_t} \max\left(\|f\circ \lambda -
    g\|_{\infty}, \: \|\lambda - \text{id}\|_{\infty}\right), 
  \]
  where $\text{id}(s) \colonequals s$. Then we define the metric
  \[
  d_{J_1}(f,g) = \int_0^\infty e^{-t} \min(1,d_{J_1(t)}(f,g))\,dt
  \]
  for $f,g\in \mathcal{D}([0,\infty),\R^d)$ \cite{melbourne2013weak}. A
  different definition that is equivalent and is also called the $J_1$
  metric is given in \cite{billingsley1999convergence}, where it is also
  proved that $d_{J_1}(f_n,f) \to 0$ if and only if $d_{J_1(t)}(\left. f_n
  \right|_{[0,t]},\left. f \right|_{[0,t]}) \to 0$ for every continuity
  point $t$ of $f$.

  Given a stochastic process $X$ started in $D$, denote by $T$ the exit
  time of the process from $D$. Denote by $\mu_{X,x}$ the occupation
  measure $\mu_{X,x}(A) \colonequals \E^x[\int_0^T \one_{X_t \in A}\,dt]$
  for all Borel sets $A\subset D$. We have $T<\infty$ almost surely, and
  $\mu_{X,x}$ is a finite measure---see the proof of
  Lemma~\ref{lem:occupation_convergence} where a stronger statement is
  proved. By Lemma~\ref{lem:occupation_convergence}, $\mu_{X_n,x} \to
  \mu_{X,x}$ weakly. Since weak convergence implies convergence of
  integrals against bounded continuous functions, we have
  \begin{equation} \label{eq:innerterm} 
  \sum_{y \in V^\delta} \E^x\left[\int_0^T \one_{\{Y_t^\delta = y\}}\,dt\right]
  \phi(y) \delta^d = \int_D \phi(y) \mu_{Y_t,x}(dy) + o(1), 
  \end{equation} 
  where the quantity denoted $o(1)$ is uniformly bounded as $x$ varies over
  the support of $\phi$ and tends to 0 as $\delta\to 0$ for each fixed
  $x$. Substituting \eqref{eq:innerterm} into \eqref{eq:varsum} and using
  the convergence of the Riemann integral (as well as dominated convergence
  to handle the $o(1)$ term), we obtain
  \begin{equation} \label{eq:discrete_greens}
    \E\left[(h^\delta,\phi)^2\right] \to \int_{D\times
      D}\phi(x)\mu_{Y,x}(dy)\,dx = \int_{D\times D} G(x,y) \,dx \,dy
  \end{equation} 
  as $\delta \to 0$, where $G$ is the density of the occupation measure
  (that is, the Green's function) of $Y$. This Green's function is in turn
  equal to $G_D^s(x,y)$ (see \eqref{eq:green_D}), the Green's function of
  the fractional Laplacian \cite{chen1998estimates}. Therefore, the
  right-hand side of \eqref{eq:discrete_greens} is equal to $\E[(h,f)^2]$,
  as desired.
\end{proof}
\begin{lemma} \label{lem:occupation_convergence}
  Let $(X_n)_{n\geq 1}$ be a sequence of processes in $\R^d$ converging in
  law with respect to the $J_1$ metric to a symmetric $\alpha$-stable
  process $X$. Let $D\subset \R^d$ be a $C^{1,1}$ domain. If $T$ is the
  hitting time of $\R^d \setminus D$, then the occupation measure of
  $X_n^T$ converges weakly to the occupation measure of $X^T$.
\end{lemma}

\begin{proof}
  For $n\geq 1$, denote by $\mu_n$ the occupation measure of $X_n^T$:
  \[
  \mu_n(A) \colonequals \E\left[\int_0^T \one_{X_n^T(t)\in A}\,dt\right],
  \]
  Similarly, define $\mu$ to be the occupation measure of $X^T$. 

  Recall the following definition of the L\'evy-Prohorov metric $\pi$ on
  the set of finite measures on $\R^d$. For $A\subset \R^d$ a Borel set,
  denote by $A^\eps$ the $\eps$-neighborhood of $A$, defined by 
  \[A^\eps \colonequals \{x\in \R^d\,:\,\exists\, y \in A \text{ such that
  }|x-y|<\eps\}.\] 
  Define for finite measures $\mu$ and $\nu$
  \[
  \pi(\mu,\nu) \colonequals \inf_{\eps>0}\{\mu(A)\leq \nu(A^\eps)+\eps \text{
    and }\nu(A)\leq \mu(A^\eps)+\eps\text{ for all }A\text{ Borel}\}.
  \]
  Recall that for probability measures, convergence with respect to $\pi$
  is equivalent to weak convergence
  \cite{billingsley1999convergence}. Since weak convergence of a sequence
  of finite measures $(\mu_n)_{n\geq 1}$ to a nonzero measure $\mu$ is
  equivalent to weak convergence of the normalized measures
  $\mu_n/\mu_n(\R^d) \to \mu/\mu(\R^d)$ along with convergence of the total
  mass (that is, $\mu_n(\R^d) \to \mu(\R^d)$), we see that convergence with
  respect to $\pi$ is equivalent to weak convergence for finite measures
  too. Therefore, it suffices to show that for all $\eps>0$ and $A\subset
  \R^d$, we have $\mu_n(A)\leq \mu(A^\eps) + \eps$ and $\mu(A)\leq
  \mu_n(A^\eps) + \eps$. Since $\mu_n(\R^d \setminus D) = \mu(\R^d
  \setminus D) = 0$, it suffices to consider $A\subset D$.  For $\eta > 0$,
  define $D_\eta = \{x\in D\,:\,\dist(x,\partial D) > \eta\}$. For $\eta >
  0$, define $B_\eta$ to be the event that $X$ stopped upon exiting
  $D_{\eta}$ is contained in $D^\eta$.  By integrating the upper bound in
  Theorem 1.5 in \cite{chen1998estimates}, we conclude that $B_\eta$ has
  probability tending to 0 as $\eta \to 0$. Furthermore, for each positive
  integer $n$, the event $E_n$ that $|X_{n+1} - X_n|$ is larger than the
  diameter of $D$ has probability bounded below. Since the events
  $(E_n)_{n\geq 1}$ are independent, it follows the amount of time $X$
  spends in $D$ has an exponential tail. Therefore, given $\eps>0$ we may
  choose $\eta \in (0,\eps/2)$ such that
  \[
  \E\left[\int_0^T \one_{\{X(t) \in A\}}\,dt \, \one_{B_\eta}\right] < \eps/2, 
  \]
  by the Cauchy-Schwarz inequality. 

  Since $(D[0,\infty),d_{J_1})$ is separable \cite[Theorem
  16.3]{billingsley1999convergence}, we may use Skorokhod's representation
  theorem \cite[Theorem 6.7]{billingsley1999convergence} to couple
  $(X_n)_{n \geq 1}$ and $X$ in such a way that $d_{J_1}(X_n,X) \to 0$ as
  $n\to\infty$. Choosing $n_0$ large enough that $d_{J_1}(X_n,X) < \eta/2$
  whenever $n \geq n_0$, we have for all $n\geq n_0$, 
  \begin{align*} \E\left[\int_0^T \one_{\{X_n(t) \in A\}}\,dt\right] &=
    \E\left[\int_0^T \one_{\{X_n(t) \in A\}}\,dt(\one_{B_\eta^c} +
      \one_{B_\eta})\right] \\
    &< \E\left[\left(\int_0^T
        \one_{\{X(t) \in A\}}\,dt\right)\one_{B^c_\eta}\right] +
    \frac{\eps}{2}.  
  \end{align*}
  By the definition of the $J_1$ metric, the first term is bounded above by
  \[\E\left[\int_0^T \one_{\{X(t) \in A^\eps\}}\,dt + \eps/2\right],\]
  which gives $\mu_n(A) \leq \mu(A^\eps) + \eps$. We conclude by applying
  the same argument with the roles of $X_n$ and $X$ reversed. 
\end{proof}

\section{Open questions}\label{sec:open}
\makeatletter{}In this section, we will ask some questions regarding the FGF. Section
\ref{sec:level line} presents several questions on level lines, and Section
\ref{sec:other questions} contains other FGF questions. 

\subsection{Questions on level sets}\label{sec:level line}
\begin{enumerate}
\item In dimension 2, $\FGF_{1+\eps}$ is a function for all $\eps>0$. Do
  the level sets of $\FGF_{1+\eps}$ converge to the level sets of the
  Gaussian free field, as defined in \cite{schramm2010contour}? One
    may interpret the mode of convergence to be in probability, with the
    coupling of Proposition~\ref{prop:FGF_coupling}, or in law.
  
  The Hausdorff dimension of the level sets of $\FGF_{1+\eps}(\R^2)$ is
  $2-\epsilon$ \cite{xiao2013recent}, while the Hausdorff dimension of
  $\SLE_4$ is $\frac{3}{2}$. Thus if the level sets of $\FGF_{1+\eps}$ do
  converge to the level sets of the Gaussian free field, then the Hausdorff
  dimension of these sets is not continuous in $\eps$.
\item Let $h$ be an instance of any $\FGF_s$ that is defined as a
  distribution, but not as a function.  One can mollify $h$ with a bump
  function supported on an $\eps$-ball in order to obtain a smooth
  function.  Under what circumstances do the level sets of these mollified
  functions converge to a continuum limit as $\epsilon \to 0$?
\item Instead of mollifying, one could instead try to project $h$ onto some
  subspace of piecewise-polynomial functions, like the projection of the
  two-dimensional GFF in \cite{schramm2010contour} onto the space of
  functions piecewise affine on the triangles of a triangular lattice with
  side length $\epsilon$. It was shown in \cite{schramm2010contour} that in
  the case of the two-dimensional GFF, the level sets of these
  approximations do converge to a continuum limit as $\epsilon\to 0$.  Can
  anything similar be obtained for any other dimension or any other value
  of $s$?
\item In $d$ dimensions, can one consider a $(d-1)$-tuple of independent
  FGFs (understood as a map from $\R^d$ to $\R^{d-1}$) and make sense of
  the scaling limit of the zero level set as a random curve?  Can one
  understand any discrete analogs of this problem? For the fractal
  properties of this curve when the corresponding $\FGF$ is a fractional
  Brownian motion, we refer to \cite{xiao2013recent}.
\end{enumerate}

\subsection{Other questions}\label{sec:other questions}
\begin{enumerate}
\item Are there any non-trivial local set explorations for FGF fields that
  are not defined as functions, as in the Gaussian free field case
  (\cite{millerimaginary1,millerimaginary2, millerimaginary3,
    millerimaginary4})?
\item If we restrict an LGF in $\R^3$ to a curved 2D surface, and
  conformally map that curved surface to a flat surface, can we pull
  back the restricted LGF to the flat surface and obtain a distribution
  whose law is locally absolutely continuous with respect to that of an
  ordinary LGF restricted to the flat surface?
    \end{enumerate}

\newpage
\appendix\section*{Notation}\label{sec:notation}
\makeatletter{}We fix the relation $H=s-d/2$ for the definitions of the following
spaces. \label{pg:notation} We refer the reader to the referenced page
numbers for the spaces' topologies.
{\footnotesize
\begin{center} \renewcommand{\arraystretch}{1.2}
\begin{tabular}{p{2.4cm}p{9.1cm}p{1.5cm}}
  \textbf{Space}&\textbf{Description}&\textbf{Page} \\ 
  $\s(\R^d)$ &  The Schwartz space of real-valued functions on $\R^d$ whose derivatives of all orders exist and decay faster than any polynomial at infinity. & \pageref{not:schw}\\
  $\s'(\R^d)$ &  The space of continuous linear functionals on $\s(\R^d)$. Elements of $\s'(\R^d)$ are called tempered distributions. & \pageref{not:tdist}\\
  $\s_k(\R^d)$ &  For $k\in \{-1,0,1,\ldots\}$, denotes the space of Schwartz functions
  $\phi$ such that $(\partial^\alpha \wh{\phi})(0) = 0$ for all
  multi-indices $\alpha$ such that $|\alpha|\leq k$. Equivalently,
  $\s_k(\R^d)$ is the space of Schwartz functions $\phi$ such that
  $\int_{\R^d} x^\alpha \phi(x)\,dx=0$ whenever $|\alpha|\leq k$. &
  \pageref{not:schwM}\\
  $\s_r(\R^d)$ & For $r\in \R$, denotes $\s_{\max(-1,\lfloor r\rfloor)} (\R^d)$ & \pageref{not:schwM} \\
  $\s'_k(\R^d)$ &  For $k\in \{-1,0,1,\ldots\}$, denotes the space of continuous linear functionals on $\s_k(\R^d)$. Equivalently, $\s'_k(\R^d)$ may defined to be the space $\s'(\R^d)$ of tempered distributions modulo polynomials of degree less than or equal to $k$. &  \pageref{not:tdistM}\\
    $\dot H^s(\R^d)$ & The subspace of $\s_{H}'(\R^d)$ consisting
  of functions whose Fourier transform $\xi\mapsto \wh{f}(\xi)$ is in
  $L^2(|\xi|^{2s}\,d\xi)$ & \pageref{not:Sob}  \\ 
  $\mathcal{U}_s(\R^d)$ &  The space of all functions $\phi\in
  C^\infty(\R^d)$ such that $x\mapsto(1+|x|^{d+2s})(\partial^\alpha
  f)(x)$ is bounded for all multi-indices $\alpha$ &  \pageref{not:yU?} \\
  $(-\Delta)^s\s_k(\R^d)$ & For $k\in\{-1,0,1,2,\ldots\}$ and
  $s>-\frac{1}{2}(d+k+1)$, this space is the range of the injective
  operator $(-\Delta)^s:\s_k(\R^d) \to \mathcal{U}_{s+(k+1)/2}$. 
  & \pageref{not:thatsyU}\\
  $T_s(\R^d)$ & The closure of $\s_\Hu(\R^d)$ in $\dot
  H^{-s}(\R^d)$. This space serves as a test function space for
  $\FGF_s(\R^d)$. & \pageref{not:tsrd}\\
  $C^\infty_c(D)$ &  The space of smooth functions supported on a compact
  subset of a domain $D\subset \R^d$. & \pageref {not:Ccinfty} \\
  $\dH^s_0(D)$ & The closure of $C^\infty_c(D)$ in $\dH^{s}(\R^d)$. & \pageref{not:hs0d} \\
  $T_s(D)$ & The closure of the space of restrictions to $D$ of Schwartz functions under the metric $d(\phi,\psi)=\|\phi-\psi\|_{\dot H^{-s}(D)}$. This space serves as a test function space for $\FGF_s(\R^d)$. & \pageref{not:TsD}\\
  $C^{k,\alpha}(D)$& For $k\in \{0,1,2,\ldots\}$, $\alpha\in
  (0,1)$, and $D\subset \R^d$, denotes the space of functions $f$ on $\R^d$
  such that $\partial^\beta f$ is $\alpha$-H\"older continuous for all
  multi-indices $\beta$ such that $|\beta|\leq k$. &\pageref{not:holder} 
\end{tabular}
\end{center}
}

\newpage

\bibliographystyle{hmralphaabbrv}

\medskip

\texttt{lodhia@math.mit.edu} \\
\texttt{sheffield@math.mit.edu} \\
\texttt{xinsun89@math.mit.edu} \\
\texttt{sswatson@math.mit.edu} \\
Department of Mathematics \\ 
Massachusetts Institute of Technology \\ 
Cambridge, MA, USA

\end{document}